\documentclass{amsart}

\usepackage{graphicx} 
\usepackage{amsfonts, amsmath,amssymb,amsthm}
\usepackage{thmtools}
\usepackage{xcolor}
\usepackage{enumitem}
\usepackage{hyperref}
\usepackage{cleveref}
\crefname{enumi}{}{}

\crefformat{enumi}{(\mathrm{#2#1#3})}
\Crefformat{enumi}{(\mathrm{#2#1#3})}
\crefname{enumi}{}{ }

\usepackage[margin=1 in]{geometry}

\newcommand{\rk}{\mathrm{rk}}
\newcommand{\crk}{\mathrm{crk}}

\newcommand{\KK}{\mathbb{K}}
\newcommand{\LL}{\mathbb{L}}

\newcommand{\RR}{\mathbb{R}}

\newcommand{\ZZ}{\mathbb{Z}}
\newcommand{\NN}{\mathbb{N}}

\newcommand{\cB}{\mathcal{B}}
\newcommand{\C}{{\mathcal{C}}}
\newcommand{\cR}{\mathcal{R}}

\newcommand{\cL}{\mathcal{L}}
\newcommand{\cT}{\mathcal{T}}
\newcommand{\cE}{\mathcal{E}}
\newcommand{\cU}{\mathcal{U}}

\newcommand{\circuits}{\mathrm{Circ}}

\newcommand{\Supp}{\mathrm{Supp}}
\newcommand{\gap}{\text{\scriptsize{gap}}}

\newcommand{\rsubad}{subadditive and cosubadditive}

\theoremstyle{plain}
\newtheorem{Theorem}{Theorem}[section]
\newtheorem{Lemma}[Theorem]{Lemma}
\newtheorem{Proposition}[Theorem]{Proposition}

\newtheorem{Corollary}[Theorem]{Corollary}

\newtheorem{SettingNotation}[Theorem]{Setting and Notation}
\newtheorem{Theorem 7.5}[Theorem]{\Cref{thm:GeneratorsOfSymbolicReesAlgebraOfMatroid}}
\theoremstyle{definition}
\newtheorem{Definition}[Theorem]{Definition}
\newtheorem{Example}[Theorem]{Example}
\newtheorem{Question}[Theorem]{Question}

\newtheorem{Remark}[Theorem]{Remark}

\makeatletter
\newcommand{\customgenericname}{} 
\newtheorem{innercustomgeneric}{\customgenericname}
\makeatother

\newcommand{\newcustomtheorem}[2]{%
  \newenvironment{#1}[1]{%
    \renewcommand{\customgenericname}{#2}%
    \renewcommand{\theinnercustomgeneric}{##1}%
    \begin{innercustomgeneric}%
  }{%
    \end{innercustomgeneric}%
  }%
}

\newcustomtheorem{customthm}{Theorem}

\def\urltilda{\kern -.15em\lower .7ex\hbox{\~{}}\kern .04em}
\newcustomtheorem{acknowledgement}{Acknowledgement}

\begin{document}

\title[Hamming weights and symbolic powers]{Generalized Hamming weights and symbolic powers of Stanley-Reisner ideals of matroids}
\author{Michael DiPasquale}

\address{Department of Mathematical Sciences \\
New Mexico State University\\
P.O. Box 30001 \\
Department 3MB \\
Las Cruces, NM 88003}
\email{midipasq@nmsu.edu}
\urladdr{\href{https://midipasq.github.io}{https://midipasq.github.io}}

\author{Louiza Fouli}
\address{Department of Mathematical Sciences \\
New Mexico State University\\
P.O. Box 30001 \\
Department 3MB \\
Las Cruces, NM 88003}
\email{lfouli@nmsu.edu}
\urladdr{\href{https://sites.google.com/view/louiza-fouli}{\tt https://sites.google.com/view/louiza-fouli}}

\author{Arvind Kumar}
\address{Department of Mathematical Sciences \\
New Mexico State University\\
P.O. Box 30001 \\
Department 3MB \\
Las Cruces, NM 88003}
\email{arvkumar@nmsu.edu}
\urladdr{\href{https://sites.google.com/view/arvkumar/home}{https://sites.google.com/view/arvkumar/home}}

\author{\c{S}tefan O. Toh\v{a}neanu}
\address{Department of Mathematics and Statistical Science\\ University of Idaho\\ 875 Perimeter Drive MS 1103\\ Moscow, ID 83844}
\email{tohaneanu@uidaho.edu}
\urladdr{\href{https://webpages.uidaho.edu/tohaneanu/}{https://webpages.uidaho.edu/tohaneanu/}}


\begin{abstract}
It is well-known that the first generalized Hamming weight of a linear code, more commonly called \textit{the minimum distance} of the linear code, corresponds to the initial degree of the Stanley-Reisner ideal of the matroid of the dual code.  Our starting point in this paper is a generalization of this fact -- namely, the $r$-th generalized Hamming weight of a matroid is the smallest degree of a squarefree monomial in the $r$-th symbolic power of the Stanley-Reisner ideal of the matroid (in the appropriate range for $r$). 
We show that the squarefree monomials in successive symbolic powers of the Stanley-Reisner ideal of a matroid suffice to describe all symbolic powers of the Stanley-Reisner ideal.  Hence, we provide explicit expressions for initial degree statistics of symbolic powers of the Stanley-Reisner ideal of a matroid in terms of its generalized Hamming weights.  A key aspect of our approach is a careful study of duality.  If the generalized Hamming weights of a matroid and its dual are both subadditive, we prove a simple expression for the initial degree of every symbolic power of the Stanley-Reisner ideal of the matroid, which closely mirrors that of a uniform matroid.  This has unexpectedly far-reaching consequences - we prove the generalized Hamming weights of a matroid and its dual are both subadditive for many interesting classes of matroids and codes, including sparse paving matroids, perfect matroid designs, matroids arising from Steiner systems, first-order affine and projective Reed-Muller codes, constant weight codes, Griesmer codes, and perfect codes.  As an application, we study the resurgence and asymptotic resurgence of the matroid configurations introduced by Geramita-Harbourne-Migliore-Nagel.  In particular, we explicitly compute the asymptotic resurgence of a matroid configuration of points arising from a perfect matroid design.
\end{abstract}

\maketitle

\section{Introduction}

Symbolic powers of a (radical) ideal are a geometric analogue of regular powers - by the Zariski-Nagata theorem~\cite{Zar49, Nag59}, the $r$-th symbolic power of the defining ideal of a variety in projective space consists of all polynomials vanishing to order $r$ along the variety.  Symbolic powers of ideals are extensively studied in commutative algebra and algebraic geometry - see the excellent survey~\cite{DDGHNB-2018}.  In the context of Stanley-Reisner ideals, a seminal result due independently to Minh-Trung~\cite{MT-2011} and Varbaro~\cite{Varbaro-2011} is that a simplicial complex is the independence complex of a matroid if and only if all the symbolic powers of its Stanley-Reisner ideal are Cohen-Macaulay.

In this paper, we study the initial degrees of symbolic powers of the Stanley-Reisner ideal of the independence complex of a matroid via a finite sequence of integers known as its \textit{generalized Hamming weights}.  Generalized Hamming weights have their origins in coding theory, where they extend the notion of the minimum distance of a linear code and characterize its performance in certain cryptographic applications, as shown in an influential paper of Wei~\cite{We}.  There is a straightforward way to extend the definition of generalized Hamming weights to matroids (e.g.~\cite{JV13}).  From the perspective of combinatorial commutative algebra, it is natural to ask how the generalized Hamming weights are encoded in the Stanley-Reisner ideal of the independence complex of the matroid (henceforth we refer to this ideal simply as the Stanley-Reisner ideal of the matroid).  Johnsen and Verdure~\cite{JV13} give a striking answer: the generalized Hamming weights of a matroid are the initial degrees of syzygy modules on its Stanley-Reisner ideal.

Many other natural connections have been made between coding theory and commutative algebra.  For instance, Gr\"{o}bner bases are used to error-correct, decode, and compute the minimum distance of any linear code (see \cite{Aug96, DePe, BuPe}).  From the work in \cite{DePe}, the minimum distance of a code (and indeed, the generalized Hamming weights~\cite{AGT-2017}) can also be interpreted in terms of the heights of ideals generated by fold products of linear forms dual to the columns of a generator matrix (see also \cite{To1}).  From this description, the authors in \cite{GaTo} made the connection between the minimum distance and the initial degree of a graded module, namely, the Fitting module.

Our paper begins with another `initial degree' characterization of generalized Hamming weights, different from that of Johnsen and Verdure~\cite{JV13}.  Namely, we prove that the $r$-th generalized Hamming weight of a matroid can be identified with the smallest degree of a squarefree monomial in the $r$-th symbolic power of the Stanley-Reisner ideal of the matroid (\Cref{cor:equivalentInterpretationsGeneralizedHammingWeights}).  In \Cref{thm:GeneratorsOfSymbolicReesAlgebraOfMatroid},  we prove that the (finitely many) squarefree monomials in successive symbolic powers of the Stanley-Reisner ideal of the matroid actually suffice to describe all symbolic powers (this collection of squarefree monomials generates the \textit{symbolic Rees algebra}).  It follows that the initial degree of any symbolic power of the Stanley-Reisner ideal of a matroid can be computed by solving an integer linear program involving the generalized Hamming weights (\Cref{thm:SR}).  In particular, the asymptotic growth rate of the initial degree sequence of symbolic powers of the Stanley-Reisner ideal of a matroid -- known as its \textit{Waldschmidt constant} -- can be expressed as the minimum of simple ratios involving the generalized Hamming weights (\Cref{thm:WaldschmidtGenHamWeights} or \Cref{thm:SR})~\eqref{thm:SRc}).

The Waldschmidt constant was originally introduced in the context of complex analysis~\cite{Waldschmidt-1977} and has taken a prominent role in the study of symbolic powers since it was re-discovered by Bocci and Harbourne~\cite{BH10}.  There are many cases where this important invariant has been computed -- see, for example, 
\cite{GHV13, BMM14, DHN14, BCGHJNSVV2016, KM25}.
Moreover, obtaining a lower bound for the Waldschmidt constant is the content of celebrated open conjectures such as Chudnovsky's conjecture \cite{Chu, Dum15, FMX, DTG, BGHN22} and Demailly's conjecture \cite{Demailly, BGHND22}.  Despite many known cases, the Waldschmidt constant is notoriously difficult to compute.  In the context of codes, the fundamental article of Vardy~\cite{Va} shows that the minimum distance -- hence the generalized Hamming weights -- is NP-hard to compute in general.  Since we express the Waldschmidt constant as a minimum of certain expressions involving the generalized Hamming weights, Vardy's result suggests that computing the Waldschmidt constant (and even the initial degree) of the Stanley-Reisner ideal of a matroid can be NP-hard.

Our motivation for studying symbolic powers of the Stanley-Reisner ideal of a matroid stems in large part from a landmark paper of Geramita, Harbourne, Migliore, and Nagel~\cite{GHMN17}, where they prove that appropriate specializations of Stanley-Reisner ideals of matroids (defining ideals of varieties called \textit{matroid configurations}) have symbolic powers which retain many of the good properties of symbolic powers of Stanley-Reisner ideals of matroids.  Star configurations, which are matroid configurations arising from the uniform matroid, are particularly highly studied.  Many invariants, including the Waldschmidt constant, resurgence, and asymptotic resurgence, have been computed for star configurations (see~\cite{GHM13, Man20, TX-2021}).

In this paper, once we make the connection between generalized Hamming weights and symbolic powers, we show that the initial degrees of symbolic powers of Stanley-Reisner ideals of many matroids behave like those of a uniform matroid of the same rank on the same ground set.  The key to our analysis is \textit{subadditivity} of generalized Hamming weights.  Suppose $M$ is a matroid of rank $k$ on a ground set of size $n$, and let ${\mathrm{U}}_{n,k}$ be the uniform matroid of rank $k$ on the same ground set.  If the generalized Hamming weights of both $M$ and its dual are subadditive, then we prove in \Cref{thm:SR} that there is a particularly simple expression for the initial degrees of symbolic powers of the Stanley-Reisner ideal of $M$ which closely mirrors that of the uniform matroid ${\mathrm U}_{n,k}$.  Moreover, the Waldschmidt constant of the Stanley-Reisner ideal of $M$ is just $\frac{n}{n-k}$, which is the same as the Waldschmidt constant for the Stanley-Reisner ideal of ${\mathrm U}_{n,k}$ (again, \Cref{thm:SR}).  In \Cref{sec:pavingMatroids} and \Cref{sec: gen-CT}, we prove that the generalized Hamming weights of many matroids and their duals are subadditive.  This includes important families such as sparse paving matroids, perfect matroid designs, matroids arising from Steiner systems, and matroids of first-order affine and projective Reed-Muller codes, constant weight codes, Griesmer codes, and perfect codes.  Consequently, matroid configurations arising from any of these matroids share some of the good properties of star configurations (\Cref{cor:resurgenceboundssparsepaving}).  In particular, matroid configurations arising from \textit{perfect matroid designs} have symbolic powers which, at least asymptotically, behave very much like symbolic powers of the defining ideal of a star configuration.

The structure of the paper is as follows. \Cref{sec:prelims} is dedicated to background and preliminaries on matroids, codes, and symbolic powers.  In \Cref{sec:SymbolicPowersGenHammingWeights}, we work directly from the definition of generalized Hamming weights to prove two initial results that motivate a deeper study of the symbolic powers of the Stanley-Reisner ideal of a matroid.  Our first result is \Cref{thm:minDist}, where we give several equivalent conditions for a squarefree monomial to be in a symbolic power of the Stanley-Reisner ideal of a matroid.  From this, we deduce that the $r$-th generalized Hamming weight of a matroid equals the smallest degree of a squarefree monomial in the $r$-th symbolic power of the Stanley-Reisner ideal of the matroid (\Cref{cor:equivalentInterpretationsGeneralizedHammingWeights}).  We then prove that the Waldschmidt constant of the Stanley-Reisner ideal of the matroid is the minimum of ratios involving the generalized Hamming weights (\Cref{thm:WaldschmidtGenHamWeights}).  Our proof involves the fractional chromatic number of a matroid, which can be expressed in terms of the generalized Hamming weights.

In \Cref{sec:subadd-finite-seq}, we take a purely sequence-theoretic perspective.  Our main object of study is what we call the \textit{initial degree sequence} of a finite sequence.  We show that the initial degree sequence of a finite sequence is subadditive and express its asymptotic growth rate using the original sequence (\Cref{prop:subadd-dfunction}).  We then define and study a duality for finite increasing sequences that we call \textit{Wei duality}.  Wei duality is an abstraction of the relationship between the generalized Hamming weights of a matroid and those of its dual, discovered in the context of linear codes by Wei~\cite{We}.  If the Wei dual of a finite sequence is subadditive, we say the sequence is \textit{cosubadditive}.  We prove that if a finite sequence is subadditive and cosubadditive, then the initial degree sequence of the finite sequence has a particularly simple form (\Cref{cor:sun&cosub}).  In the last part of \Cref{sec:subadd-finite-seq}, we prove that if the sequence of maximal cardinalities of flats of the matroid is convex, then the sequence of generalized Hamming weights of the matroid forms a \rsubad{} sequence (\Cref{prop:flat-sub&cosub} and \Cref{rmk:dc}).  This establishes a compelling and somewhat mysterious connection between convexity and subadditivity/cosubadditivity.

We capitalize on our study of the initial degree sequence in \Cref{sec:noetherian-rees-algebra}.  Quite generally, we show that the initial degree sequence of the sequence of degrees of generators of a Noetherian Rees algebra of a filtration is the sequence of initial degrees of the ideals constituting the filtration (\Cref{thm:ReesAlgebraFiltration}).  We then prove that the squarefree monomials in the symbolic powers of the Stanley-Reisner ideal of a matroid generate its symbolic Rees algebra (\Cref{thm:GeneratorsOfSymbolicReesAlgebraOfMatroid}).  Applying the machinery of \Cref{sec:subadd-finite-seq}, we then deduce our main results on the initial degrees of symbolic powers of Stanley-Reisner ideals of matroids, which are collected in \Cref{thm:SR}.

In \Cref{sec:pavingMatroids} and \Cref{sec: gen-CT}, we prove that large classes of matroids and matroids of linear codes have generalized Hamming weights that form subadditive and cosubadditive sequences.   
In \Cref{sec:MatroidConfigurations}, we compute or bound invariants of matroid configurations such as the Waldschmidt constant, regularity, and resurgence (\Cref{cor: matroid config invariants}, \Cref{cor: matroid config resurgence} and \Cref{cor:matroid config points resurgence}).  These bounds have a simpler expression for sparse paving matroids.  Moreover, we compute explicitly the asymptotic resurgence of a matroid configuration of points coming from a perfect matroid design (\Cref{cor:resurgenceboundssparsepaving}). Our work in \Cref{sec:MatroidConfigurations}, particularly on sparse paving matroids and perfect matroid designs, was inspired by \cite{BFGM21} and recovers many of the results therein.  We conclude the article with final remarks and a list of questions, \Cref{sec: conclusion}.

\textbf{Acknowledgement:} 
We learned in a personal communication with Mantero and Nguyen that they had independently proved a description of the generators for the symbolic Rees algebra of the Stanley-Reisner ideal of a matroid and some formulas for the Waldschmidt constant before us.  We deem it appropriate to cite their preprint~\cite{ManNgu} for the proof of the structure of the symbolic Rees algebra, which appears in our \Cref{thm:GeneratorsOfSymbolicReesAlgebraOfMatroid}.
Our proof of \Cref{thm:GeneratorsOfSymbolicReesAlgebraOfMatroid} can be found in \Cref{app}.  
We recommend readers consult~\cite{ManNgu} for additional interesting results on the symbolic powers of Stanley-Reisner ideals of matroids.

DiPasquale was partially supported by the NSF grant DMS--2344588. Kumar was partially supported by an AMS-Simons Travel grant.

 \section{Preliminaries on matroids, codes, and ideals}\label{sec:prelims}

\subsection{Matroids}
This subsection covers some basic facts about matroids.  We refer the readers to standard texts by Welsh~\cite{Welsh76} and Oxley~\cite{Oxley2011}  for more details.

Let $E$ be a finite set.  A matroid $M=(E,\cB(M))$ is a pair consisting of a ground set $E$ and a collection $\cB=\cB(M)$ of subsets of $E$, called bases, which satisfy the following axioms:
\begin{enumerate} 
\item[(i)] $\cB\neq\emptyset$,
\item[(ii)] (\textit{symmetric exchange axiom}) If $B,B'\in\cB$ and $x\in B\setminus B'$, then there is some $y\in B'\setminus B$ so that $\left(B\setminus \{x\}\right)\cup\{y\}\in\cB$ and $\left(B'\setminus \{y\}\right)\cup\{x\}\in\cB$.
\end{enumerate}

A subset $J\subseteq  E$ is called \textit{independent} in $M$ if $J\subseteq  B$ for some basis $B\in\cB$. A subset $D\subseteq  E$ is called  \textit{dependent} in $M$ if it is not independent, i.e., if $D\not\subseteq  B$ for any $B\in\cB$. The \textit{minimal} dependent sets of $M$ are called \textit{circuits}; we denote the circuits of $M$ by $\circuits(M)$.

The collection of sets $\{E\setminus B~:~ B\in\cB(M)\}$ satisfies the symmetric exchange axiom and thus forms the bases of another matroid, called the \textit{dual matroid} of $M$, and denoted by $M^*$.

The \textit{rank} function of a matroid $M$ is a function $\rk_M:2^E\to \ZZ_{\ge 0}$ taking each subset $A\subseteq  E$ to a non-negative integer $\rk_M(A)$ defined by
\begin{equation} \label{eq:rankfun}
    \rk_M(A)=\max\left\{|A\cap B|~:~ B\in \cB(M)\right\}.
\end{equation}
The rank function of a matroid satisfies the three properties:
\begin{enumerate}
\item[(i)] $\rk_M(\emptyset)=0$
\item[(ii)] For any subsets $A,B\subseteq  E$, $\rk_M(A\cup B)+\rk_M(A\cap B)\le \rk_M(A)+\rk_M(B)$
\item[(iii)] For any $e\in E$, $A\subseteq  E$, $\rk_M(A)\le \rk_M(A\cup\{e\})\le \rk_M(A)+1$.
\end{enumerate}

Conversely, the existence of a function with these properties gives another way to define a matroid.  The rank of the entire ground set, $\rk_M(E)$, is called the \textit{rank of the matroid}, which we denote by $\rk(M)$.  Evidently, $\rk(M)$ is the common cardinality of every basis of $M$.

Given a subset $A\subseteq  E$, the \textit{closure} of $A$ in  $M$, written $\mbox{cl}_M(A)$, is defined as the maximal set $F$ under inclusion, which contains $A$ and has the same rank as $A$ (the existence of such a maximal set is a consequence of property (ii) of rank functions). A  subset $A\subseteq  E$  is called a \textit{flat} of the matroid if $\mbox{cl}_M(A)=A$. We denote the flats of $M$ by $\cL(M)$ and the flats of rank $k$ by $\cL_k(M)$.  The set $\cL(M)$ has the structure of a geometric lattice, and the meet and join operations for two flats $F_1, F_2\subseteq  \cL(M)$ are, respectively, $F_1\cap F_2$ and $\mbox{cl}(F_1\cup F_2)$.

Matroids can be characterized by independent set axioms, circuit axioms,  flat axioms, or rank function axioms (and in many other ways), see \cite{Welsh76, Oxley2011}. 

A \textit{cocircuit} of $M$ is a circuit of $M^*$. 
 We write $\circuits(M^*)$ for the set of cocircuits of $M$.  The cocircuits of $M$ are precisely the complements of flats of $M$ of rank $\rk(M)-1$.  Indeed, a flat $F$ of $M$ with $\rk_M(F)=\rk(M)-1$ is, by definition, a maximal subset of $E$ which does not contain any basis of $M$.  So removing any element of $E\setminus F$ results in a set whose complement contains a basis of $M$; thus, removing any element of $E\setminus F$ yields an independent set of $M^*$. Therefore, $E\setminus F$ is a circuit of $M^*$.  The converse is similar.

 An element $e\in E$ is called a \textit{loop} if $\{e\}$ is a circuit of $M$, and is called a \textit{coloop} if $\{e\}$ is circuit of $M^*$. Two distinct elements $e_1,e_2 \in E$ are called \textit{parallel} in $M$ if $\{e_1,e_2\}$ is a circuit in $M$. If $M$ has no loop and no parallel elements, then we say that $M$ is a \textit{simple matroid}.

\subsection{Matroids and coding theory}

In this subsection, we consolidate key terminologies from coding theory relevant to linear codes and the associated matroids that emerge from them. We direct readers to \cite{van-Lint-1999} for a standard text on coding theory and to \cite{We} for the topic of generalized Hamming weights.

Let $\mathbb K$ be any field, and let $n$ be a positive integer. A {\it linear code} $\C$ is a linear subspace of $\mathbb K^n$. We say that $n$ is the {\it length} of $\C$, and $k:=\dim_{\mathbb K}(\C)$ is the {\it dimension} of $\C$. We assume $2\le k<n$ to avoid trivialities unless otherwise stated.

For any vector ${\bf w}\in\mathbb K^n$, the {\em weight} of ${\bf w}$, denoted $wt({\bf w})$, is the number of nonzero entries in ${\bf w}$. The {\em minimum distance} of $\C$ is the integer: $$d=d(\C):=\min\{wt({\bf c})~:~{\bf c}\in\C\setminus\{{\bf 0}\}\}.$$ The numbers $n$, $k$, and $d$ are called the {\em parameters} of $\C$, and $\C$ is called an {\em $[n,k,d]$-linear code} (or simply an $[n,k]$-linear code if we do not wish to specify the minimum distance).

A {\em generator matrix} of $\C$ is any $k\times n$ matrix whose rows form a basis for $\C$. If $G$ is such a matrix, then any element ${\bf c}$ of $\C$ (called {\em codeword}) is a linear combination of the rows of $G$, so ${\bf c}={\bf v}^TG$, for some ${\bf v}\in\mathbb K^k$. Using simple linear algebra, we obtain that the transpose of any such ${\bf c}$ is in the kernel of an $(n-k)\times n$ matrix, called a {\em parity-check matrix} of $\C$. 

A generator matrix $G$ for a code $\C$ is in \textit{standard form} if $G=\left[I_k|P\right]$ where $I_k$ is the $k\times k$ identity matrix and $P$ is a $k\times (n-k)$ matrix; in this case $H:=\left[-P^T|I_{n-k}\right]$ is a parity-check matrix of $\C$.  By the row reduction algorithm, and possibly after a permutation of the columns of $G$ (which produces a generator matrix of a code equivalent to $\C$;\footnote{By definition, two linear codes $\mathcal C_1$ and $\mathcal C_2$ are {\em (monomially) equivalent} if a generator matrix of one can be obtained from the other by right multiplication with a monomial matrix (i.e., the product of a diagonal matrix and a permutation matrix).} none-the-less, this ``new'' linear code has the same parameters as $\C$), any matrix $G$ can be brought to standard form.

For any  $[n,k,d]$- linear code $\C\subset\KK^n$  with generator matrix $G$, we define a matroid $M(\C)$ on the ground set $[n]$ whose independent sets are those subsets $J\subseteq  [n]$ so that the columns of $G$ indexed by $J$ are linearly independent. Since $\dim_{\KK} \mathcal{C}=k$, the rank of $M(\C)$ is $k$.
Note that we can define $M(\C)$ from any choice of a generating matrix of $\C$. Clearly, equivalent codes give rise to isomorphic matroids. 

Let $\C^{\perp}$ denote the orthogonal dual of $\C$ in $\mathbb K^n$, i.e.,
\[
\C^{\perp}=\{{\bf v}\in\mathbb K^n~:~ {\bf v}\cdot \bf c=0, \mbox{ for all } {\bf c}\in\C\},
\]
where $\cdot$ denotes the usual dot product in $\mathbb K^n$. Then, $\C^{\perp}$ is an $[n,n-k,d(\C^{\perp})]$-linear code, called the {\em dual code} of $\C$, and it has a generator matrix $H$, where $H$ is a parity-check matrix of $\C$.  

One can readily verify that the matroid corresponding to the dual of a code is, in fact, the dual of the matroid associated with the original code, i.e., $M(\C^\perp)=M(\C)^*$. We refer to $M(\C)$ as \textit{the matroid of $\C$} or \textit{the parity check matroid of $\C^\perp$}, and to $M(\C^\perp)$ as \textit{the matroid of $\C^\perp$} or \textit{the parity check matroid of $\C$}.

The notion of generalized Hamming weights first emerged within the realm of coding theory~\cite{HeKlMy}, specifically concerning the support structures of subcodes, see also \cite{HeKlMy, We}. Let $\mathcal C$ be an $[n,k,d]$-linear code.  We set  $[n]:=\{1,\ldots,n\}$. The {\em support of a codeword} ${\bf c}\in \mathcal{C}$ is ${\rm supp}({\bf c}):=\{i \in [n]~:~c_i\neq 0\}$.
 A {\em subcode} of $\mathcal{C}$ is a linear subspace $\mathcal{D}$ of $\mathcal{C}$ and its {\em support} is defined as 
 \[
 \Supp(\mathcal{D}):=\bigcup \limits_{{\bf c}\in \mathcal{D}} {\rm{supp}}({\bf c})=\{i\in [n]~:~ \exists {\bf c} \in \mathcal{D} \mbox{ with } c_i\neq 0\}.
 \]

For any $r\in [k]$, the {\em $r$-th generalized Hamming weight of $\mathcal C$} is the positive number 
\[
d_r(\mathcal C):=\min\{|\Supp(\mathcal D)|: \mathcal D\subseteq \mathcal C,\,\dim_{\KK}\mathcal D=r\}.
\]

Note that $d_1(\mathcal C)$ equals the minimum distance $d$ of $\mathcal C$; this is because all subcodes of $\C$ of dimension 1 are the linear spans (i.e., scalar multiples) of the nonzero codewords of $\C$. On the other hand, $d_k(\C)=n-t$, where $t$ is the number of zero columns of a generator matrix  $G$  of $\C$, since the only subspace of $\C$ of dimension $k$ is $\C$ itself.

Wei established the following rank formula for the generalized Hamming weights of linear codes: 

\begin{Theorem}$($\cite[Theorem 2]{We}$)$\label{thm:GenHammingNullityFormula}  Let $\mathcal C$ be an $[n,k,d]$-linear code with a parity check matrix $H$. For any  $j\in [n]$, let $H_j$ denote the $j$-th column of $H$.   Then for any $r\in [k]$,
\begin{align*}
d_r(\mathcal C)=\min_{J\subseteq[n]}\{|J| ~:~  |J|-\dim_{\mathbb K}{\rm Span}_{\mathbb K}\{H_j\,:\, j\in J\}\ge r\}.
\end{align*}
\end{Theorem}

The notion of generalized Hamming weights can be generalized for arbitrary matroids; we give a formulation that appears in \cite{JV13}.

\begin{Definition}$($\cite[Definition 2.1]{JV13}$)$\label{def:HammingWeightsMatroids}
Let $M$ be a matroid of rank $k$ on the ground set $E$ of size $n$. Then  for any integer $r\in[n-k]$ {\em the $r$-th generalized Hamming weight of $M$},  denoted $d_r(M)$, is defined as 
\[
d_r(M)=\min\{|U|~:~ U\subseteq E \text{ and } |U|-\rk_{M}(U)= r\}.
\] 
\end{Definition}

We will see additional matroidal interpretations of generalized Hamming weights in \Cref{sec:SymbolicPowersGenHammingWeights} and \Cref{ss:IDS-GHWs}.  In the following remark, we connect the generalized Hamming weight of a code to that of the generalized Hamming weight of the matroid associated with its dual code.

\begin{Remark}\label{rem:HammingWeightsOfMatroidsVsLinearCodes} Let $\C$ be a $[n,k]$-linear code and, let $\C^{\perp}$ be its dual code. For $J\subseteq  [n]$, using the notations from \Cref{thm:GenHammingNullityFormula},   $\dim_{\KK}({\rm Span}_{\KK}\{H_j~:~ j\in J\})$ can be interpreted in the language of matroids as $\rk_{M(\C^\perp)}(J)$. By applying \Cref{thm:GenHammingNullityFormula} in conjunction with \Cref{def:HammingWeightsMatroids},  we arrive  at the following formulation: \begin{equation*}\label{eq:GenHammingMatroidFormulation}
    d_r(\mathcal C)=\min_{J\subseteq[n]}\{|J| ~:~  |J|- \rk_{M(\C^\perp)}(J) \ge r\}=d_r\left(M(\C^{\perp})\right)=d_r(M(\C)^*).
\end{equation*}
Thus the \(r\)-th generalized Hamming weight of the code \(\C\) is equal to the \(r\)-th generalized Hamming weight associated with the \textit{parity check matroid} of \(\C\), rather than the matroid of $\C$. \end{Remark}

We conclude this subsection by highlighting the key properties of generalized Hamming weights.  Most of these properties are due to Wei for linear codes \cite{We}.  The extension of these to matroids can be found in~\cite{Wei-type-Matroid} and~\cite{JV13}.

\begin{Lemma}\label{lem:genMindistMatroid}
Let $M$ be a matroid of rank $k$ on a ground set $E$ of size $n$.  Then
\begin{enumerate}
\item \label{lem:genMindistMatroida} $1\le d_1(M)<d_2(M)<\cdots < d_{n-k}(M) \le n$,
\item \label{lem:genMindistMatroidb} $d_r(M)\le k+r$,
\item \label{lem:genMindistMatroidc} If $d_{r_0}(M)=k+r_0$ for some $r_0\in[n-k]$, then $d_r(M)=k+r$ for $r_0\le r\le n-k$.
\item \label{lem:genMindistMatroidd} $d_{n-k}(M)=n-\ell$, where $\ell$ is the number of coloops in $M$ (or loops in $M^*$).  In particular, if $M$ has no coloops, then $d_{n-k}(M)=n$.
\item \label{lem:genMindistMatroide} $\{d_r(M)~:~ r\in[n-k]\}=[n]\setminus\{n+1-d_s(M^*)~:~ s\in [k]\}$.
\end{enumerate}
\end{Lemma}

\subsection{Simplicial complexes and ideals associated to a matroid}
Given a matroid $M=(E,\cB(M))$, we build a simplicial complex $\Delta(M)$ -- called {\em the independence complex of $M$} -- whose simplices are the independent sets of $M$. 
Also, we associate $S=\LL[x_i~:~i\in E]$ a polynomial ring over a field $\LL$ with variables indexed by $E$. For  $U\subseteq  E$, we put $x^U:=\prod_{u\in U} x_u$ and $P_U:=\langle x_u~:~ u\in U\rangle$.

The non-faces of $\Delta(M)$ are the dependent sets of $M$; therefore, the minimal non-faces of $\Delta(M)$ correspond to circuits of $M$.  The Stanley-Reisner ideal of $\Delta(M)$ is therefore
\[
I_{\Delta(M)}=\langle x^C~:~ C\in\circuits(M)\rangle=\bigcap_{B\in \cB(M^*)} P_{B}.
\]
To verify the above is a primary decomposition, it suffices to observe that the bases of $M^*$ are precisely the minimal subsets of $E$, which meet every circuit of $M$. 
 Similarly,
\[
I_{\Delta(M^*)}=\langle x^C~:~ C\in\circuits(M^*)\rangle=\bigcap_{B\in \cB(M)} P_{B}.
\]

\subsection{Symbolic powers and the Waldschmidt constant}
Let $I$ be a homogeneous ideal in $S$. Recall that the $s$-th \textit{symbolic power} of an ideal $I$ is 
\[
I^{(s)}:=\bigcap_{P\in\mbox{Ass}(S/I)} (P^sR_P)\cap R,
\]
which for a squarefree monomial ideal simplifies (see~\cite{CEHH17}) to 
\[
I^{(s)}=\bigcap_{P\in\mbox{Min}(S/I)}P^s.
\]
In particular, if $M=(E,\cB(M))$ is a matroid, then 
\[I^{(s)}_{\Delta(M)}=\bigcap_{B\in \cB(M^*)} P_B^s \text{ and } I^{(s)}_{\Delta(M^*)}=\bigcap_{B\in \cB(M)} P_B^s.\]
For a vector ${\bf a}=(a_1,\ldots,a_{|E|}) \in \NN^{|E|}$, we set $x^{{\bf a}}:=\prod_{e \in E} x_e^{a_e}$. It follows that  \begin{align}\label{eq:PDsymbolicpowersSR}
    x^{\textbf{a}} \in I^{(s)}_{\Delta(M)} & \text{ if and only if } \sum\limits_{e\in B} a_e\ge s \text{ for all } B\in \cB(M^*) \\ & \text{ if and only if }
    \sum\limits_{e\notin B} a_e\ge s \text{ for all } B\in\cB(M). \nonumber 
\end{align}

We conclude this section by recalling the definitions of initial degrees and the Waldschmidt constant of an ideal.

\begin{Definition}
Let $S$ be a positively graded ring and let $M=\oplus_{i\ge 0}M_i$  be a finitely generated graded module over $S$.  The {\em initial degree} of $M$, denoted $\alpha(M)$, is the smallest $i$ for which $M_i\neq 0$; in other words, it is the smallest degree of a generator of $M$. 

Let $I$ be a homogeneous ideal in $S$.  The {\it Waldschmidt constant of } $I$, 
denoted by $\widehat{\alpha}(I)$, is defined to be $\widehat{\alpha}(I)=\lim \limits_{s\rightarrow \infty} \frac{\alpha (I^{(s)})}{s}=\inf \limits_{s\ge 1} \frac{\alpha (I^{(s)})}{s}$. This limit exists because of the subadditivity of $\alpha$ on symbolic powers -- see~\cite{BH10}.
\end{Definition}

\begin{Remark}\label{rmk:loopless}
Our interests in this article are centered around the question of membership in a symbolic power and determining $\alpha(I_{\Delta(M)})$ and $\widehat{\alpha}(I_{\Delta(M)})$.  We will always assume $0<\rk(M)<n$ to avoid trivialities.

Notice further that if $M$ is a matroid with a loop $e\in E$, then $x_e\in I_{\Delta(M)}$ and hence $x_e^s \in I_{\Delta(M)}^{(s)}$. In this case, $\alpha(I_{\Delta(M)}^{(s)}) =s$ for all $s \ge 1$ and $\widehat{\alpha}(I_{\Delta(M)})=1$.  Similarly, if $M^*$ has a loop $e\in E$, then $e$ is not contained in any circuit of $M$, and thus $I_{\Delta(M)}$ does not have any generator divisible by $x_e$.  It follows that nothing is lost (for purposes of analyzing membership in symbolic powers) by considering instead the Stanley-Reisner ideal $I_{\Delta(M/\{e\})}$, where $M/\{e\}$ is the \textit{contraction} of $M$ along $\{e\}$.  Precisely, $M/\{e\}$ is the matroid on the ground set $E\setminus\{e\}$ whose bases are obtained from the bases of $M$ that contain $e$ by removing $e$.  If $e$ is a loop of $M^*$, then the circuits of $M$ coincide with the circuits of $M/\{e\}$. 
 With these two observations, we reduce to the case that both $M$ and $M^*$ do not have loops.
\end{Remark}

\section{Generalized Hamming weights of matroids and symbolic powers}\label{sec:SymbolicPowersGenHammingWeights}
 
In this section, we introduce the main themes of the paper by highlighting two results (\Cref{thm:minDist}, \Cref{thm:WaldschmidtGenHamWeights}) that bridge generalized Hamming weights with symbolic powers of Stanley-Reisner ideals. These results serve as a launching point for the deeper structural connections developed in the following sections.

Let $M$ be a matroid of rank $k$ on a ground set $E=[n]$. Let $ S:=\LL[x_1,\ldots,x_n]$ be a polynomial ring over a field $\LL$ and $I_{\Delta(M)}$ be the Stanley-Reisner ideal of $M$.  Our first result gives a number of equivalent interpretations for membership of a squarefree monomial in a symbolic power of $I_{\Delta(M)}$.  Recall that if $U\subset E$, the quantity $|U|-\rk_{M}(U)$ is called the \textit{nullity} of $U$ (in $M$) and is denoted $n_M(U)$.

\begin{Theorem}\label{thm:minDist} Let $M$ be a matroid of rank $k$ on the ground set $E=[n]$, $r\in [n-k]$, and $J \subseteq  E$.  The following are equivalent: 
\begin{enumerate}
\item $x^J\in I^{(r)}_{\Delta(M)}.$
\item $n_M(J)=|J|-\rk_M(J)\ge r.$
\item $\rk_{M^*}(E\setminus J)\le n-k-r.$
\end{enumerate}
\end{Theorem}

\begin{proof}
Let $J\subseteq E$ and let $r\in[n-k]$. 
First, we prove the equivalence of $(\rm b)$ and $(\rm c)$.  We use the following relationship between rank functions of $M$ and $M^*$ (see \cite[Proposition 2.1.9]{Oxley2011}):
\begin{align*} \label{rank-rel}
|J| -\rk_{M}(J) = \rk(M^*) -\rk_{M^*}(E \setminus J)= n-k - \rk_{M^*}(E \setminus J).
\end{align*} 
This gives us that $|J|-\rk_{M}(J) \ge r $ if and only if $n-k - \rk_{M^*}(E \setminus J) \ge r$ if and only if $\rk_{M^*}(E \setminus J) \le n-k - r$, as desired.

Next, we prove $(\rm c)$ and $(\rm a)$ are equivalent.  Observe that $\rk_{M^*}(E\setminus J)=\max\{|B\cap (E\setminus J)|: B\in\cB(M^*)\}$.  So $(\rm c)$ is equivalent to
\[
|B\cap (E\setminus J)|\le n-k-r \mbox{ for all } B\in\cB(M^*),
\]
which is equivalent to
\[
|B\cap J|+|B\cap (E\setminus J)|\le |B\cap J|+ n-k-r \mbox{ for all } B\in\cB(M^*).
\]
The left hand side of this inequality is simply $|B|=n-k$ (since $B\in\cB(M^*)$ and $\rk(M^*)=n-k$).  Re-arranging, we get the equivalent condition 
\[
r\le |B\cap J| \mbox{ for all } B\in\cB(M^*).
\]
By \Cref{eq:PDsymbolicpowersSR}, 
this last condition on $J$ is equivalent to $x^J\in I^{(r)}_{\Delta(M)}$, as desired.
\end{proof}

As a corollary, we show the connection between symbolic powers of the Stanley-Reisner ideal of a matroid and its generalized Hamming weights.

\begin{Corollary}\label{cor:equivalentInterpretationsGeneralizedHammingWeights}
Let $M$ be a matroid of rank $k$ on a ground set $E$ of size $n$, and $r\in [n-k]$.  Then
\begin{enumerate}
\item $d_r(M)=\min\{|J|: J\subseteq E \mbox{ and }x^J\in I^{(r)}_{\Delta(M)}\}$ (i.e., the smallest degree of a squarefree monomial in $I^{(r)}_{\Delta(M)})$.
\item $d_r(M)=\min\limits\{|U|:U\subseteq E \mbox{ and } n_M(U)\ge r\}.$
\item $d_r(M)=n-\max\{|F|:F\in \cL_{n-k-r}(M^*)\}.$
\end{enumerate}
\end{Corollary}
\begin{proof}
To show that the expression in $({\rm b})$ is equal to $d_r(M)$, observe that for any $e\in U$, $n_M(U)-1\le n_M(U\setminus\{e\})\le n_M(U)$.  Thus $\min\{|U|:U\subseteq E \mbox{ and } n_M(U)\ge r\}=\min\{|U|:U\subseteq E \mbox{ and } n_M(U)= r\}=d_r(M)$ by \Cref{def:HammingWeightsMatroids}.  The expression in $(\rm a)$ is then equal to $d_r(M)$ by \Cref{thm:minDist}.  To show the expression in $(\rm c)$ is equal to $d_r(M)$, observe that
\begin{align*}
d_r(M)& =  \min\{|J|~:~J\subset E,\rk_{M^*}(E\setminus J)\le n-k-r\}\\
& =  n-\max\{|F|~:~ F\subset E, \rk_{M^*}(F)\le n-k-r\}\\
& =  n-\max\{|F|~:~  F\in \cL_{n-k-r}(M^*)\},
\end{align*}
where the first equality follows from \Cref{thm:minDist}.
\end{proof}

As an immediate consequence of \Cref{cor:equivalentInterpretationsGeneralizedHammingWeights}, we obtain a lower bound for the generalized Hamming weights of matroids.
\begin{Corollary}\label{gHW-lb}
    Let $M$ be a matroid of rank $k$ on the ground set $E=[n]$, and let $ r \in [n-k]$. Then,  $$d_r(M) \ge \alpha(I_{\Delta(M)}^{(r)}).$$
\end{Corollary}

The following result demonstrates that the generalized Hamming weights of a matroid suffice to ascertain the Waldschmidt constant of its Stanley-Reisner ideal.  Observe that the following result and its proof use only the definition of generalized Hamming weights (\Cref{def:HammingWeightsMatroids}) and thus are independent of~\Cref{thm:minDist}.

\begin{Theorem}\label{thm:WaldschmidtGenHamWeights}
Let $M$ be a matroid on the ground set $E=[n]$ of rank $k<n$. Then
\[
\widehat{\alpha}(I_{\Delta(M)})=\min_{r\in [n-k]}\left\lbrace \frac{d_r(M)}{r}\right\rbrace.
\]
\end{Theorem}

\begin{proof}
We know that  $d_1(M)=1$ if and only if  $M$ has a circuit of size one. 
Equivalently, $d_1(M)=1$ if and only if $M$ has a loop. But in that case, $\widehat{\alpha}(I_{\Delta(M)})=1$, see \Cref{rmk:loopless}. Moreover, $1\le \min \limits_{r\in [n-k]}\left\{\frac{d_r(M)}{r}\right\}\le \frac{d_1(M)}{1}=1$ and the result follows in this case.

Next, we assume that $d_1(M)\ge 2$, i.e., $M$ has no loop. It follows from a result of Edmonds~\cite{Edmonds65} that the \textit{fractional chromatic number} of a loopless matroid $M$ is
\begin{equation}\label{eq:frchromaticmatroid}
\chi_f(M)=\max\limits_{\emptyset\neq J \subseteq  E}\frac{|J|}{\rk_{M}(J)}=\max\limits_{J\in \mathcal{D}(M)}\frac{|J|}{\rk_{M}(J)},
\end{equation}
where $\mathcal{D}(M)$ is the collection of dependent sets of $M$ (our assumption that $\rk(M)=k<n$ assures that $\mathcal{D}(M)\neq \emptyset$).  A proof of this identity can be found in the thesis of Laso\'{n}~\cite{lason2017coloring}.

The fractional chromatic number of $M$ is the same as the fractional chromatic number of the hypergraph, or clutter, whose edges correspond to the circuits of $M$.  According to~\cite[Theorem~4.6]{BCGHJNSVV2016}
\[
\widehat{\alpha}(I(H))=\frac{\chi_f(H)}{\chi_f(H)-1},
\]
where $I(H)$ is the squarefree monomial ideal whose generators correspond to the edges of a clutter $H$.
Observe that if $H$ is the clutter of circuits of $M$, then $I(H)=I_{\Delta(M)}$.  Thus
\begin{equation}\label{eq:WaldschmidtMatroidfr}
\widehat{\alpha}(I_{\Delta(M)})=\frac{\chi_f(M)}{\chi_f(M)-1}.
\end{equation}
Since $0<\rk(M) <n$,  $M$ must have at least one dependent set with positive rank. Therefore, $\chi_f(M)>1$, ensuring that the denominator does not vanish.  Combining \Cref{eq:frchromaticmatroid} with \Cref{eq:WaldschmidtMatroidfr}   yields
\begin{align*}
    \widehat{\alpha}(I_{\Delta(M)})& =1+\frac{1}{\chi_f(M)-1}=\min_{J\in \mathcal{D}(M)}\left\lbrace 1+\frac{1}{\frac{|J|}{\rk_{M}(J)}-1} \right\rbrace \\& =\min_{J\in \mathcal{D}(M)}\left\lbrace \frac{|J|}{|J|-\rk_{M}(J)} \right\rbrace.
\end{align*}
Since $J\in \mathcal{D}(M)$, $|J|-\rk_M(J)$ ranges between $1$ and $n-k$.  So we can write
\begin{align*}
\widehat{\alpha}(I_{\Delta(M)})=\min_{J\in \mathcal{D}(M)}\left\lbrace \frac{|J|}{|J|-\rk_{M}(J)} \right\rbrace = & \min_{r\in [n-k]} \left\lbrace \frac{|J|}{r} ~:~ |J|-\rk_M(J)=r\right\rbrace
=  \min_{r\in [n-k]}\left\lbrace \frac{d_r(M)}{r}~\right\rbrace,
\end{align*}
as desired.
\end{proof}

\begin{Example}\label{Ex:MDS-uniform} Let $k,n\in\mathbb{N}$ with $k<n$, and let $M={\mathrm{U}}_{n,k}$ be the \emph{uniform matroid} of rank $k$ on the ground set $[n]$.  The bases of $M$ are all the subsets of $[n]$ of size $k$, and the circuits of $M$ are all subsets of $[n]$ of size $k+1$. The uniform matroid $\mathrm{U}_{n,k}$ is the matroid of a \textit{maximum distance separable} (MDS) $[n,k]$-code, which we discuss in more detail in \Cref{sec:pavingMatroids}.

From the description of the circuits above, the Stanley-Reisner ideal of $\Delta(M)$ is
\[
I_{\Delta(M)}=\langle \{x^C~: ~C \subseteq [n] \text{ and } |C|=k+1\}\rangle.
\]
Recall that $d_1(M)$ is the minimum size of a circuit of $M$, so $d_1(M)=\alpha(I_{\Delta(M)})=k+1$.  By \Cref{lem:genMindistMatroid}~\eqref{lem:genMindistMatroidc}, $d_s(M)=k+s$ for all $ s\in [n-k]$.  It follows from \Cref{gHW-lb} that $\alpha(I_{\Delta(M)}^{(s)})\le k+s$ for all $s\in [n-k]$.  Since $\{\alpha(I_{\Delta(M)}^{(s)})\}_{s\ge 1}$ is an increasing sequence, we have $\alpha(I_{\Delta(M)}^{(s)})=k+s$ for $s\in[n-k]$.

Observe that $I_{\Delta(M)}$ is the ideal generated by the $(k+1)$-fold products of the variables $x_1,\ldots,x_n$; or more classically, this is the defining ideal of a {\em monomial star configuration of codimension $n-k$}. The initial degree of any symbolic power of a monomial star configuration is known; see \cite[Corollary 4.6]{GHM13}.  Later, we recover \cite[Corollary 4.6]{GHM13} in full as a special case of \Cref{cor:sparse-paving}.

Finally, observe that by \Cref{thm:WaldschmidtGenHamWeights},
\[
\widehat{\alpha}(I_{\Delta(M)})=\min_{s\in [n-k]}\left\lbrace \frac{k+s}{s}=1+\frac{k}{s}
\right\rbrace=\frac{n}{n-k},
\]
which is the (well-known) Waldschmidt constant for the Stanley-Reisner ideal of the independence complex of a uniform matroid ${\mathrm U}_{n,k}$ (see, for example, ~\cite[Section~7.2]{BCGHJNSVV2016}).
\end{Example}

In \Cref{gHW-lb}, we observe that \( d_s(M) \ge  \alpha(I^{(s)}_{\Delta(M)}) \). This inequality becomes an equality in \Cref{Ex:MDS-uniform} for all $s\in [n-k]$.
We provide examples in \Cref{sec:noetherian-rees-algebra} (\Cref{ex:completeIntersections} and \Cref{ex:thetaMatroids}) that demonstrate the possibility of a strict inequality.  The correlation between these quantities is tied to the subadditive property of generalized Hamming weights, as we will see in subsequent sections.

\section{Subadditivity and duality for finite sequences}\label{sec:subadd-finite-seq}
In this section, we distill the relationship between generalized Hamming weights of matroids and the initial degrees of symbolic powers into purely sequence-theoretic terms, with an emphasis on subadditivity and its interaction with a form of duality which we call \textit{Wei duality}.  The details are somewhat technical, and while the results in this section are essential for the remainder of the paper, it is feasible for readers to skip this section for now and return to it as needed.

\subsection{The initial degree sequence}\label{ss:InitialDegreeSequence}

Let $\cU$ be a finite subset of $\NN$ satisfying $1\in\cU$ and let $d:\cU\to \NN$ be a function.  For each $a\in \cU$, we use sequence notation and write $d_a$ for $d(a)$.  We will use this index notation consistently for any function from $\cU$ to $\NN$ in this and the following sections.

\begin{Definition}\label{def:init-degree-sequence}
Let $\cU$ be a finite subset of $\NN$ with $1\in\cU$. For a function $d:\cU\to \NN$, define the   \textit{initial degree sequence} $\{\alpha_s(d)\}_{s\in \NN}$ by  
\[
\alpha_s=\alpha_s(d):=\min\left\lbrace \sum\limits_{i\in \cU} d_ib_i: b_i\in\ZZ_{\ge 0} \mbox{ for all } i\in\cU \mbox{ and } \sum\limits_{i\in \cU} ib_i=s \right\rbrace.
\]
\end{Definition}

We require $1\in\cU$ so that, for each $s\in\NN$, there is at least one tuple of nonnegative integers $(b_i)_{i\in\cU}$ satisfying $\sum\limits_{i\in \cU} ib_i=s$; namely $b_1=s$ and $b_i=0$ for all $i\neq 1\in \cU$.  This ensures that the minimum defining $\alpha_s$ is taken over a nonempty set and $\alpha_s\le sd_1$.

For what follows, recall that a sequence $\{x_i\}_{i\in\NN}\subset \RR$ is called \textit{subadditive} if $x_{s+t}\le x_s+x_t$ for all $s,t\in\NN$.

\begin{Lemma}\label{lem:appWald}
Let $\cU\subset \NN$ and $d:\cU\to \NN$ be as in \Cref{def:init-degree-sequence}, and let $\alpha_s=\alpha_s(d)$.  Then,
\begin{enumerate} 
\item \label{lem:appWalda}$\{\alpha_s\}_{s\in \NN}$ is a subadditive sequence and
\item \label{lem:appWaldb} $\lim\limits_{s\to\infty} \dfrac{\alpha_s}{s}=\min\left\lbrace \dfrac{d_i}{i}: i\in\cU\right\rbrace$.
\end{enumerate}
\end{Lemma}
\begin{proof}
We first show that $\{\alpha_s\}_{s \in \NN}$ is a subadditive sequence.  Let $s,t\in \NN$. By definition, $\alpha_{s}=\sum\limits_{i\in\cU} d_ib_i$ and $\alpha_{t}=\sum\limits_{i\in\cU} d_ic_i$ for some nonnegative integers $\{b_i:i\in \cU\}$ and $\{c_i:i\in\cU\}$ satisfying $\sum\limits_{i\in\cU}ib_i=s$ and $\sum\limits_{i\in\cU}ic_i=t$.  Therefore
\[
\alpha_s+\alpha_t = \sum\limits_{i\in\cU} d_ib_i+\sum\limits_{i\in\cU} d_ic_i=\sum\limits_{i\in\cU} d_i(b_i+c_i).
\]
Moreover, $\sum\limits_{i\in\cU}i(b_i+c_i)=s+t$, so by definition of $\alpha_{s+t}$,
\[
\alpha_{s+t}\le \sum\limits_{i\in\cU} d_i(b_i+c_i)=\alpha_s+\alpha_t,
\]
and thus $\{\alpha_s\}_{s \in \NN}$ is subadditive.  In particular, by Fekete's subadditive lemma,  $\lim\limits_{s\to\infty}\dfrac{\alpha_s}{s}=\inf\limits_{s\in\NN}\left\lbrace\dfrac{\alpha_s}{s}\right\rbrace$.

Now we prove $\inf\limits_{s\in\NN}\left\lbrace\dfrac{\alpha_s}{s}\right\rbrace=\min\left\lbrace \dfrac{d_i}{i}: i\in\cU\right\rbrace$.  We appeal to convex geometry.  Let $\RR^\cU$ be the $\RR$-vector space with standard basis vectors $\{\vec{e}_i:i\in\cU\}$ indexed by $\cU$.  Define
\[
A=\left\lbrace \sum\limits_{i\in\cU}\frac{b_i}{s}d_i\vec{e}_i~:~ b_i\in\ZZ_{\ge 0} \mbox{ for all } i\in \cU, s\in\NN, \mbox{ and } \sum\limits_{i\in \cU} ib_i=s\right\rbrace.
\]
Consider the functional $\beta=\sum\limits_{i\in\cU} y_i\in (\RR^{\cU})^\vee$, where $\{y_i~:~i\in\cU\}$ is the basis of $(\RR^{\cU})^\vee$ satisfying $y_i(\vec{e}_j)=\delta_{ij}$, where $\delta_{ij}=0$ if $i\neq j$ and $\delta_{ij}=1$ if $i=j$.  When we evaluate $\beta$ on some $\vec{x}=\sum\limits_{i\in\cU}\frac{b_i}{s}d_i\vec{e}_i\in A$, we obtain $\beta(\vec{x})=\sum\limits_{i\in\cU} \frac{b_i}{s}d_i$.  It is apparent from the definitions of $\alpha_s$ and $A$ that $\inf\limits_{\vec{x}\in A}\{\beta(\vec{x})\}=\inf\limits_{s\in\NN}\{\alpha_s/s\}$.

Now, let $\Delta$ be the simplex in $\RR^\cU$ which is the convex hull of $\left\{\frac{d_i}{i}\vec{e}_i:i\in \cU\right\}$.  We claim $A\subset\Delta$.
Suppose $\vec{x}\in A$ and $\vec{x}=\sum\limits_{i\in\cU}\frac{b_i}{s}d_i\vec{e}_i=\sum\limits_{i\in\cU}\frac{ib_i}{s}\frac{d_i}{i}\vec{e}_i\in A$.  Since $\vec{x}\in A$, then $\sum\limits_{i\in\cU} ib_i=s$, and so $\vec{x}$ is a convex combination of the vertices of $\Delta$. Hence $\vec{x}\in\Delta$.

It follows that $\inf\limits_{\vec{x}\in\Delta}\{\beta(\vec{x})\}\le \inf\limits_{\vec{x}\in A}\{\beta(\vec{x})\}$.  Since $\Delta$ is a simplex in $\RR^\cU$, it follows that $\inf\limits_{\vec{x}\in\Delta}\{\beta(\vec{x})\}$ is achieved along a face of $\Delta$ (see e.g. \cite[Chapter~2]{Ziegler95}) and hence at a vertex of $\Delta$, so $\inf\limits_{\vec{x}\in\Delta}\{\beta(\vec{x}\}=\min\left\{\frac{d_i}{i}:i\in\cU\right\}$.  Observe that $\frac{d_i}{i}\vec{e}_i\in A$ for $i\in\cU$ -- take $s=i$, $b_i=1$, and $b_j=0$ for $j\in\cU, j\neq i$.  Thus we have $\inf\limits_{\vec{x}\in A}\{\beta(\vec{x})\}=\min\left\{\frac{d_i}{i}:i\in\cU\right\}$ as well.
\end{proof}

We now consider how the subadditive properties of $d$ impact the sequence $\{\alpha_s\}_{s\in\NN}$.  For this, we need to define what subadditivity means for $\{d_i\}_{i\in\cU}$.

\begin{Definition}\label{def:subadd-dfunction}
Let $\cU\subset\NN$ and $d:\cU\to \NN$ be as in \Cref{def:init-degree-sequence}.  For a fixed $s\in\cU$, we say $d_s$ is a \textit{subadditive term} of $d$ if, for all $k\in\NN$ and all $a_1,\ldots,a_k\in\cU$ (not necessarily distinct) so that $\sum\limits_{i=1}^k a_i=s$, we have  $d_s\le \sum\limits_{i=1}^k d_{a_i}$.  Similarly, we say that $d_s$ is a \textit{strictly subadditive term} of $d$ if, for all $k\in\NN$ and all $a_1,\ldots,a_k\in\cU$ (not necessarily distinct) so that $\sum\limits_{i=1}^k a_i=s$, we have $d_s< \sum\limits_{i=1}^k d_{a_i}$.

By convention, the term $d_1$ is always considered both a subadditive term and a strictly subadditive term of $d$. We say that the function $d$ is subadditive if all its terms are subadditive.
\end{Definition}

\begin{Remark}
If $\cU=[n]=\{1,\ldots,n\}$ for some $n\in\NN$ and $d:\cU\to\NN$ is some function, then it is straightforward to check that $d$ is subadditive in the sense of \Cref{def:subadd-dfunction} if and only if $d_{i+j}\le d_i+d_j$ whenever $i,j,i+j\in \cU$.
\end{Remark}

\begin{Proposition}\label{prop:subadd-dfunction}
Let $\cU\subset\NN$ and $d:\cU\to \NN$ be as in \Cref{def:init-degree-sequence}. The following statements characterize the relationship between the subadditive terms of $d$ and its initial degree sequence $\{\alpha_s(d)\}_{s\in\NN}$. 
\begin{enumerate}  
\item \label{prop:subadd-dfunctiona} For any $s\in\cU$, $d_s$ is a subadditive term of $d$ if and only if $\alpha_s=d_s$.

\item \label{prop:subadd-dfunctionb} $d_s$ is a strictly subadditive term of $d$ if and only if $d_s< \sum\limits_{i\in\cU}b_id_i$ for all non-negative sets of integers $\{b_i:i\in\cU\}$ so that $\sum\limits_{i\in\cU} ib_i=s$ and $b_s=0$.

\item \label{prop:subadd-dfunctionc} $\alpha_s(d)=\alpha_s(d|_{\cU'})$, where $\cU'=\{i\in \cU: d_i \mbox{ is strictly subadditive}\}$. In other words, in order to compute $\alpha_s$, it suffices to only consider the strictly subadditive terms of $d$.

\item \label{prop:subadd-dfunctiond} $\lim\limits_{s \to \infty}\left\lbrace\frac{\alpha_s}{s}\right\rbrace=\min \limits_{i\in\cU'}\left\lbrace\frac{d_i}{i}\right\rbrace.
$
\end{enumerate}
\end{Proposition}

\begin{proof}
    For part \eqref{prop:subadd-dfunctiona}  observe that by \Cref{def:subadd-dfunction} we have $d_s$ is a subadditive term of $d$ if and only if $d_s\le \sum\limits_{i=1}^kd_{a_i}$ whenever $\{a_1, \ldots, a_k\}\subseteq \cU$ with $\sum\limits_{i=1}^ka_i=s$. For any such choice of $\{a_1, \ldots, a_k\}$ let $b_i\in\ZZ_{\ge 0}$ be the number of times that $i\in\cU$ appears in the sequence $\{a_1,\ldots,a_k\}$. Therefore, $d_s$ is subadditive if and only if $d_s\le \sum\limits_{i\in\cU}b_id_i$ for all non-negative sets of integers $\{b_i:i\in\cU\}$ so that $\sum\limits_{i\in\cU} ib_i=s$. In other words, $d_s\le 
    \alpha_s$. On the other hand, $\alpha_s\le d_s$ by construction. Thus, the result follows.

    Part \eqref{prop:subadd-dfunctionb} follows from \eqref{prop:subadd-dfunctiona} by noticing that $\sum\limits_{i\in \cU}ib_i=s$ implies that either $b_s=1$ or $b_s=0$. In the first case, $b_i=0$ for all $i\neq s$. Therefore, $d_s$ is a strictly subadditive term of $d$ if and only if $b_s=0$ and $d_s< \sum\limits_{i\in\cU}b_id_i$ for all non-negative sets of integers $\{b_i:i\in\cU\}$ so that $\sum\limits_{i\in\cU} ib_i=s$.

    For part \eqref{prop:subadd-dfunctionc} suppose that $d_\ell$ is \textit{not} strictly subadditive for some $\ell\in\cU$.  By part \eqref{prop:subadd-dfunctionb}, there exist non-negative integers $\{c_i:i\in\cU\}$ with $c_\ell=0$ so that $\sum\limits_{i\in\cU} ic_i=\ell$ and $\sum\limits_{i\in \cU}{c_id_i}\le d_\ell$.

Now fix $s\in\NN$.  Let $\{b_i\in\ZZ_{\ge 0}:i\in\cU\}$ satisfy $\sum\limits_{i\in\cU} ib_i=s$ and suppose that $b_\ell\neq 0$.  Then
\[
\sum\limits_{i\in\cU} b_id_i\ge \sum\limits_{i\neq \ell\in\cU} (b_i+b_\ell c_i)d_i.
\]
Observe also that 
\begin{align*}
\sum\limits_{i\neq\ell\in\cU} i(b_i+b_\ell c_i)
& =\sum\limits_{i\neq\ell\in\cU}ib_i+b_\ell\sum\limits_{i\in\cU}ic_i =\sum\limits_{i\neq \ell\in\cU} ib_i+\ell b_\ell  =\sum\limits_{i\in\cU} ib_i=s.
\end{align*}
Therefore, it suffices to take the prescribed minimum over indices coming from $\cU'$, as desired.

Finally, for part \eqref{prop:subadd-dfunctiond} notice that $\alpha_s(d)=\alpha_s(d|_{\cU'})$ by part \eqref{prop:subadd-dfunctionc} and hence 
$$\lim\limits_{s \to \infty}\left\lbrace\frac{\alpha_s(d)}{s}\right\rbrace=\lim\limits_{s \to \infty}\left\lbrace\frac{\alpha_s(d|_{\cU'})}{s}\right\rbrace=\min_{i\in\cU'}\left\lbrace\frac{d_i}{i}\right\rbrace,
$$
where the last equality follows from \Cref{lem:appWald}~\eqref{lem:appWaldb}.
\end{proof}

\subsection{Duality for increasing sequences}\label{ss:WeiDual}
In this section, we restrict our attention to those functions $d:[k]\to[n]$, which are strictly increasing, where $k\le n$ are positive integers.  We introduce a duality for such functions which is inspired by the relationship between generalized Hamming weights of a matroid and its dual (\Cref{lem:genMindistMatroid}~\eqref{lem:genMindistMatroide}).  In \Cref{ss:IDS-GHWs}, we will apply these results to generalized Hamming weights.
Our main duality operation will be a composite of the following two duality operations.
\begin{Definition}\label{def:bar-function}
Let $d:[k]\to[n]$ be a strictly increasing function.  We define $\overline{d}:[k]\to[n]$ as the function
\[
\overline{d}_i=n+1-d_{k+1-i}.
\]
If $\{d_j:j\in [k]\}$ is a proper subset of $[n]$ (i.e. $n>k$) then we define the \textit{gap function} $d^{\gap}:[n-k]\to [n]$ of $d$ by 
\[
d^{\gap}_i:= \mbox{the } i^{\text{\scriptsize{th}}} \mbox{ element of } [n]\setminus \{d_j:j\in [k]\} \mbox{ read in increasing order}.
\] 
That is, $d^{\gap}_1$ is the minimum element of $[n]\setminus \{d_j:j\in [k]\}$ and, for $2\le i\le n-k$, $d^{\gap}_i$ is the minimum element of $[n]\setminus \left(\{d_j:j\in [k]\}\cup \{d^{\gap}_j:1\le j<i\}\right)$. 
\end{Definition}

\begin{Lemma}\label{lem:basicdualfacts}
If $d:[k]\to[n]$ is a strictly  increasing function with $n>k$, then
\begin{enumerate} 
\item \label{lem:basicdualfacts-a} $\overline{\overline{d}}=d$,
\item \label{lem:basicdualfacts-b} $(d^{\gap})^{\gap}=d$, and
\item \label{lem:basicdualfacts-c} $(\overline{d})^{\gap}=\overline{d^{\gap}}$.
\end{enumerate}
\end{Lemma}
\begin{proof}
These are clear.
\end{proof}

The following lemma gives a formula for the gap function.

\begin{Lemma}\label{lem:gap-formula}
Let $d:[k]\to[n]$ be a strictly increasing function with $n>k$ and let $d^{\gap}:[n-k]\to [n]$ be its gap function.   For an integer $i\in [n-k]$, define
\[
\ell_i:=
\begin{cases}
\max\{j\in [k]: d_j-j<i\} & \mbox{if } d_1-1<i\\
0 & \mbox{if } d_1-1\ge i
\end{cases}
\]
and
\[
L_i := 
\begin{cases}
\min\{j\in [k]: d_j-j\ge i\} & \mbox{if } d_k-k\ge i\\
k+1 & \mbox{if } d_k-k< i.
\end{cases}
\]
Then, $d^{\gap}_i=i+\ell_i=i-1+L_i$ for all $i\in [n-k]$.
\end{Lemma}
\begin{proof}
Observe that, since $d$ is strictly increasing, $\{d_i-i\}_{i=1}^k$ is increasing .    It is straightforward to verify that $\ell_i= L_i-1$ for $i\in [n-k]$.
 To prove the equality $d^{\gap}_i=i+\ell_i=i-1+L_i$, we induct on $i\in [n-k]$.  

First, if $i=1$, $d^{\gap}_1$ is by definition $\min\left\lbrace [n]\setminus \{d_i: 1\le i\le k\}\right\rbrace$.  Since $d_i$ is strictly increasing and $d_1\ge 1$, the first gap equals the first index $j$ so that $d_j>j$.  That is, $d^{\gap}_1=\min\{j:d_j-j\ge 1\}$, as desired.  

Now suppose that for some $1\le i<n-k$, $d^{\gap}_i=i+\ell_i=i-1+L_i$.  If $\ell_i=k$, then $d_k-k<i$.  In other words, $d_k<i+k=d_i^{\gap}$.   In this case, every integer larger than or equal to $i+k$ and at most $n$ is in $[n]\setminus\{d_i:1\le i\le k\}$.  So $d^{\gap}_{i+1}=d_i+1=i+k+1$.  Since $d_k-k<i$, we also have $d_k-k<i+1$, so $\ell_{i+1}=\ell_i=k$.  Therefore, $d^{\gap}_{i+1}=i+1+\ell_{i+1}$, as desired.  Henceforth, we may assume $\ell_i<k$.

By definition of $\ell_i$, either $\ell_i=0$ which means $d_1 -1\ge i$ or $d_{\ell_i}-\ell_i<i$, so $d_{\ell_i}<i+\ell_i=d_i^{\gap}$. Now suppose $d_{\ell_i+1}-(\ell_i+1)\ge i+1$.  Then $d_{\ell_i+1}\ge i+\ell_i+2=d^{\gap}_{i}+2$.  In this case, $d^{\gap}_i+1$ is also a gap, and $\ell_i=\ell_{i+1}$, so this matches the given formula.

 Finally, suppose that $d_{\ell_i+1}-(\ell_i+1)=i$, so $d_{\ell_i+1}=i+\ell_i+1=d^{\gap}_i+1$.  By definition, $\ell_{i+1}$ is the maximum index $j$ so that $d_j-j<i+1$, hence $\ell_i+1\le\ell_{i+1}$.  Since $\{d_j-j\}_{j=1}^k$ is an increasing function, $d_{\ell_{i+1}}-\ell_{i+1}\ge d_{\ell_i+1}-(\ell_i+1)=i$.  Thus, $\ell_{i+1}$ is the maximum index $j$ so that $d_j-j=i$.  It follows that, for $1\le j\le \ell_{i+1}-\ell_{i}$, $d_{\ell_i+j}=i+\ell_i+j=d^{\gap}_i+j$.  Furthermore, either $\ell_{i+1}=k$ or $d_{\ell_{i+1}+1}-(\ell_{i+1}+1)\ge i+1$, so $d_{\ell_{i+1}+1}\ge i+1+\ell_{i+1}+1$.  In either case, $i+1+\ell_{i+1}$ is a gap, and it is the smallest gap that appears after $d^{\gap}_i$, and so $d^{\gap}_{i+1}=i+1+\ell_{i+1}$, as desired. 
\end{proof}

\begin{Definition}\label{def:WeiDual}
Let $d:[k]\to [n]$ be a strictly increasing function with $n>k$.  The \textit{Wei dual} of $d$, written $d^*$, is the gap function of $\overline{d}$.  That is, $d^*_i$ is the $i^{\text{\scriptsize{th}}}$ element of $[n]\setminus \{n+1-d_i:  i\in [k]\}$ read in increasing order.  We say $d$ is \textit{cosubadditive} if $d^*$ is subadditive.
\end{Definition}

\begin{Remark}
If $\C$ is an $[n,k]$-code and $d:[k]\to [n]$ is defined by $d_i:=d_i(\C)$ -- where $d_i(\C)$ is the $i$-th generalized Hamming weight of $\C$ -- then Wei proves in~\cite[Theorem~3]{We} that $\{d^*_i:1\le i\le n-k\}$ is the weight hierarchy of the dual code of $\C$ (see \Cref{lem:genMindistMatroid} for matroids).  This is why we have chosen the name `Wei dual' for $d^*$.
\end{Remark}

Our primary finding in this section is a characterization of the subadditivity of \( d^* \) based solely on \( d \), along with a formula for \( \alpha_s(d) \) under the condition that both \( d \) and \( d^* \) are subadditive. We will explore the subadditivity conditions through a series of lemmas.

\begin{Lemma}\label{lem:bar-superadd}
Let $d:[k]\to[n]$ be a strictly increasing function with $n>k$.  The following are equivalent:
\begin{enumerate} 
\item $d_r+d_k\le d_a+d_b$ for all $a,b,r\in[k]$ so that $r+k=a+b$.
\item $\overline{d}_s+\overline{d}_t\le \overline{d}_{s+t-1}+\overline{d}_1$ for all $s,t\in [k]$ so that $s+t-1\in [k]$.
\end{enumerate}
\end{Lemma}
\begin{proof}
We first show $(\rm a)$ implies $(\rm b)$.  Let $s,t\in [k]$ so that $s+t-1\in [k]$.  Put $a=k+1-s$,  $b=k+1-t$, and $r=a+b-k=k+2-s-t$.  Since we assume $(\rm a)$ is true, $d_r+d_k\le d_a+d_b$ and so
\[
\bar{d}_{s+t-1}+\bar{d}_1=n+1-d_r+n+1-d_k\ge n+1-d_a+n+1-d_b=\bar{d}_s+\bar{d}_t.
\]

Now we show $(\rm b)$ implies $(\rm a)$.  Let $a,b,r\in [k]$ so that $r+k=a+b$.  Put $s=k+1-a$ and $t=k+1-b$.  Then $s+t-1=2k+1-(a+b)=2k+1-(r+k)=k+1-r$.  Since we assume $(\rm b)$, $\bar{d}_s+\bar{d}_t\le \bar{d}_{s+t-1}+\bar{d}_1$.  Thus
\[
k+1-d_a+k+1-d_b\le k+1-d_r+k+1-d_k.
\]
Re-arranging, we get $d_r+d_k\le d_a+d_b$, as desired.
\end{proof}

\begin{Lemma}\label{lem:gap-subadd}
Let $d:[k]\to [n]$ be a strictly increasing function with $n>k$.  Then
\begin{enumerate} 
\item \label{lem:gap-subadda} If $d$ is subadditive and $d_k=n$, then $d^{\gap}_s+d^{\gap}_t\le d^{\gap}_{s+t-1}+1$ for all $s,t\in [n-k]$ so that $s+t-1\in [n-k]$.
\item \label{lem:gap-subaddb} If $d_s+d_t\le d_{s+t-1}+1$ for all $s,t\in [k]$ so that $s+t-1\in [k]$, then $d^{\gap}$ is subadditive.
\end{enumerate}
\end{Lemma}
\begin{proof}
For part $(\rm a)$, assume that $d$ is subadditive and $d_k=n$.  Let $s,t\in [n-k]$ so that $s+t-1\in [n-k]$.  It is equivalent to prove that $(d_s^{\gap}-s)+(d_t^{\gap}-t)\le d^{\gap}_{s+t-1}-(s+t-1)$.  Using the notation of \Cref{lem:gap-formula}, it suffices to prove that $\ell_s+\ell_t\le \ell_{s+t-1}$.  Recall
\[
\ell_i:=
\begin{cases}
\max\{j\in [k]: d_j-j<i\} & \mbox{if } d_1-1<i\\
0 & \mbox{if } d_1-1\ge i.
\end{cases}
\]
If $\ell_s=0$ or $\ell_t=0$, we are done since $\{\ell_i\}_{i=1}^{n-k}$ is an increasing sequence.  Henceforth, we assume $\ell_s,\ell_t\ge 1$.

We now show that $\ell_s+\ell_t\le k$.  Assume for the sake of contradiction that $\ell_s+\ell_t>k$.  Then we may choose indices $1\le i\le \ell_s$, $1\le j\le \ell_t$ so that $i+j=k$.  Since $d$ is strictly increasing, $d_i-i\le d_{\ell_s}-\ell_s$ and $d_j-j\le d_{\ell_t}-\ell_t$.  By definition, $d_{\ell_s}-\ell_s\le s-1$ and $d_{\ell_t}-\ell_t\le t-1$.  So we have
\[
s+t-1>(d_{\ell_s}-\ell_s)+(d_{\ell_t}-\ell_t)\ge (d_i-i)+(d_j-j)\ge d_k-k=n-k,
\]
where the final inequality follows since $d$ is subadditive and $i+j=k$ and the final equality follows since $d_k=n$.  This contradicts that $s+t-1\in [n-k]$.  So $\ell_s+\ell_t\le k$.

It now follows that
\[
s+t-1>(d_{\ell_s}-\ell_s)+(d_{\ell_t}-\ell_t)\ge d_{\ell_s+\ell_t}-(\ell_s+\ell_t),
\]
where the final inequality follows since $d$ is subadditive.  Since $d_{\ell_s+\ell_t}-(\ell_s+\ell_t)<s+t-1$, it follows (from the definition of $\ell_i$) that $\ell_s+\ell_t\le \ell_{s+t-1}$.  This finishes the proof of part $(\rm a)$.

For part $(\rm b)$, assume that $d_s+d_t\le d_{s+t-1}+1$ for all $s,t\in [k]$ so that $s+t-1\in [k]$.  Let $i,j\in [n-k]$ so that $i+j\in [n-k]$.  We prove that $d^{\gap}_i+d^{\gap}_j\ge d^{\gap}_{i+j}$.  It suffices to prove that $(d^{\gap}_i-i)+(d^{\gap}_j-j)\ge d^{\gap}_{i+j}-(i+j)$.  Using the notation of \Cref{lem:gap-formula}, it suffices to prove that $(L_i-1)+(L_j-1)\ge L_{i+j}-1$, or equivalently $L_{i+j}\le L_i+L_j-1$.  Recall that
\[
L_i := 
\begin{cases}
\min\{j\in [k]: d_j-j\ge i\} & \mbox{if } d_k-k\ge i\\
k+1 & \mbox{if } d_k-k< i.
\end{cases}
\]
If $L_i+L_j\ge k+2$, we are done, since $L_{i+j}\le k+1$ by definition.  So we assume $L_i+L_j-1\le k$.

By definition, $d_{L_i}-L_i\ge i$ and $d_{L_j}-L_j\ge j$.  Hence
\[
i+j\le (d_{L_i}-L_i)+(d_{L_j}-L_j)\le d_{L_i+L_j-1}-(L_i+L_j-1),
\]
where the final inequality follows from our hypothesis.  Therefore, by definition, $L_{i+j}\le L_i+L_j-1$, as desired.
\end{proof}

Our next result characterizes the subadditivity of $d^*$ solely in terms of $d$.

\begin{Theorem}\label{thm:d^*subadd}
Let $d:[k]\to[n]$ be a strictly increasing function with $n>k$, $d_1>1$, and $d_k=n$.  Then $d$ is cosubadditive if and only if $d_k+d_r\le d_a+d_b$ for all $r,a,b\in [k]$ satisfying $k+r=a+b$.
\end{Theorem}
\begin{proof}
We first show that $d^*_{n-k}=n$.  Since $d_1>1$, $\overline{d}_{k}=n+1-d_{k+1-k}=n+1-d_1<n$.  Therefore, $n$ is the largest gap of $\overline{d}$, and $d^*_{n-k}=n$.

Now assume that $d^*=(\overline{d})^{\gap}$ is subadditive.  By \Cref{lem:basicdualfacts}~\eqref{lem:basicdualfacts-b}, $(d^*)^{\gap}=\overline{d}$ and by \Cref{lem:gap-subadd}~\eqref{lem:gap-subadda}, $\overline{d}$ satisfies that $\overline{d}_s+\overline{d}_t\le \overline{d}_{s+t-1}+\overline{d}_1$ for all $s,t\in [k]$ so that $s+t-1\in [k]$.  Now, by \Cref{lem:basicdualfacts}~\eqref{lem:basicdualfacts-a}, $\overline{\overline{d}}=d$.  By \Cref{lem:bar-superadd}, $d_r+d_k\le d_a+d_b$ for all $a,b,r\in [k]$ so that $a+b=k+r$, as desired.

Finally, assume that $d_r+d_k\le d_a+d_b$ for all $a,b,r\in [k]$ so that $a+b=k+r$.  Since $d_k=n$, $\bar{d}_1=n+1-d_k=1$.  By \Cref{lem:bar-superadd}, $\overline{d}_s+\overline{d}_t\le \overline{d}_{s+t-1}+\bar{d}_1=\overline{d}_{s+t-1}+1$ for all $s,t\in[k]$ so that $s+t-1\in [k]$.  By \Cref{lem:gap-subadd}~\eqref{lem:gap-subaddb}, $(\overline{d})^{\gap}=d^*$ is subadditive, as desired.
\end{proof}

We now characterize subadditivity and cosubadditivity of $d$ using its initial degree sequence.

\begin{Theorem}\label{lem:extended-subadd-alpha-formula}
Let $d:[k]\to [n]$ be a strictly increasing function with $n>k$, $d_1>1$, and $d_k=n$.  Then $d$ is \rsubad{} if and only if $\alpha_s(d)=qd_k+d_r$ for all $s =kq+r \in\NN$ with $r \in [k]$ and $q \in \ZZ_{\ge 0}$.
\end{Theorem}
\begin{proof}

First, suppose that $\alpha_s(d)=qd_k+d_r$ for all $s=kq+r \in\NN$ with $r \in [k]$ and $q \in \ZZ_{\ge 0}$.  By \Cref{lem:appWald}~\eqref{lem:appWalda}, $\{\alpha_s(d)\}_{s \in \NN}$ is a subadditive sequence.  By taking $q=0$ and $r=s \in [k]$, we have $\alpha_s(d)=d_s$ for all $s \in [k]$  so $d$ is subadditive.  Now, let $a,b,r\in [k]$ so that $k+r=a+b$. Note that for any $r \in [k]$,  $\alpha_{k+r}(d)=d_k+d_r$.  Since $\{\alpha_s(d)\}_{s \in \NN}$ is subadditive and $k+r=a+b$, $d_k+d_r=\alpha_{k+r}(d)\le \alpha_a(d)+\alpha_b(d)=d_a+d_b$.  Thus, by \Cref{thm:d^*subadd}, $d$ is cosubadditive.

Now suppose that $d$ is \rsubad{}.  By \Cref{thm:d^*subadd}, $d$ is subadditive and $d_r+d_k\le d_a+d_b$ for all $r,a,b\in[k]$ so that $r+k=a+b$.

We induct on $s\in\NN$ to prove that $\alpha_s(d)=qd_k+d_r$ for all $s=kq+r\in\NN$ with $r \in [k]$ and $q \in \ZZ_{\ge 0}$. 
When $s\in [k]$, $q=0$, and therefor, it is clear from \Cref{prop:subadd-dfunction}~\eqref{prop:subadd-dfunctiona} that $\alpha_s(d)=qd_k+d_s=d_s$.

Now suppose that $s>k$ and $\alpha_t(d)=q_td_k+d_{r_t}$ for all $t=q_tk+r_t<s$ with $r_t \in [k]$ and $q_t \in \ZZ_{\ge 0}$. Pick $b_1,\ldots,b_k\in\ZZ_{\ge 0}$ so that $\sum\limits_{i=1}^k ib_i=s$ and $\alpha_s(d)=\sum\limits_{i=1}^k b_id_i$, which is possible by \Cref{def:init-degree-sequence}.  Let $\ell\in[k]$ be an index for which $b_\ell> 0$.  Observe that, since $\ell\le k<s$, $s-\ell>0$.  So $0<s-\ell=\sum\limits_{i=1}^k ib_i-\ell=\sum\limits_{i\in[k],i\neq \ell}ib_i+\ell(b_\ell-1)$, and thus, $\alpha_s(d)-d_\ell=\sum\limits_{i=1}^k b_id_i-d_\ell=\sum\limits_{i\in[k],i\neq \ell} b_id_i+(b_\ell-1)d_\ell\ge \alpha_{s-\ell}(d)$, by \Cref{def:init-degree-sequence}.

Let $q' \in \ZZ_{\ge 0}$ and $r' \in [k]$ be such that  $s-\ell=q'k+r'$.  By our inductive hypothesis, $\alpha_{s-\ell}(d)=q'd_k+d_{r'}$. We now observe that $\alpha_s(d)\le qd_k+d_r$, where $q \in \ZZ_{\ge 0}$ and $r \in [k]$ so that $s=qk+r$.  This follows since we can choose $b_k=q$, $b_r=1$ and $b_i=0$ for $i \neq r,k$ in \Cref{def:init-degree-sequence}, so $\sum\limits_{i=1}^k ib_i=kq+r=s$, and hence, $\alpha_s(d)\le \sum\limits_{i=1}^k b_id_i=qd_k+d_r$.

From the preceding paragraphs, it follows that
\[
\alpha_{s-\ell}(d)+d_{\ell}=q'd_k+d_{r'}+d_{\ell}\le \alpha_s(d)\le qd_k+d_r.
\]
We now prove that $qd_k+d_r\le q'd_k+d_{r'}+d_{\ell}$, which will imply $\alpha_s(d)= qd_k+d_r$ and finish the induction.

We proceed with two cases based on how $q,q',r,r',$ and $\ell$ are related to one another.  Since $ \ell\in [k]$, the two cases are:
\begin{enumerate}
\item[(1)] $q=q'$ and $r=r'+\ell$ and
\item[(2)] $q=q'+1$ and $r=r'+\ell-k$
\end{enumerate}
In the first case, since $d$ is subadditive, $d_r\le d_{r'}+d_{\ell}$, which implies $qd_k+d_r=q'd_k+d_{r}\le q'd_k+d_{r'}+d_{\ell}$, as desired.  In the second case, since $d$ is \rsubad{} and $k+r=r'+\ell$, $d_k+d_r\le d_{r'}+d_{\ell}$.  Hence,
\[
qd_k+d_r \le qd_k+d_{r'}+d_{\ell} -d_k= q'd_k+d_{r'}+d_\ell,
\]
as desired.
\end{proof}

\begin{Corollary}\label{cor:sun&cosub}
Let $d:[k]\to [n]$ be a strictly increasing function with $n>k$, $d_1>1$, and $d_k=n$.  Then the following are equivalent.
\begin{enumerate} 
\item $d$ is subadditive and $d_r+d_k\le d_a+d_b$ for all $r,a,b\in [k]$ so that $r+k=a+b$.
\item $d$ is \rsubad{}.
\item $d^*$ is \rsubad{}.
\item $\alpha_s(d)=qd_k+d_r$ for all $s=kq+r\in\NN$ with $r \in [k]$ and $q \in \ZZ_{\ge 0}$.
\item $\alpha_s(d^*)=qd^*_{n-k}+d^*_r$ for all $s=(n-k)q+r\in\NN$ with $r \in [n-k]$ and $q \in \ZZ_{\ge 0}$.
\end{enumerate}
\end{Corollary}
\begin{proof}
This follows from \Cref{thm:d^*subadd} and \Cref{lem:extended-subadd-alpha-formula}.
\end{proof}

\subsection{Sequence of Generalized Hamming weights of matroids and their Wei dual}\label{ss:IDS-GHWs}

In this section, we apply the results of \Cref{ss:WeiDual} to the sequence of generalized Hamming weights of a matroid. Throughout this subsection, let $M$ be a matroid of rank $k$ on the ground set $E$ of size $n$, and let $\{d_i(M)\}_{i=1}^{n-k}$ and $\{d_i(M^*)\}_{i=1}^k$ be the generalized Hamming weights of $M$ and $M^*$, respectively. Recall from \Cref{sec:prelims} that for $ r \in [n-k],$ the $r$-th generalized Hamming weight of a matroid $M$ of rank $k$ on a ground set $E$ of size $n$ is  
\[
d_r(M)=\min\{|U|~:~ U \subseteq E \text{ and } |U|-\rk_{M}(U)=r\}.
\]
The following lemma is our primary reason for introducing and studying the Wei dual of a finite sequence in \Cref{ss:WeiDual}.
\begin{Lemma}\label{lem:WeiDualityMatroids}
 The Wei dual of the sequence $\{d_i(M)\}_{i=1}^{n-k}$ of generalized Hamming weights of $M$ is the sequence $\{d_i(M^*)\}_{i=1}^{k}$ of generalized Hamming weights of $M^*$.
\end{Lemma}
\begin{proof}
This follows from \Cref{def:WeiDual} and \Cref{lem:genMindistMatroid}~\eqref{lem:genMindistMatroide}.
\end{proof}

For an integer $0 \le i \le k$, we define the sequence $$f_i^{\max}(M):=\max\{|F|~:~ F\subseteq E \text{ and } \rk_M(F)=i\}.$$ The sequence $\{f_i^{\max}(M)\}_{i=0}^k$ collects the maximum cardinality of flats of each rank of a matroid. It follows from \Cref{cor:equivalentInterpretationsGeneralizedHammingWeights} that  for $r \in [n-k]$, 
\[
d_r(M)=n-f_{n-k-r}^{\max}(M^*).
\] 
The dual version of this is as follows: \begin{align}\label{eq:flat-dual}
  d_r(M^*)=n-f_{k-r}^{\max}(M) \text{ for }  r \in [k].  
\end{align}

In the following two propositions, we summarize equivalent conditions for the sequence of generalized Hamming weights to form a subadditive sequence.

\begin{Proposition}\label{prop:d-sub}
Let $M$ be a matroid of rank $k$ on a ground set $E$ of size $n$. Then, the following are equivalent:
    \begin{enumerate} 
         \item \label{prop:d-suba} $\{d_i(M)\}_{i=1}^{n-k}$ is subadditive.
         \item \label{prop:d-subb} For all $a,b,r \in [k]$ so that $r+k=a+b,$  $d_r(M^*)+d_k(M^*)\le d_a(M^*)+d_b(M^*)$. 
        \item \label{prop:d-subc} For all $s,t \in [k]$ so that $s+t-1 \in[k],$ $\overline{d_s(M^*)}+\overline{d_t(M^*)}\le \overline{d_{s+t-1}(M^*)}+1$ .
        \item \label{prop:d-subd} For all $0 \le i,j \le i+j \le k-1,$ $f_i^{\max}(M) +f_j^{\max}(M) \le f_{i+j}^{\max}(M)$. 
       
    \end{enumerate}
\end{Proposition}
\begin{proof}
Let $d_i=d_i(M)$ for each $i \in [n-k]$.  By \Cref{lem:WeiDualityMatroids}, $d^*=\{d_i(M^*)\}_{i=1}^k$.
Since $M$ and $M^*$ are both loopless, it follows that $d^*_1=d_1(M^*)>1$ and $d^*_k=d_k(M^*)=n$.  The equivalence of $(\rm a)$ and $(\rm b)$ follows by applying \Cref{thm:d^*subadd} to $d^*$.  The equivalence of $(\rm b)$ and $(\rm c)$ follows from \Cref{lem:bar-superadd}.

      So, we show that $(\rm c) \iff (\rm d)$. By \Cref{def:bar-function}, $\overline{d_i(M^*)}=n+1-d_{k+1-i}(M^*)$ for each $i\in [k]$. This implies that $\overline{d_s(M^*)}+\overline{d_t(M^*)}= n+1-d_{k+1-s}(M^*) +n+1-d_{k+1-t}(M^*)= n-d_{k-(s-1)}(M^*) +n-d_{k-(t-1)}(M^*)+2$ and  $\overline{d_{s+t-1}(M^*)}+1=n+1-d_{k+1-(s+t-1)}(M^*) +1=n-d_{k-(s+t-2)}(M^*) +2$.  Thus, $(\rm c)$ holds if and only if for all  $s,t \in [k]$ so that  $s+t-1 \in [k]$, the inequality below holds \begin{align*}
         n-d_{k-(s-1)}(M^*) +n-d_{k-(t-1)}(M^*) \le n-d_{k-(s+t-2)}(M^*) .
    \end{align*} In view of \Cref{eq:flat-dual}, statement $(\rm c)$ holds if and only if the inequality $f_{s-1}^{\max}(M)+f_{t-1}^{\max}(M) \le f_{s+t-2}^{\max}(M)$ is satisfied for all $s,t\in [k]$ so that $ s+t-1 \in [k]$. By rearranging the subscripts from the previous line, we conclude that statement  $(\rm c)$ is equivalent to statement $(\rm d)$.  Hence, the assertions follow. 
\end{proof}

\begin{Proposition}\label{prop:d-cosub}
Let $M$ be a matroid of rank $k$ on a ground set $E$ of size $n$. Then, the following are equivalent:
    \begin{enumerate} 
         \item \label{prop:d-cosuba} $\{d_i(M^*)\}_{i=1}^{k}$ is subadditive.
         \item \label{prop:d-cosubb} For all $1 \le i,j  \le k-1$ with $i+j\ge k$,  $f_i^{\max}(M) +f_j^{\max}(M) \le f_k^{\max}(M)+f_{i+j-k}^{\max}(M)$. 
\end{enumerate}
\end{Proposition}
\begin{proof}
The equivalence of $(\rm a)$ and $(\rm b)$ follows using \Cref{eq:flat-dual}. Indeed, in view of \Cref{eq:flat-dual}, for all $a,b \in [k]$ so that $a+b \in [k]$,  $d_{a+b}(M^*)\le d_a(M^*)+d_b(M^*)$  if and only if $n-f_{k-a-b}^{\max}(M) \le n-f_{k-a}^{\max}(M)+n-f_{k-b}^{\max}(M)$. Thus, by rearranging the subscripts as $i=k-a, j=k-b$ and simplifying, we get that $\{d_i(M^*)\}_{i=1}^{k}$ is subadditive if and only if for all $0 \le i,j  \le k-1$ with $i+j\ge k$,  $f_i^{\max}(M) +f_j^{\max}(M) \le n +f_{i+j-k}^{\max}(M)=f_k^{\max}(M) +f_{i+j-k}^{\max}(M)$. Hence, the assertion follows. 
\end{proof}

We conclude this section by providing a sufficient condition for the subadditivity and cosubadditivity of generalized Hamming weights based on the discrete convexity (see \Cref{rmk:dc}) of the sequence $\{f_i^{\max}(M)\}_{i=0}^{k}$.
\begin{Proposition}\label{prop:flat-sub&cosub}
Let $M$ be a matroid of rank $k$ on a ground set $E$ of size $n$. If $f_{i+1}^{\max}(M)-f_{i}^{\max}(M) \ge f_{i}^{\max}(M)-f_{i-1}^{\max}(M)$ for all $i=1,\ldots,k-1$,  then $\{d_i(M)\}_{i=1}^{n-k}$ is subadditive and cosubadditive. 
\end{Proposition}

\begin{proof}
 It follows from \Cref{prop:d-sub} that $\{d_i(M)\}_{i=1}^{n-k}$ is subadditive if and only if $f_i^{\max}(M)+f_j^{\max}(M)\le f_{i+j}^{\max}(M)$ for all $0\le i, j \le i+j \le k-1.$ Since $M$ is loopless $f_0^{\max}(M)=0$, and hence, $f_0^{\max}(M)+f_j^{\max}(M)\le f_{j}^{\max}(M)$ for all $0\le  j \le k-1.$ So, take $1 \le i\le j \le i+j \le k-1$, and consider \begin{align*}
     f_{i+j}^{\max}(M)-f_j^{\max}(M)&=\sum\limits_{\ell=1}^i\left(f_{\ell+j}^{\max}(M)-f_{\ell-1+j}^{\max}(M)\right)\\&\ge \sum\limits_{\ell=1}^i\left(f_{\ell+j-1}^{\max}(M)-f_{\ell+j-2}^{\max}(M)\right)\\&  ~~~~~~~~~~~~\vdots \text{ (use the hypothesis } j \text{ times on each summand) }\\& \ge \sum\limits_{\ell=1}^i\left(f_{\ell+j-j}^{\max}(M)-f_{\ell-1+j-j}^{\max}(M)\right)\\&= \sum\limits_{\ell=1}^i\left(f_{\ell}^{\max}(M)-f_{\ell-1}^{\max}(M)\right)\\&=f_i^{\max}(M)-f_0^{\max}(M)=f_i^{\max}(M). 
 \end{align*}  This proves that $\{d_i(M)\}_{i=1}^{n-k}$ is subadditive. 

 It follows from \Cref{prop:d-cosub} that $\{d_i(M)\}_{i=1}^{n-k}$ is cosubadditive if and only if for all $0 \le i \le j  \le k-1$ with $i+j\ge k$,  $f_i^{\max}(M) +f_j^{\max}(M) \le f_k^{\max}(M)+f_{i+j-k}^{\max}(M)$.  So, take $1 \le i\le j \le  k-1$ with $i+j \ge k$, and consider \begin{align*}
     f_{k}^{\max}(M)-f_j^{\max}(M)&=\sum\limits_{\ell=1}^{k-j}\left(f_{\ell+j}^{\max}(M)-f_{\ell-1+j}^{\max}(M)\right)\\&\ge \sum\limits_{\ell=1}^{k-j}\left(f_{\ell+j-1}^{\max}(M)-f_{\ell+j-2}^{\max}(M)\right)\\&  ~~~~~~~~~~~~\vdots \text{ (use the hypothesis } k-i \text{ times on each summand) }\\& \ge \sum\limits_{\ell=1}^{k-j}\left(f_{\ell+j-k+i}^{\max}(M)-f_{\ell-1+j-k+i}^{\max}(M)\right)\\&=f_{i}^{\max}(M)-f_{i+j-k}^{\max}(M). 
 \end{align*}  This proves that $\{d_i(M)\}_{i=1}^{n-k}$ is cosubadditive. 
\end{proof}

\begin{Remark}\label{rmk:dc}
In the literature, the condition $f_{i+1}^{\max}(M)-f_{i}^{\max}(M) \ge f_{i}^{\max}(M)-f_{i-1}^{\max}(M)$ for all $i=1,\ldots,k-1$ is taken as a definition of discrete convexity (e.g. \cite{Yuceer-2002}).
\end{Remark}

\section{Finitely Generated Rees Algebras and Matroids}\label{sec:noetherian-rees-algebra}

In this section, we explain how the initial degree sequence of a function, studied in \Cref{sec:subadd-finite-seq}, arises naturally in the context of finitely generated Rees algebras.  We then give an explicit set of generators for the symbolic Rees algebra of the Stanley-Reisner ideal of a matroid (\Cref{thm:GeneratorsOfSymbolicReesAlgebraOfMatroid}).  The degrees of these generators are closely tied to the generalized Hamming weights of the matroid; applying the theoretical framework of \Cref{sec:subadd-finite-seq} in this context yields a number of results concerning the initial degree of the symbolic powers of the Stanley-Reisner ideal of the matroid and its Waldschmidt constant (\Cref{thm:SR}).

\subsection{Noetherian Rees algebras}
Let $S=\bigoplus\limits_{i\ge 0} S_i$ be a Noetherian graded integral domain and $\mathcal{F}=\{I_i\}_{i\ge 1}$ be a graded family of homogeneous ideals with the convention that $I_0=S$.  Recall that $\{I_i\}_{i\ge 1}$ is a graded family provided that $I_iI_j\subset I_{i+j}$ for all $i,j\ge 1$.  The \textit{Rees algebra} of $\mathcal{F}$ is the graded $S[T]$-algebra
\[
\cR(\mathcal{F}):=\bigoplus_{i\ge 0}I_iT^i\subseteq S[T],
\]
where $T$ is an indeterminant over $S$.

Let $K=\mbox{frac}(S)$, the fraction field of $S$. 
 Recall that a \textit{discrete valuation} $\nu:K \setminus \{0\} \to\ZZ$ is a function which satisfies $\nu(fg)=\nu(f)+\nu(g)$ and $\nu(f+g)\ge \min\{\nu(f),\nu(g)\}$ for all $f,g\in K \setminus \{0\}$.  A valuation on $K$ is completely determined by its values on $S$, so we will often consider the domain of the valuation to be $S$ rather than $K$.  We say a valuation $\nu: K \setminus \{0\}\to \ZZ$ is an $S$-valuation if $\nu$ only takes non-negative values on $S \setminus \{0\}$.  For an $S$-valuation $\nu$ and for an ideal $J\subseteq S$, let $\nu(J):=\min\{\nu(f):f\in J \setminus \{0\}\}$.

 The valuation relevant for our work is the initial degree valuation $\alpha~:~S\setminus\{0\} \to \ZZ$ defined by $\alpha(f):=\min\{i~:~f\in\bigoplus_{k\ge i} S_k\}$.  If $f$ is homogeneous -- that is, $f\in S_k$ for some $k$ -- then $\alpha(f)=k$.

If $\mathcal{F}=\{I_i\}_{i\ge 1}$ is a graded family of ideals, then $\{\nu(I_i)\}_{i\ge 1}$ is a subadditive sequence.  This follows from the fact that $\nu(I_sI_t)=\nu(I_s)+\nu(I_t)$ (because $\nu$ is a valuation) and $I_sI_t\subseteq I_{s+t}$ (because $\mathcal{F}$ is a graded family).  It follows that $\lim\limits_{s\to\infty} \frac{\nu(I_s)}{s}$ exists by Fekete's lemma (also see \cite{HJKN23}). 

\begin{Definition}
Let $\mathcal{F}=\{I_i\}_{i\ge 1}$ be a graded family of homogeneous ideals in $S$, and $\nu~:~ S \setminus \{0\} \to\ZZ$ an $S$-valuation.  We define
\[
\widehat{\nu}(\mathcal{F}):=\lim_{s\to\infty}\frac{\nu(I_s)}{s}.
\]
If $\mathcal{F}=\{I^{(i)}\}_{i\ge 1}$ is the graded family of symbolic powers of an ideal $I$, then we write $\widehat{\nu}(I)$ instead of $\widehat{\nu}(\mathcal{F})$.
\end{Definition}

The work in \cite{DFMS19} established the significance of the invariant $\widehat{\nu}(I)$ in the context of asymptotic resurgence for homogeneous ideals. Building on this perspective, the authors in \cite{HJKN23, HKJN24} introduced the invariant $\widehat{\nu}(\mathcal{F})$, which plays a central role in addressing the containment problem for pairs of graded families of ideals. In this context, we provide a sequential method for calculating these invariants specifically for Noetherian graded families of ideals: 

\begin{Theorem}\label{thm:ReesAlgebraFiltration}
Let $\mathcal{F}=\{I_i\}_{i\ge 1}$ be a graded family of homogeneous ideals, and $\nu$ be an $S$-valuation.  Suppose that $\cR(\mathcal{F})$ is finitely generated as an $S$-algebra by $\mathcal{G}=\{f_1T^{n_1},f_2T^{n_2},\ldots,f_kT^{n_k}\}$, where $f_1,\ldots,f_k$ are homogeneous elements of $S$ and $n_1,\ldots,n_k\in\NN$ (these are not necessarily distinct).  Let $\cU=\{n_1,\ldots,n_k\}$ and define $d:\cU\to\NN$ by $d_i=\min\{\nu(f_j)~:~j\in [k] \text{ and } n_j=i\}$ for $i\in\cU$.  Then, for any $s\in\NN$,
\[
\nu(I_s)=\min\left\lbrace \sum\limits_{i\in\cU} d_ib_i:b_i\in\ZZ_{\ge 0} \mbox{ for all } i\in\cU \mbox{ and } \sum\limits_{i\in\cU} ib_i=s \right\rbrace.
\]
In other words, $\{\nu(I_s)\}_{s\ge 1}$ is {the initial degree sequence} of $d$ -- see \Cref{def:init-degree-sequence}.  It follows that

\begin{enumerate} 
\item $\nu(I_s)=d_s$ for $s\in\cU$ if and only if $d_s$ is a subadditive term of $d$.
\item $\widehat{\nu}(\mathcal{F})=\min\left\lbrace\dfrac{d_i}{i}:i\in\cU\right\rbrace=\min\left\lbrace\dfrac{d_i}{i}:i\in\cU, d_i \mbox{ is a strictly subadditive term of } d\right\rbrace$
\end{enumerate}
\end{Theorem}
\begin{proof}
The last two items follow immediately from \Cref{prop:subadd-dfunction}~\eqref{prop:subadd-dfunctiona} and \Cref{lem:appWald}~\eqref{lem:appWaldb}, once we prove that $\{\nu(I_s)\}_{s\ge 1}$ is the initial degree sequence of $d$, so we focus on proving this fact.

Let $L:= \min\left\lbrace \sum\limits_{i\in\cU} d_ib_i:b_i\in\ZZ_{\ge 0} \mbox{ for all } i\in\cU \mbox{ and } \sum\limits_{i\in\cU} ib_i=s\right\rbrace$.  Let $\{b_i:i\in\cU\}\subseteq \ZZ_{\ge 0}$ be such that $\sum\limits_{i\in\cU}ib_i=s$.  For each $i\in\cU$, choose $g_iT^i\in\mathcal{G}$ so that $\nu(g_i)=d_i$ (this is possible by the definition of $d_i$).  In other words, $g_i\in I_i$.  Put $g=\prod\limits_{i\in\cU} g_i^{b_i}$.  Since $\nu$ is a valuation, $\nu(g)=\sum\limits_{i\in \cU}b_i\nu(g_i)=\sum\limits_{i\in\cU} b_id_i$.  Also, since $\{I_i\}_{i\ge 1}$ is a graded family, 
\[
g=\prod_{i\in\cU} g_i^{b_i}\in \prod_{i \in \cU} I_i^{b_i}\subset I_{\sum\limits_{i\in \cU} ib_i}=I_s.
\]
It follows that $\nu(I_s)\le \nu(g)= \sum\limits_{i\in\cU} b_id_i$.  Since $\{b_i:i\in\cU\}\subseteq \ZZ_{\ge 0}$ was arbitrary subject to $\sum\limits_{i\in\cU}ib_i=s$, $\nu(I_s)\le L$.

Now suppose that $f\in I_s$, or equivalently $fT^s\in \cR(\mathcal{F})$.  Since $\mathcal{G}$ generates $\cR(\mathcal{F})$ as an $S$-algebra, $fT^s$ can be written as a polynomial in the elements of $\mathcal{G}$ with coefficients in $S$.  That is, 
\begin{equation}\label{eq:polyexp}
fT^s=\sum\limits_{\alpha=(a_1,\ldots,a_k)\in\NN^k} c_\alpha (f_1T^{n_1})^{a_1}\cdots (f_kT^{n_k})^{a_k},
\end{equation}
for some coefficients $c_\alpha\in S$, where only finitely many of $c_\alpha$'s are non-zero.  In fact, since $\mathcal{G}$ consists of homogeneous generators for the graded algebra $\cR(\mathcal{F})$, we may assume that $\sum\limits_{i=1}^k n_ia_i=s$ for each $\alpha=(a_1,\ldots,a_k)$ so that $c_\alpha\neq 0$ in \Cref{eq:polyexp}.  That is, we may assume
\begin{equation}\label{eq:polyexp2}
f=\sum\limits_{\alpha=(a_1,\ldots,a_k)\in\NN^k} c_\alpha f_1^{a_1}\cdots f_k^{a_k},
\end{equation}
where $\sum\limits_{i=1}^k a_in_i=s$ for each non-zero $c_\alpha$.  Since $\nu$ is a valuation, 
\begin{equation}\label{eq:nuineq}
\nu(f)\ge\min\limits_{c_\alpha\neq 0}\{ \nu(c_\alpha f_1^{a_1}\cdots f_k^{a_k})\}=\min\limits_{c_\alpha\neq 0}\left\lbrace\nu(c_\alpha)+\sum\limits_{i=1}^k a_i\nu(f_i)\right\rbrace\ge \min\limits_{c_\alpha\neq 0}\left\lbrace\sum\limits_{i=1}^k a_i\nu(f_i)\right\rbrace\ge \min\limits_{c_\alpha\neq 0}\left\lbrace\sum\limits_{i=1}^k a_id_{n_i}\right\rbrace,
\end{equation}
where the penultimate inequality follows since $\nu$ is an $S$-valuation and the final inequality follows by definition of the function $d:\cU\to\NN$.

Now, for a fixed $\alpha=(a_1,\ldots,a_k)\in\NN^k$ so that $c_\alpha\neq 0$ and each $i\in\cU$, let $b_i=\sum\limits_{n_j=i} a_j$.  Then the equality $\sum\limits_{i=1}^k a_in_i=s$ for every term in \Cref{eq:polyexp2} can be re-written as $\sum\limits_{i\in\cU} ib_i=s$.  Furthermore, the sum $\sum\limits_{i=1}^k a_id_{n_i}$ appearing in the final inequality of \Cref{eq:nuineq} can be re-written as $\sum\limits_{i\in\cU} b_id_i$.  Therefore, there is some choice of non-negative integers $\{b_i:i\in\cU\}$ so that $\sum\limits_{i\in\cU}ib_i=s$ and $\nu(f)\ge \sum\limits_{i\in\cU}b_id_i$.  It follows that $L\le \nu(I_s)$, so we are done.
\end{proof}

The Noetherian Rees algebra relevant to our work is the \textit{symbolic Rees algebra} of a monomial ideal. We utilize \Cref{thm:ReesAlgebraFiltration} in the context of the symbolic Rees algebra derived from the Stanley-Reisner ideal of a matroid in the following subsection. 

\begin{Definition}\label{def:symReesAlgebras}
If $I$ is an ideal in $S$, the symbolic Rees algebra of $I$ is the subalgebra of $S[T]$ (where $T$ is an indeterminate) defined as
\[
\cR_s(I):=\bigoplus_{i\ge 0} I^{(i)}T^i.
\]
\end{Definition}

In general, it is very difficult to determine whether $\cR_s(I)$ is Noetherian, but for monomial ideals, it is known to be Noetherian and coincides with an object called the \textit{vertex cover algebra}: by~\cite{HHT07}, $\cR_s(I)$ is generated as an $S$-algebra by finitely many monomials of the form $x^{{\bf a}_i}T^{s_i}$ (i.e., $x^{{\bf a}_i}\in I^{(s_i)}$).

\subsection{Generators for the symbolic Rees algebra of the Stanley-Reisner ideal of a matroid}

We now turn to the symbolic Rees algebra of the Stanley-Reisner ideal of a matroid.  To convey a complete picture, we start with two additional operations on the matroidal side. 
 Given a matroid $M=(E,\mathcal{B})$ with $\rk(M)=k$ and an integer $0 \le r\le k $, the \textit{truncation of $M$ by rank $r$}, which we write as $\cT^r(M)$, is the matroid of rank $k-r$ whose independent sets are those independent subsets $A\subseteq E$ of $M$ so that $|A|\le k-r$. 
 $\cT^r(M)$ can also be characterized by the rank function $\rk_{\cT^r(M)}(A)=\min\{\rk_M(A),k-r\}$.

For an integer $0 \le r \le \rk(M^*)$, the \textit{elongation of $M$ by rank $r$}, which we write as $\cE^r(M)$, is the matroid of rank $k+r$ on $E$ whose independent sets are those subsets $A\subseteq E$ so that $n_M(A)=|A|-\rk_M(A)\le r$.

 We fix a setting and notation that will be used throughout this section. To make the notation more manageable and suggest the connection to symbolic powers, we define $M^{(r)}:=\cE^{r-1}(M)$.
 
\begin{SettingNotation}\label{set: Notation for section 7}  Let $M=(E,\cB)$ be a rank $k$ matroid  on a ground set $E$ of size $n$, and let \( S=\LL[x_e~:~e\in E] \) be a polynomial ring over a field \( \LL \). Let  $\Delta=\Delta(M)$ be the independent complex associated with $M$. For each $ r \in [n-k]$, \begin{enumerate} 
    \item Let $M^{(r)}=\cE^{r-1}(M)$ be the elongation of $M$ by rank $r-1$. 
    \item     Write $\Delta^{(r)}$ for the independence complex $\Delta(M^{(r)})$ and $I_{\Delta^{(r)}}$ for the Stanley-Reisner ideal of $\Delta^{(r)}$ in $S$.  Note that when $r =1$, $\Delta^{(1)}=\Delta=\Delta(M)$.
    \end{enumerate}
\end{SettingNotation}

Elongation and truncation are related to one another via matroid duality.  The following proposition can be found in~\cite[Observation~1]{LMPS18} (see~\cite{Welsh76} for a standard reference).  Note that our definitions of truncation by rank $r$ and elongation by rank $r$, respectively, are the same as the definitions of $(k-r)$-truncation and $(k+r)$-elongation in~\cite{LMPS18}.

\begin{Proposition}\label{prop:truncationdualelongation}
Let $M$ and $r$ be as in \Cref{set: Notation for section 7}. Then $\cE^r(M)=\cT^r(M^*)^*$.
\end{Proposition}

In the following lemma, we record equivalent descriptions of circuits and dependent sets of $M^{(r)}$.

\begin{Lemma}\label{lem:ElongationBasesAndCircuits}
Let $M$, $r$ and $M^{(r)}$ be as in \Cref{set: Notation for section 7}.  Then we have:
\begin{enumerate} 
\item \label{lem:ElongationBasesAndCircuitsa} The circuits of $M^{(r)}$ are complements of the flats of $M^*$ of rank $\rk(M^*)-r$.
\item \label{lem:ElongationBasesAndCircuitsb} The dependent sets of $M^{(r)}$ are those subsets $A\subseteq E$ so that $|A\cap B|\ge r$ for all $B\in\cB(M^*)$.
\end{enumerate}
\end{Lemma}

\begin{proof} 
For \eqref{lem:ElongationBasesAndCircuitsa}, the circuits of $M^{(r)}=\cE^{r-1}(M)$ are the circuits of $\cT^{r-1}(M^*)^*$ by \Cref{prop:truncationdualelongation}.  In turn, the circuits of $\cT^{r-1}(M^*)^*$ are the same as the complements of the flats of $\cT^{r-1}(M^*)$ of rank $\rk(\cT^{r-1}(M^*))-1$.  By the definition of truncation, the flats $\cT^{r-1}(M^*)$ of rank $\rk(\cT^{r-1}(M^*))-1$ are precisely the flats of $M^*$ of rank $\rk(M^*)-r$. Thus \eqref{lem:ElongationBasesAndCircuitsa} follows.

For \eqref{lem:ElongationBasesAndCircuitsb}, the independent sets of $M^{(r)}=\cE^{r-1}(M)$ are by definition those sets $U$ satisfying that $|U|-\rk_M(U)\le r-1$.  Thus the dependent sets of $M^{(r)}$ are those subsets $A\subset E$ that satisfy $|A|-\rk_M(A)\ge r$.  Observe that $|A\cap B|\ge r$ for all $B\in \cB(M^*)$ if and only if $|A\cap (E\setminus B')|\ge r$ for every $B'\in \cB(M)$ if and only if $|A|=|A\cap B'|+|A\cap(E\setminus B')|\ge |A\cap B'|+r$ for every $B'\in\cB(M)$ if and only if $|A|-|A\cap B'|\ge r$ for every $B'\in\cB(M)$.  But $|A\cap B'|\le \rk_M(A)$.  Thus $|A|-|A\cap B'|\ge r$ for every $B'\in\cB(M)$ if and only if $|A|-\rk_M(A)\ge r$; that is $A$ is a dependent set of $M^{(r)}$.
\end{proof}

In the following lemma, we prove that for $1 \le r \le \rk(M^*)$, the ideal generated by squarefree monomials in $I_{\Delta(M)}^{(r)}$ is the Stanley-Reisner ideal of $\Delta^{(r)}$.

\begin{Lemma}\label{lem:squarefreeMonomialsInSymbolicPowers}
Let $M$ and $r$ be as in \Cref{set: Notation for section 7}. For each  $U\subseteq  E$, the monomial $x^U\in I^{(r)}_{\Delta}$ if and only if $U$ is a dependent set of $M^{(r)}$.  In particular, $I_{\Delta^{(r)}}$ is minimally generated by the squarefree monomials $$\{x^C~:~ C\in \circuits(M^{(r)})\}=\{x^{E\setminus F}~:~ F\in \cL_{\rk(M^*)-r}(M^*)\}.$$
\end{Lemma}
\begin{proof}
Let $U \subseteq E$. From \Cref{lem:ElongationBasesAndCircuits}~\eqref{lem:ElongationBasesAndCircuitsb}, $U$ is a dependent set of $M^{(r)}$ if and only if $|U\cap B|\ge r$ for every $B\in\cB(M^*)$ if and only if $x^U\in P_B^r$ for every $B\in\cB(M^*)$ if and only if $x^U\in I_{\Delta}^{(r)}$, see \Cref{eq:PDsymbolicpowersSR}.
Recall that $P_B:=\langle x_u\, :\, u\in B \rangle$. 

If $U$ is not a minimal dependent set of $M^{(r)}$, then $U\supseteq  C$ for some circuit $C\in \circuits(M^{(r)})$.  Thus  $\{x^C~:~ C\in \circuits(M^{(r)})\}$ is the minimal generating set for $I_{\Delta^{(r)}}$.  The final equality follows immediately from \Cref{lem:ElongationBasesAndCircuits}~\eqref{lem:ElongationBasesAndCircuitsa}.
\end{proof}

The following theorem describes the structure of the symbolic Rees algebra of the Stanley-Reisner ideal of a matroid.
Here we denote the \textit{corank} of a flat $F\in\cL(M)$ by $\crk_{M}(F):=\rk(M)-\rk_M(F)$.

\begin{Theorem}
\label{thm:GeneratorsOfSymbolicReesAlgebraOfMatroid}
Let $M$ and $\Delta$  be as in \Cref{set: Notation for section 7}.
Then, the symbolic Rees algebra $\cR_s(I_{\Delta})$ is generated as an $S$-algebra by the set of monomials $$\{x^CT^i~:~C\in \circuits(M^{(i)}),~ i \in [n-k]\}.$$  Equivalently, $\cR_s(I_{\Delta})$ is generated as an $S$-algebra by $$\{x^{E\setminus F}T^{\crk_{M^*}(F)}~:~F\in \cL(M^*) \setminus \{E\}\}.$$
\end{Theorem}
\begin{proof}
To streamline our presentation, we provide our proof in \Cref{app}.  This is equivalent to a result proved independently by Mantero and Nguyen, see \cite[Theorem~3.7]{ManNgu}.
\end{proof}

As a consequence of \Cref{thm:GeneratorsOfSymbolicReesAlgebraOfMatroid}, we see that $d_r(M)$ -- the $r$-th generalized Hamming weight of $M$ -- is the minimum size of a circuit of the elongation $M^{(r)}=\cE^{r-1}(M)$.  Indeed, we see from the definition of $\cE^{r-1}(M)$ that $A\subseteq  E$ is independent in $\cE^{r-1}(M)$ if and only if $|A|-\rk_{M}(A)\le r-1$.  Consequently, a dependent set $U$ of $\cE^{r-1}(M)$ must satisfy $|U|-\rk_{M}(U)\ge r$.  Thus, by \Cref{cor:equivalentInterpretationsGeneralizedHammingWeights}, we have: \begin{align}\label{eq:flat-matroid-elong}
  d_r(M)=d_1(M^{(r)})=d_1(\cE^{r-1}(M))=n-f_{n-k-r}^{\max}(M^*) \text{ for }  r \in [n-k]  
\end{align}
and 
\begin{align}\label{eq:flat-dual-elong}
  d_r(M^*)=d_1(\cE^{r-1}(M^*))=n-f_{k-r}^{\max}(M) \text{ for }  r \in [k].  
\end{align}

The next result is the culmination of our work in \Cref{sec:subadd-finite-seq} and \Cref{sec:noetherian-rees-algebra}.  Notably, we can treat a matroid and its dual simultaneously when the sequence of generalized Hamming weights forms a \rsubad{} sequence.  In \Cref{sec:pavingMatroids}, we show that this hypothesis on the generalized Hamming weights is satisfied for several interesting classes of matroids.

\begin{Theorem}\label{thm:SR}
    Let $M$, $\Delta$, $n$, and $k$ be as in  \Cref{set: Notation for section 7}.  Let $\{d_i(M)\}_{i=1}^{n-k}$ be the generalized Hamming weights of $M$.  Then we have the following: 
  \begin{enumerate} 
   \item  \label{thm:SRa} For $ i \in [n-k],$ $d_i(M)=\alpha(I_{\Delta^{(i)}})\ge \alpha(I_{\Delta}^{(i)})$. 
    \item \label{thm:SRb} For $ i \in [n-k]$,
$\alpha(I^{(i)}_{\Delta})=d_i(M)$
 if and only if $d_i(M)$ is a subadditive term of the sequence $\{d_i(M)\}_{i=1}^{n-k}$. 
 \item \label{thm:SRc} The Waldschmidt constant of $I_{\Delta(M)}$ is given by 
 \[
\widehat{\alpha}(I_{\Delta})=\min\limits_{i \in [n-k]}\left\lbrace \frac{d_i(M)}{i} \right\rbrace=\min\left\lbrace \frac{n-f_j^{\max}(M^*)}{n-k-j}~:~ 0 \le j \le n-k-1 \right\rbrace.
\] 
 \item \label{thm:SRd} The initial degree of $s$-th symbolic power of $I_\Delta$ is given by 
 \[
\alpha(I_{\Delta}^{(s)})=\min\left\{\sum\limits_{i=1}^{n-k} b_id_i(M)~:~ b_i \in \ZZ_{\ge 0} \text{ for } i \in [n-k] \text{ and } \sum\limits_{i=1}^{n-k} ib_i=s \right\}.
\] 
We may further assume that $b_i>0$ only when $d_i(M)$ is a strictly subadditive term of $\{d_i(M)\}_{i=1}^{n-k}$. 
\item \label{thm:SRe} The following are equivalent: \begin{itemize}
    \item[-] $\{d_i(M)\}_{i=1}^{n-k}$ is a \rsubad{} sequence.
    \item[-] For any $s=(n-k)q+r$ with $r \in [n-k]$ and $q \in \ZZ_{\ge 0}$,  $$\alpha(I_{\Delta}^{(s)})=qd_{n-k}(M)+d_r(M).$$
    \item[-] For any $s=kq+r$ with $r \in [k]$ and $q \in \ZZ_{\ge 0}$,  $$\alpha(I_{\Delta(M^*)}^{(s)})=qd_{k}(M^*)+d_r(M^*).$$
\end{itemize} Moreover, in this case,
\[
\widehat{\alpha}(I_{\Delta})=\frac{d_{n-k}(M)}{n-k}=\frac{n}{n-k},~\widehat{\alpha}(I_{\Delta(M^*)})=\frac{d_{k}(M^*)}{k}=\frac{n}{k},~\mbox{and } \frac{1}{\widehat{\alpha}(I_{\Delta})}+\frac{1}{\widehat{\alpha}(I_{\Delta(M^*)})}=1.
\]
\end{enumerate}
\end{Theorem}

\begin{proof}
The fact that $\alpha(I_{\Delta^{(i)}})=d_i(M)$ follows from \Cref{lem:squarefreeMonomialsInSymbolicPowers} and \Cref{eq:flat-matroid-elong}.  We have $d_i(M)\ge \alpha(I^{(i)}_{\Delta})$ from \Cref{gHW-lb}, establishing part \eqref{thm:SRa}. By \Cref{thm:GeneratorsOfSymbolicReesAlgebraOfMatroid}, we know that the symbolic Rees algebra $\cR_s(I_{\Delta})$ is generated as an $S$-algebra by the set of monomials $\{x^CT^i~:~C\in \circuits(M^{(i)}),~ i \in [n-k]\}.$ Take $\cU =[n-k]$ and $d:\cU \to [n]$ with $d_i=\min\limits_{C \in \circuits(M^{(i)})} \deg(x^C).$ By \Cref{lem:squarefreeMonomialsInSymbolicPowers} and part \eqref{thm:SRa}, we know that $d_i= \alpha(I_{\Delta^{(i)}}) = d_i(M)$.  Now, apply  \Cref{thm:ReesAlgebraFiltration} with   $\{x^CT^i~:~C\in \circuits(M^{(i)}),~ i \in [n-k]\}$ as $S$-algebra generators of the symbolic Rees algebra $\cR_s(I_{\Delta})$ and the initial degree $S$-valuation $\alpha:~S\setminus\{0\} \to \ZZ$ defined by $\alpha(f):=\min\{i~:~f\in\bigoplus_{k\ge i} S_k\}$, we get $\alpha_s(d)=\alpha(I_{\Delta}^{(s)})$ for all $s \in \NN$.  Parts \eqref{thm:SRb}-\eqref{thm:SRd} follow in the same way from \Cref{thm:ReesAlgebraFiltration}.  The first equality in part \eqref{thm:SRe} immediately follows from \Cref{cor:sun&cosub} and the second equivalence follows from the fact that the Wei dual of $d$ is $d^*=\{d_i(M^*)\}_{i=1}^k$ using \Cref{lem:WeiDualityMatroids}, and by switching the role of $M$ to $M^*$, we know that $\{\alpha(I_{\Delta(M^*)}^{(s)})\}_{s \in \NN}$ is the initial degree sequence of $d^*$.  
\end{proof}

\begin{Example}[Complete intersections]\label{ex:completeIntersections}
Suppose $I=\langle m_1,\ldots,m_c\rangle$ is a complete intersection squarefree monomial ideal.  Put $a_i=\deg(m_i)$ for $i \in [c]$ and assume that $a_1\le \cdots \le a_c$.  It is well known that $I^{(s)}=I^s$ for complete intersections, so $\alpha(I^{(s)})=sa_1$ and $\widehat{\alpha}(I)=a_1$.  We now explain how this also follows from \Cref{thm:SR}.  Note that $I$ is the Stanley-Reisner ideal of the graphic matroid $M$ of a union of $c$ disjoint cycles, with a cycle of length $a_i$ for each monomial $m_i$.  The rank of $M$ is $\sum_{i=1}^c a_i-c$. The generalized Hamming weights are $d_r(M)=\sum_{j=1}^r a_i$ for $i\in [r]$.  Observe that $d_r(M)$ is not a strictly subadditive term for any $2\le r\le c$.  So $d_2,\ldots,d_c$ may be disregarded in~\Cref{thm:SR}~\eqref{thm:SRd}, recovering that $\alpha(I^{(s)})=sa_1$ and $\widehat{\alpha}(I)=a_1$.
\end{Example}

We next illustrate \Cref{thm:SR} for an infinite family of self-dual matroids that appears in the appendix of `interesting matroids' in~\cite{Oxley2011}.

\begin{Example}[Theta matroids]\label{ex:thetaMatroids} Let $n \ge 2$ be a positive integer, and let $U=\{u_1,\ldots,u_n\}$ and $V=\{v_1,\ldots,v_n\}$ be disjoint ordered sets. Set $E=U\cup V.$ The {\it $n$-theta matroid}, denoted by $\Theta_n$, is a rank $n$ self-dual matroid on the ground set $E$ with set of basis as follows: $$\cB(\Theta_n)=\left\{B \subset E~:~ |B|=n, \, |B \cap U| \le 2 \text{ and } B \neq (V\setminus \{v_i\} ) \cup \{u_i\} \text{ for each } i \in [n]\right\}.$$ We direct the reader to \cite[page 664]{Oxley2011} for additional properties of Theta matroids. We now compute the generalized Hamming weights of $\Theta_n$ using \Cref{eq:flat-matroid-elong}.
 It follows from the description of the bases that $\rk_{\Theta_n}(U)=2$. Additionally, it follows from \cite[page 664]{Oxley2011} that $U$ is a modular flat, that is, $\rk_{\Theta_n}(F)+\rk_{\Theta_n}(U)=\rk_{\Theta_n}(U\cup F) +\rk_{\Theta_n}(U \cap F)$ for every flat $F$ of $\Theta_n$.  Using the description of the bases and the fact that $U$ is a modular flat of rank $2$, one can get that $f_j^{\max}(\Theta_n)=n+j-2$ for $2 \le j \le n-1$ (take the flat $U\cup V'$ with $V'\subset V, |V'|=j-2$), $f_0^{\max}(\Theta_n)=0$, $f_1^{\max}(\Theta_n)=1$ and $f_n^{\max}(\Theta_n)=2n$.   
 
Next, using \Cref{eq:flat-matroid-elong}, and the fact that $\Theta_n^* \simeq \Theta_n$, we obtain \[d_i(\Theta_n) =2n-f_{2n-n-i}^{\max}(\Theta_n)= 2n-f_{n-i}^{\max}(\Theta_n)=\begin{cases}
    i+2 & \text{ if } i \in [n-2], \\
    2n-1 & \text{ if } i=n-1, \\
    2n & \text{ if } i=n.
\end{cases}\]
Observe that $\{d_i(\Theta_n)\}_{i=1}^n$ is a subadditive sequence if and only if $n \le 4$. In this case, using \Cref{thm:SR}~\eqref{thm:SRe}, we get $\widehat{\alpha}(I_{\Delta(\Theta_n)})=2$, and for any $s=nq+r$ with $ r \in [n]$ and $q \in \ZZ_{\ge 0}$, $\alpha(I_{\Delta(\Theta_n)}^{(s)})=2nq+d_r(\Theta_n).$ 

Now, we assume that $n \ge 5.$ Observe that the sequence $\{d_i(\Theta_n)\}_{i=1}^{n-2}$ is a linear sequence with a non-zero constant, which establishes it as a strictly subadditive sequence. Furthermore, the terms $d_{n-1}(\Theta_n)$ and  $d_{n}(\Theta_n)$ are not strictly subadditive terms of the sequence $\{d_i(\Theta_n)\}_{i=1}^n$. Thus, it follows from \Cref{thm:SR}~(\ref{thm:SRb}, \ref{thm:SRc}) that  \[\alpha(I_{\Delta(\Theta_n)}^{(i)})=i+2=d_i(\Theta_n) \text{ for } i \in [n-2], \text{ and } \widehat{\alpha}(I_{\Delta(\Theta_n)})=\min\limits_{i \in [n]} \left\{ \frac{d_i(\Theta_n)}{i}\right\}= \frac{n}{n-2}.\]  Next, we compute the initial degree sequence of symbolic powers. It follows from \Cref{thm:SR}~\eqref{thm:SRd} and \Cref{def:init-degree-sequence} that $\alpha(I_{\Delta(\Theta_n)}^{(s)})= \alpha_s(d),$ where $d~:~[n-2] \to [n]$ defined as $d_i=d_i(\Theta_n).$  As we noticed above $d$ is a subadditive sequence. The Wei dual $d^*$, by \Cref{def:WeiDual}, given as $d_1^*=n-1$ and $d_2^*=n$, is also a subadditive sequence. Thus, by \Cref{lem:extended-subadd-alpha-formula}, for any $s =q(n-2)+r$ with $r \in [n-2]$ and $q \in \ZZ_{\ge 0}$, $\alpha(I_{\Delta(\Theta_n)}^{(s)})=\alpha_s(d)=qd_{n-2}+d_r=qn+r+2$.
\end{Example}

\Cref{ex:completeIntersections} and \Cref{ex:thetaMatroids} indicate that we can partially extend \Cref{thm:SR}\eqref{thm:SRe}.

\begin{Corollary}
Let $M$, $\Delta$, $n$, and $k$ be as in  \Cref{set: Notation for section 7} and let  $\{d_i(M)\}_{i=1}^{n-k}$  be the sequence of generalized Hamming weights of $M$. 
Suppose there exists $j<n-k$ so that the subsequence $\{d_i(M)\}_{i=1}^{j}$ is subadditive and cosubadditive and each term of $\{d_i(M)\}_{i=j+1}^{n-k}$ is not a strictly subadditive term of $\{d_i(M)\}_{i=1}^{n-k}$.  If $s=qj+r$ where $r\in [j]$ and $q\in\ZZ_{\ge 0}$, we have $\alpha(I_{\Delta}^{(s)})=qd_j(M)+d_r(M)$. In particular,  $\widehat{\alpha}(I_{\Delta})=\frac{d_{j}(M)}{j}$.
 \end{Corollary}
 \begin{proof}
 Observe that $\{\alpha(I^{(s)})\}_{s\ge 1}$ is the initial degree sequence of the function $d|_{[j]}$ (whose image is the sequence $\{d_1(M),\ldots,d_j(M)\}$ and codomain is $[d_j(M)]$) by \Cref{thm:SR}\eqref{thm:SRc}.  Now apply \Cref{lem:extended-subadd-alpha-formula} to $d|_{[j]}$.
 \end{proof}

\section{Matroids with subadditive and cosubadditive generalized Hamming weights}\label{sec:pavingMatroids}

For a given matroid $M$, \Cref{thm:SR} gives the most information concerning the initial degrees of $I^{(s)}_{\Delta(M)}$ and $I^{(s)}_{\Delta(M^*)}$ when the generalized Hamming weights of $M$ form a \rsubad{} sequence.  In this section, we will see that this happens with remarkable frequency.  In particular, we show that the generalized Hamming weights of sparse paving matroids and of perfect matroid designs form a \rsubad{} sequence.  We continue to assume that our matroids have no loops or coloops.

\subsection{Paving and sparse paving matroids}\label{gen-Paving} A matroid is called a \textit{paving} matroid if the size of every circuit is either equal to the rank of the matroid or is one greater than the rank of the matroid (e.g., projective plane matroid).  A matroid is called a \textit{sparse paving} matroid if it is paving and its dual is paving (e.g., uniform matroid, the V\'amos matroid).  It has been conjectured that the proportion of paving matroids on $n$ elements to connected matroids on $n$ elements tends to $1$ as $n$ tends to infinity~\cite[Conjecture~1.6]{MNDW11}.  We see in this subsection that the sequence of generalized Hamming weights of a paving matroid is subadditive and, in addition, a cosubadditive sequence if the matroid is sparse paving.

We first show that the properties of paving and sparse paving can be encoded using the generalized Hamming weights of $M$. Recall that the uniform matroid ${\mathrm U}_{n,k}$ is the rank $k$ matroid on the ground set $E$ of $n$ elements so that every subset of $E$ consisting of $k$ elements is a basis of ${\mathrm U}_{n,k}$ (see \Cref{Ex:MDS-uniform}).  It is straightforward to check that ${\mathrm U}^*_{n,k}={\mathrm U}_{n,n-k}$.

\begin{Proposition}\label{prop:PavingHamming}
Let $M =(E,\cB)$ be a matroid of rank $2\le k\le n-1$, where $n=|E|$.  Then
\begin{enumerate} 
\item \label{prop:PavingHamminga} $M$ is a uniform matroid if and only if $d_1(M)=k+1$.
\item \label{prop:PavingHammingb} $M$ is paving if and only if $d_1(M)\ge k$ if and only if $d_2(M^*)=n-k+2$.
\end{enumerate}    
Thus, for $k \le n-2$, $M$ is sparse paving if and only if $k\le d_1(M)$ and $d_2(M)=k+2$.
\end{Proposition}
\begin{proof}
If $M$ is the uniform matroid ${\mathrm U}_{n,k}$, then the smallest size of a circuit of $M$ is $d_1(M)=k+1$.  Conversely, suppose $d_1(M)=k+1$.  Then, the smallest size of a circuit of $M$ is $k+1$.  Since $\rk(M)=k$, this implies every subset of $E$ of size $k$ is independent in $M$, thus $M={\mathrm U}_{n,k}$.

$M$ is paving, which means every circuit of $M$ has cardinality equal to $k$ or $k+1$.  Since $M$ cannot have a circuit of size larger than $k+1$, this is equivalent to $d_1(M)\ge k$, which happens if and only if every subset of $E$ of size $k-1$ is independent in $M$.  This, in turn, is true if and only if the truncation $\cT^1(M)$ is the uniform matroid ${\mathrm U}_{n,k-1}$.  We have $\cT^1(M)=\cE^1(M^*)^*$ by \Cref{prop:truncationdualelongation}, and so $\cT^1(M)={\mathrm U}_{n,k-1}$ if and only if $\cE^1(M^*)={\mathrm U}_{n,n-k+1}$. Suppose $\cE^1(M^*)={\mathrm U}_{n,n-k+1}$, then by \Cref{eq:flat-dual-elong}, $d_2(M^*)=d_1(\mathcal{E}^1(M^*))=n-k+2$. Conversely, if $d_2(M^*)=n-k+2$, then by \Cref{eq:flat-dual-elong}, $d_1(\mathcal{E}^1(M^*))=n-k+2$, which implies using part $(\rm a)$ that $\cE^1(M^*)={\mathrm U}_{n,n-k+1}$. Hence, the assertion follows. 
\end{proof}

\begin{Remark}\label{rmk:paving}
If $M$ is paving, then $d_2(M^*)=n-k+2$. Therefore, from \Cref{lem:genMindistMatroid}~\eqref{lem:genMindistMatroidc}, it follows that $d_r(M^*)=n-k+r$ for $2\le r\le k$.  In particular, if $M$ is sparse paving and not uniform, then $d_1(M)=k$ and $d_r(M)=k+r$ for $2\le r\le n-k$.
\end{Remark}

The following result provides a characterization of paving matroids for which the generalized Hamming weights form a sequence that is both subadditive and cosubadditive. 
\begin{Theorem}\label{thm:Paving-indegree}
 Let $M=(E,\cB)$ be a paving matroid of rank $2 \le k \le n-1$, where  $n=|E|$.  Then,  
 \begin{enumerate}
     \item \label{thm:Paving-indegreea} $\{d_i(M)\}_{i=1}^{n-k}$ is a subadditive sequence. In particular, 
$\alpha(I_{\Delta(M)}^{(i)})=d_i(M)$ for $r \in [n-k]$.   
\item \label{thm:Paving-indegreeb} $\{d_i(M)\}_{i=1}^{n-k}$ is cosubadditive if and only if $d_1(M^*) \ge \frac{n-k+2}{2}$.
\item \label{thm:Paving-indegreec} Assume that $d_1(M^*) \ge \frac{n-k+2}{2}$. Then, for any $s=(n-k)q+r$ with $r \in [n-k]$ and $q \in \ZZ_{\ge 0}$, 
$\alpha(I_{\Delta(M)}^{(s)})=qd_{n-k}(M)+d_r(M)=qn+d_r(M)$. In particular, $\widehat{\alpha}(I_{\Delta(M)})=\frac{n}{n-k}$.
\item \label{thm:Paving-indegreed} Assume that $d_1(M^*) \ge \frac{n-k+2}{2}$. Then, for any $s=kq+r$ with $r\in [k]$ and $q \in \ZZ_{\ge 0}$, 
$\alpha(I_{\Delta(M^*)}^{(s)})=qd_{k}(M^*)+d_r(M^*)=qn+d_r(M^*)$. In particular, $\widehat{\alpha}(I_{\Delta(M^*)})=\frac{n}{k}$.
 \end{enumerate}
\end{Theorem}

\begin{proof}
 If $M$ is a uniform matroid, then it follows from \Cref{prop:PavingHamming}~\eqref{prop:PavingHamminga} that $d_1(M)=k+1$, and hence by \Cref{lem:genMindistMatroid}, $d_i(M)=k+i$ for each $i\in [n-k]$.  And, if $M$ is paving but not uniform,    since $M$ is loopless, by \Cref{prop:PavingHamming}~\eqref{prop:PavingHammingb}, we know that $d_1(M) = k$. Then, by \Cref{lem:genMindistMatroid}, there exists  $i_0\in [n-k]$ such that $d_i(M) =k+i-1$ for $i \le i_0$ and $d_i(M)=k+i$ for $i_0< i \le n-k$.  Observe that, in both cases, for all $i,j \in [n-k]$ with $i+j \in [n-k]$, $d_{i+j}(M) \le k+i+j \le d_i(M) +k+j-1 \le d_i(M)+d_j(M)$. This shows that $\{d_r(M)\}_{r=1}^{n-k}$ is a  subadditive sequence, and therefore, by \Cref{thm:SR}~\eqref{thm:SRb}, we conclude that $\alpha(I_{\Delta(M)}^{(r)})=d_r(M)$ for $r \in [n- k]$. This completes part $(\rm a)$.

     By \Cref{def:WeiDual} (and \Cref{lem:WeiDualityMatroids}), we know that $\{d_r(M)\}_{r=1}^{n-k}$ is cosubadditive if and only if $\{d_r(M^*)\}_{r=1}^k$ is subadditive. Since $M$ is paving, by \Cref{rmk:paving}, $d_r(M^*)=n-k+r$ for $2 \le r \le k$. Assume that $\{d_r(M^*)\}_{r=1}^k$ is subadditive. Therefore,  $n-k+2=d_2(M^*)\le 2d_1(M^*)$, which implies that $d_1(M^*) \ge \frac{n-k+2}{2}$. Conversely, we assume that $d_1(M^*) \ge \frac{n-k+2}{2}$. Clearly, $d_2(M^*)=n-k+2 \le 2d_1(M^*)$. Now, take any $i>1$ and $j \ge 1$ with $i+j \le k$. Then, $d_{i+j}(M^*)=n-k+i+j=d_i(M^*)+j \le d_i(M^*)+d_j(M^*)$. This proves that $\{d_r(M^*)\}_{r=1}^k$ is subadditive if and only if $d_1(M^*) \ge \frac{n-k+2}{2}$.

     Since $M$ and $M^*$ are loopless, we have $d_1(M)>1$ and $d_{n-k}(M)=n$.  By parts $(\rm a),(\rm b)$, we know that $\{d_r(M)\}_{r=1}^{n-k}$ is subadditive and cosubadditive. Thus, the proof of parts $(\rm c),(\rm d)$ immediately follows from \Cref{thm:SR}~\eqref{thm:SRe}. 
\end{proof}

\begin{Corollary}
\label{cor:sparse-paving}
Let $M=(E,\cB)$ be a sparse paving matroid of rank $2 \le k \le n-2$ where $n=|E|$.  Then, $\{d_i(M)\}_{i=1}^{n-k}$ is \rsubad{}. 
Moreover, in this case, for any $s=(n-k)q+r$ with $r\in [n-k]$ and $q \in \ZZ_{\ge 0}$,
\[
\alpha(I_{\Delta(M)}^{(s)})=nq+d_r(M)=
\begin{cases}
nq+k & r=1 \mbox{ and } M \mbox{ is not uniform }\\
nq+(k+r) & \mbox{otherwise}.
\end{cases}
\]
In particular, $$\widehat{\alpha}(I_{\Delta(M)})=\frac{n}{n-k}.$$
\end{Corollary}

 We close this subsection by discussing matroids arising from Steiner systems, and as a consequence of \Cref{cor:sparse-paving}, we recover \cite[Proposition~3.8]{BFGM21}. {\it{A Steiner system of type $S(t,k,n)$}}, where $t,k,n$ are positive integers satisfying $t< k< n$ (to avoid trivialities), is a collection of subsets $\mathfrak{B}$ of size $k$ (called \textit{blocks}) of a ground set $E$ of size $n$ so that \textit{every} subset of $E$ of size $t$ is contained in a \textit{unique} block.  In analogy with our notation for matroids, we let $\mathcal{S}=(E,\mathfrak{B})$ denote a Steiner system; we will specify the type $S(t,k,n)$ separately.  There is a matroid one can naturally associate with a Steiner system, as follows, see~\cite[Theorem~3.4]{BFGM21} for proof that the collection of sets in the following definition does, in fact, satisfy the basis exchange axiom of a matroid.

\begin{Definition}\label{def:matroidfromSteiner}
If $\mathcal{S}=(E,\mathfrak{B})$ is a Steiner system of type $S(t,k,n)$, then $M(\mathcal{S})=(E,\mathcal{B})$ is a matroid of rank $k$, where $\mathcal{B}$ consists of all subsets of $E$ of size $k$ which are not blocks of $\mathcal{S}$.
\end{Definition}

We now show that the matroid $M(\mathcal{S})=(E,\cB)$ corresponding to a Steiner system is a sparse paving matroid.  We believe this is known, but we include proof for completeness.

\begin{Proposition}\label{prop:steinerSystemsAreSparsePaving}
If $\mathcal{S}=(E,\mathfrak{B})$ is a Steiner system of type $S(t,k,n)$, then the matroid $M(\mathcal{S})=(E,\cB)$ is a sparse paving matroid.  
\end{Proposition}
\begin{proof}
We first prove that every subset $U\subset E$ of size $k-1$ is an independent set of $M(\mathcal{S})$.  Since $t\le k-1$, choose a subset $A\subset U$ of size $t$.  Then, since $\mathcal{S}$ is an $S(t,k,n)$ Steiner system, $A$ is contained in a unique block $T$ of $\mathcal{S}$.  If $U\subset T$, then $T\setminus U=\{i\}$ for some $i\in E$.  Pick any $j\in E$, with $j\neq i$ and  $j\notin U$. Such an element exists as long as $k<n$.  Let $B=U\cup\{j\}$.  Then, $B\neq T$ and $A\subset B$, so $B$ is not a block of $\mathcal{S}$.  Hence, since $|B|=k$, we must have $B\in\mathcal{B}$.  If $U\not\subset T$, then we can pick any $k$-element subset $B$ of $E$ containing $U$.  It contains $A$  but does not equal $T$, hence again $B\in\mathcal{B}$.  Hence, $U$ is an independent set.  It follows that all dependent sets of $M(\mathcal{S})$ have size at least $k$.  Since every set of size $k+1$ is dependent in $M$, the minimal dependent sets of $M(\mathcal{S})$ (i.e., the circuits) all have size $k$ or $k+1$.  Thus, $M(\mathcal{S})$ is paving.

Next, we prove that every subset $U\subset E$ of size $|U|=k+1$ has rank $k$.  That is, $U$ contains a basis of $M$.  Suppose not.  Then, every $k$-subset of $U$ is a block of $\mathcal{S}$.  The intersection $A_1\cap A_2$ of two distinct $k$-subsets of $U$ is a $(k-1)$-element subset.  Pick a $t$-element subset of $A_1\cap A_2$ (possible because $t<k$).  Then $A_1$ and $A_2$ are both blocks of $\mathcal{S}$ that contain this $t$-element subset, contradicting that a $t$-element subset of $E$ is contained in a \textit{unique} block of $\mathcal{S}$.  So $U$ contains a basis of $M(\mathcal{S})$ and hence has rank $k$.

Now we prove that the flats of $M(\mathcal{S})$ of rank $k-1$ all have cardinality $k-1$ or $k$.  We have seen that every $(k+1)$-subset of $E$ has rank $k$.  Thus, every flat of rank $k-1$ has cardinality at most $k$.  Moreover, we have seen that every $(k-1)$-subset is independent.  Thus, a subset of $E$ with rank $k-1$ has a size of either $k-1$ or $k$.  Hence, the flats of rank $k-1$ of $M(\mathcal{S})$ all have cardinality $k-1$ or $k$.  

It follows that the circuits of $M(\mathcal{S})^*$, which are the complements of rank $k-1$ flats (see \Cref{sec:prelims}) of $M(\mathcal{S})$, have cardinality $n-k$ or $n-(k-1)=n-k+1$.  Since $M(\mathcal{S})^*$ has rank $n-k$, this proves that $M(\mathcal{S})^*$ is paving.  Hence, $M(\mathcal{S})$ is sparse paving.
\end{proof}

To recover \cite[Proposition~3.8]{BFGM21}, we apply \Cref{cor:sparse-paving} to $I=I_{\Delta(M)}$, where $M$ is the dual of the matroid $M(S)$ associated to a Steiner system of type $S(t,k,n)$.

\begin{Corollary}$($\cite[Proposition~3.8]{BFGM21}$)$
Let $\mathcal{S}=(E, \mathfrak{B})$ be a Steiner system of type $S(t,k,n)$, $M=M(\mathcal{S})$ be the matroid associated to $\mathcal{S}$ as in \Cref{def:matroidfromSteiner}, and $I=I_{\Delta(M^*)}$ be the Stanley-Reisner ideal of $M^*$. Then
\begin{enumerate}  
    \item $\alpha(I)=d_1(M^*)=n-k$;
    \item $\alpha(I^{(s)})=d_{s}(M^*)=n-k+s$ for $2\le s \le k$;
    \item If $s=qk+r$ with $r\in[k]$ and $q \in \ZZ_{\ge 0}$, then
    \[
    \alpha(I_{\Delta(M^*)}^{(s)})=nq+d_r(M^*)=
\begin{cases}
nq+(n-k) & r=1 \\
nq+(n-k+r)& \mbox{otherwise}.
\end{cases}
    \]
\end{enumerate}
\end{Corollary}

\subsection{Perfect Matroid Designs}\label{gen-PMD}
A matroid of rank $k$ is called a {\em perfect matroid design} if every rank $i$ flat, for all $i\in[k]$, 
 has a uniform cardinality.  Standard examples of perfect matroid designs include uniform matroids, projective geometries, affine geometries, and the matroid structure derived from Steiner systems. It is important to distinguish that the matroids associated with Steiner systems discussed in Subsection~\ref{gen-Paving} differ from those classified as perfect matroid designs derived from Steiner systems. This differentiation will be elucidated further in this subsection. 

Our main result of this section is that the generalized Hamming weights of a perfect matroid design form a \rsubad{} sequence.

 \begin{Theorem}\label{thm:PMD-indegree}
     Let $M=(E,\cB)$ be a perfect matroid design of rank $k\in [n-1]$, where $n=|E|$. Then,
\begin{enumerate} 
    \item \label{thm:PMD-indegreea} $\{d_i(M)\}_{i=1}^{n-k}$ is subadditive and cosubadditive. 
    \item \label{thm:PMD-indegreeb} For any $s=(n-k)q+r$ with $r \in [n-k]$ and $q \in \ZZ_{\ge 0}$, 
$\alpha(I_{\Delta(M)}^{(s)})=qd_{n-k}(M)+d_r(M)=qn+d_r(M)$. In particular, $\widehat{\alpha}(I_{\Delta(M)})=\frac{n}{n-k}$.
\item \label{thm:PMD-indegreec} For any $s=kq+r$ with $r \in [k]$ and $q \in \ZZ_{\ge 0}$, 
$\alpha(I_{\Delta(M^*)}^{(s)})=qd_{k}(M^*)+d_r(M^*)=qn+d_r(M^*)$. In particular, $\widehat{\alpha}(I_{\Delta(M^*)})=\frac{n}{k}$.

\end{enumerate}
     
 \end{Theorem}

 \begin{proof}
     To establish  that $\{d_r(M)\}_{r=1}^{n-k}$ is both subadditive and cosubadditive, we employ \Cref{prop:flat-sub&cosub}. We claim that for all $i\in[k-1]$ we have $f_{i+1}^{\max}(M)-f_i^{\max}(M) \ge f_{i}^{\max}(M)-f_{i-1}^{\max}(M)$.  For any $i\le j\le \ell \in \{0,1,\ldots,k\}$ and given a rank $i$-flat $F$ and a rank $\ell$-flat $G$ with $F \subset G$, we define $t_M(i,j,\ell)$ as the number of rank $j$-flats which contain $F$ and are contained within $G$. It follows from \cite[Theorem 12.5.1]{Welsh76} that $t_M(i,j,\ell)$ is independent of the choice of $F$ and $G$, and also satisfies  $$t_M(i,i+1,j) =\dfrac{t_M(0,1,j)-t_M(0,1,i)}{t_M(0,1,i+1)-t_M(0,1,i)}, \text{ for } 0 \le i \le j \le k.$$ It is pertinent to note that for each $i\in [k]$, $t_M(0,1,i)=\dfrac{f_i^{\max}(M)}{f_1^{\max}(M)}.$ Now let  $i \in [k-1]$ and observe that \begin{align*}
         \dfrac{f_{i+1}^{\max}(M)-f_i^{\max}(M)}{f_{i}^{\max}(M)-f_{i-1}^{\max}(M)}&=\dfrac{f_{i+1}^{\max}(M)-f_{i-1}^{\max}(M)}{f_{i}^{\max}(M)-f_{i-1}^{\max}(M)}-1\\& =\dfrac{t_M(0,1,i+1)-t_M(0,1,i-1)}{t_M(0,1,i)-t_M(0,1,i-1)}-1\\&=t_M(i-1,i,i+1)-1\ge 1,
     \end{align*} where the last inequality follows from the fact that there are at least two flats of rank $i$ that contain a given flat $F$ of rank $i-1$ and are contained in a given flat $G$ of rank $i+1$ so that $F \subset G$. Specifically, if $F$ is a flat of rank $i-1$ and $G$ is a flat of rank $i+1$ so that $F \subset G$, then there exist two elements $x,y \in G \setminus F$ such that $F \cup\{x,y\}$ has rank $i+1$. Consequently, ${cl}(F\cup \{x\})$ and $cl(F \cup \{y\})$ form flats of rank $i$ that both contain $F$ and are contained within $G$. Thus, for each $i\in[k-1]$, we have  $f_{i+1}^{\max}(M)-f_i^{\max}(M) \ge f_{i}^{\max}(M)-f_{i-1}^{\max}(M) .$ Now, using \Cref{prop:flat-sub&cosub}, we conclude that $\{d_i(M)\}_{i=1}^{n-k}$ is both subadditive and cosubadditive. This completes the proof of part \eqref{thm:PMD-indegreea}. The proof of parts \eqref{thm:PMD-indegreeb} and \eqref{thm:PMD-indegreec} immediately follows from~\Cref{thm:SR}~\eqref{thm:SRe}.
 \end{proof}

We now consider the three aforementioned families of perfect matroid designs: projective geometries, affine geometries, and perfect matroid designs arising from Steiner systems. 

\begin{Example}[Projective geometries]\label{ex:projectiveGeometries}
Let $\KK=\mathbb{F}_q$ be a finite field of order $q$, where $q$ is a power of a prime.  Given an integer $m\ge 0$, there is a matroid of rank $m+1$, derived from the projective space $\mathbb{P}^{m}_{\KK}$, called the \textit{projective geometry} $PG(m,\KK)$.  See~\cite[Chapter~6.1]{Oxley2011} for a detailed discussion.  The ground set of the projective geometry $PG(m,\KK)$ is the set of all $(q^{m+1}-1)/(q-1)$ points of $\mathbb{P}^m_{\KK}$.  If $X\subset \mathbb{P}^m_{\KK}$ is a set of points, we can lift them to a set of vectors $\hat{X}\subset \KK^{m+1}$ (recall that $\mathbb{P}^m_{\KK}$ can be defined as the set of lines through the origin in $\KK^{m+1}$). The set $X\subset \mathbb{P}^m_{\KK}$ is regarded as independent in the matroid $PG(m,\KK)$ if and only if the direction vectors of the lines $\hat{X}$ are linearly independent.  For all $i\in[m+1]$ a flat of rank $i$ in $PG(m,\KK)$ corresponds to a  linear subspace of $\KK^{m+1}$ of dimension $i$, and thus is isomorphic to the projective geometry $PG(i-1,\KK)$.  It follows that a flat of rank $i$ in $PG(m,\KK)$ has size $(q^{i}-1)/(q-1)$, for all $i\in[m+1]$.  Thus, $PG(m,\KK)$ is a perfect matroid design.  By~\Cref{thm:PMD-indegree},
\[
\widehat{\alpha}(I_{\Delta(PG(m,\KK))})=\dfrac{q^{m+1}-1}{q^{m+1}-1-(m+1)(q-1)} \quad\mbox{and}\quad \widehat{\alpha}(I_{\Delta(PG(m,\KK)^*)})=\dfrac{q^{m+1}-1}{(m+1)(q-1)}.
\]
From \Cref{eq:flat-dual-elong}, we have $d_r(PG(m,\KK)^*)=\dfrac{q^{m-r+1}(q^r-1)}{q-1}$ for $1\le r\le m+1$.  Thus, if $s\in \NN$ satisfies $s=s'(m+1)+r$ for $r\in [m+1]$ and $s'\in\ZZ_{\ge 0}$, then \Cref{thm:PMD-indegree}~\eqref{thm:PMD-indegreec} yields
\[
\alpha(I_{\Delta(PG(m,\KK)^*)}^{(s)})=s'\dfrac{q^{m+1}-1}{q-1}+\dfrac{q^{m-r+1}(q^r-1)}{q-1}.
\]
Using \Cref{thm:PMD-indegree}~\eqref{thm:PMD-indegreeb}  (and Lemma~\ref{lem:WeiDualityMatroids}), the initial degree of $I^{(s)}_{\Delta(PG(m,\KK))}$ can also be determined in terms of the Wei dual of the sequence $\{d_r(PG(m,\KK)^*)\}_{r=1}^{m+1}$, which we now describe.

We know using \Cref{lem:WeiDualityMatroids} and \Cref{def:WeiDual} that $\{d_r(PG(m,\KK))\}_{r=1}^{\frac{q^{m+1}-1}{q-1}-m-1}$ is the gap function of $\{\overline{d}_r(PG(m,\KK)^*)\}_{r=1}^{m+1}$ in 
$[\frac{q^{m+1}-1}{q-1}]$. Notice that $$\left[\frac{q^{m+1}-1}{q-1}\right]=\bigsqcup\limits_{\ell =1}^m \left\{ \frac{q^{\ell}-1}{q-1}-\ell+1, \ldots, \frac{q^{\ell+1}-1}{q-1}-\ell-1\right\},$$ that is, for each $r$ between $1$ and $\frac{q^{m+1}-1}{q-1}-m-1$, there exists a unique $\ell \in [m]$ so that $\frac{q^{\ell}-1}{q-1}-\ell+1 \le r \le \frac{q^{\ell+1}-1}{q-1}-\ell -1$. 
It follows using \Cref{def:bar-function} that $\overline{d}_r(PG(m,\KK)^*)=\frac{q^{r-1}-1}{q-1}+1$ for $ r \in[m+1].$ Consequently, for each $r$ between $1$ and $\frac{q^{m+1}-1}{q-1}-m-1$, there exists a unique $\ell \in [m]$ so that $\overline{d}_{\ell +1}(PG(m,\KK)^*)-(\ell+1)+1 \le r \le \overline{d}_{\ell +2}(PG(m,\KK)^*)-(\ell+2).$ Thus, using \Cref{lem:gap-formula}, we obtain $d_r(PG(m,\KK))=r+\ell +1,$ where $\ell \in [m]$ is the unique number so that $\frac{q^{\ell}-1}{q-1}-\ell+1 \le r \le \frac{q^{\ell+1}-1}{q-1}-\ell -1$. 
\end{Example}

\begin{Example}[Affine geometries]\label{ex:affineGeometries}
Let $\KK=\mathbb{F}_q$ be a finite field of order $q$, where $q$ is a power of a prime.  Given a natural number $m\ge 1$, the \textit{affine geometry} $AG(m,\KK)$ is a rank $m+1$ matroid obtained from the projective geometry $PG(m,\KK)$ by removing a hyperplane of $PG(m,\KK)$.  See~\cite[Chapter~6.2]{Oxley2011} for a detailed discussion.  Explicitly, the ground set of $AG(m,\KK)$ consists of all $q^m$ points in $\KK^m$.  We regard $x_1,\ldots,x_m$ as the coordinates on $\KK^m$.  For a point $x=(x_1,\ldots,x_m)\in \KK^m$ we define the lift $\hat{x}$ of $x$ to $\KK^{m+1}$ (with coordinates $x_0,\ldots,x_m$) as the vector $\hat{x}:=(1,x_1,\ldots,x_m)$.  If $X\subset \KK^m$ is a collection of points, then $\hat{X}\subset \KK^{m+1}$ is the collection of vectors in $\KK^{m+1}$ obtained by lifting every point of $X$.  The subset $X$ is independent in the matroid $AG(m,\KK)$ if $\hat{X}$ is a linearly independent set of vectors.  The flats of rank $i$ ($1\le i\le m+1$) of $AG(m,\KK)$ correspond exactly to affine linear subspaces of $\KK^m$ of dimension $i-1$ (equivalently linear subspaces of $\KK^{m+1}$ which intersect non-trivially with $x_0=1$).  Thus, a flat of rank $i$ in $AG(m,\KK)$ is isomorphic to $AG(i-1,\KK)$ and hence has $q^{i-1}$ elements for $1\le i\le m+1$.  Evidently, $AG(m,\KK)$ is a perfect matroid design.  Hence, \Cref{thm:PMD-indegree} yields
\[
\widehat{\alpha}(I_{\Delta(AG(m,\KK))})=\dfrac{q^m}{q^m-(m+1)} \quad\mbox{and}\quad \widehat{\alpha}(I_{\Delta(AG(m,\KK)^*)})=\dfrac{q^m}{m+1}.
\]
From \Cref{eq:flat-dual-elong}, we have $d_r(AG(m,\KK)^*)=q^m-q^{m-r}$ for $r\in [m]$ and  $d_{m+1}(AG(m,\KK)^*)=q^m$.  Thus, if $s\in \NN$ satisfies $s=s'(m+1)+r$ for $r\in [m+1]$ and $s'\in\ZZ_{\ge 0}$, \Cref{thm:PMD-indegree}~\eqref{thm:PMD-indegreec} yields
\[
\alpha(I_{\Delta(AG(m,\KK)^*)}^{(s)})=s'q^m+d_r(AG(m,\KK)^*)=
\begin{cases}
(s'+1)q^m & r=m+1 \\
(s'+1)q^m-q^{m-r} &   \mbox{otherwise}.
\end{cases}
\]
Using \Cref{thm:PMD-indegree}~\eqref{thm:PMD-indegreeb}  (and Lemma~\ref{lem:WeiDualityMatroids}), the initial degree of $I^{(s)}_{\Delta(AG(m,\KK))}$ can also be determined in terms of the Wei dual of the sequence $\{d_r(AG(m,\KK))\}_{r=1}^{m+1}$, which can be described via a computation similar to the one in \Cref{ex:projectiveGeometries}.
\end{Example}

\begin{Example}[Steiner systems, II]\label{ex:SteinerSystemPMD}
Let $\mathcal{S}=(E,\mathfrak{B})$ be a Steiner system of type $S(t,k,n)$. Let  $\cB$ be those $t+1$ subsets of $E$ that are not contained in any block of $\mathcal{S}$. Then $\cB$ satisfies the basis exchange property and, hence, gives rise to a matroid $M$ of rank $t+1$ on the ground set $E$ of size $n$. The matroid arising this way is known as the {\it matroid design of the Steiner system} $\mathcal{S}$, and it is a perfect matroid design (also a paving matroid) as its hyperplanes are the blocks of the Steiner system $\mathcal{S}$, and for each $0 \le i \le t-1$, the rank $i$ flats are all size $i$ subsets of $E$. See \cite[Chapter 12]{Welsh76} for a detailed discussion. Generally, this matroid $M$ differs from the matroid that appears in \Cref{def:matroidfromSteiner}. In fact, these two matroids match if and only if $k=t+1.$ Let $M$ be a matroid design of a Steiner system $\mathcal{S}$ of type $S(t,k,n).$ Then, by \Cref{thm:PMD-indegree}, the generalized Hamming weights of $M$ and $M^*$ both are subadditive, and hence, \Cref{thm:PMD-indegree} yields that \[\widehat{\alpha}(I_{\Delta(M)})=\dfrac{n}{n-t-1} \text{ and } \widehat{\alpha}(I_{\Delta(M^*)})=\dfrac{n}{t+1}.\] From \Cref{eq:flat-dual-elong}, we have $d_1(M^*)=n-k$ and $d_r(M^*)=n-t+r-1$ for $2 \le r \le t+1.$ Thus, if $s=q(t+1)+r$ for some $q \in \ZZ_{\ge 0} $ and $r \in [t+1]$, then \Cref{thm:PMD-indegree}~\eqref{thm:PMD-indegreec} yields that \[
    \alpha(I_{\Delta(M^*)}^{(s)})=nq+d_r(M^*)=
\begin{cases}
nq+n-k &  r=1 \\
nq+n-t+r-1 & \mbox{otherwise}.
\end{cases}
    \] Using \Cref{lem:WeiDualityMatroids}, the sequence $\{d_r(M)\}_{r=1}^{n-t-1}$ is the Wei dual of the sequence $\{d_r(M^*)\}_{r=1}^{t+1}$, which is given as follows \[ 
    d_r(M)=
\begin{cases}
t+r &   \text{when } 1 \le r \le k-t \\
t+r+1 &  \text{when } k-t+1 \le r \le n-t-1.
\end{cases} \] Thus, if $s =q(n-t-1)+r$ for some $q \in \ZZ_{\ge 0} $ and $r \in [n-t-1]$, then \Cref{thm:PMD-indegree}~\eqref{thm:PMD-indegreeb} yields that  \[
    \alpha(I_{\Delta(M)}^{(s)})=nq+d_r(M)=
\begin{cases}
nq+t+r &  \text{when } 1 \le r \le k-t \\
nq+t+r+1 & \text{when } k-t+1 \le r \le n-t-1.
\end{cases}\]
    
\end{Example}

\section{Codes with subadditive and cosubadditive generalized Hamming weights}\label{sec: gen-CT}

In this section, we make connections between the coding theory literature and the previous section. We identify several classes of linear codes whose generalized Hamming weights form a \rsubad{} sequence.  
Recall that, for a code $\C$ with matroid $M=M(\C)$, the $r$-th generalized Hamming weight of $\C$ satisfies $d_r(\C)=d_r(M^*)=d_r(M(\C^\perp))$ (see \Cref{rem:HammingWeightsOfMatroidsVsLinearCodes}).  Throughout this section, we will assume that our codes never have a column of zeros in their generator matrix or parity check matrix, equivalently $M(\C)$ has no loops or coloops (in terms of generalized Hamming weights, $d_1(\C)\ge 2$ and $d_k(\C)=n$).

We begin with maximum distance separable (or MDS) codes, which made an appearance in \Cref{Ex:MDS-uniform}.

\begin{Example}[MDS codes]
An $[n,k]$-code $\C$ is MDS if its parity check matroid $M(\C^\perp)=M(\C)^*$ is the uniform matroid ${\mathrm{U}}_{n,n-k}$.  Equivalently, via \Cref{prop:PavingHamming}~\eqref{prop:PavingHamminga}, $\C$ is MDS if and only if $d_1(\C)=d_1(M(\C^\perp))=n-k+1$.  Since $\C$ is MDS if and only if $\C^\perp$ is MDS, it follows from \Cref{cor:sparse-paving} that $\widehat{\alpha}(I_{\Delta(M(\C)^*)})=\dfrac{n}{k}$ and $\alpha(I_{\Delta(M(\C)^*)}^{(s)})=qn+(n-k+r)$, where $q\in\ZZ_{\ge 0}$, $r\in [k]$, and $s=qk+r$, for all $s\in \NN$.
\end{Example}

\begin{Example}[Near MDS and almost MDS codes]\label{rem:lincod_paving}
The notion of a \textit{near} MDS code was introduced in~\cite{DL95}; an $[n,k]$-code $\C$ is \textit{near} MDS if $d_1(\C)=n-k$ and $d_r(\C)=n-k+r$ for $2\le r\le n-k$.  Around the same time, \textit{almost} MDS codes were introduced in~\cite{DeBoer96, DeBoerThesis97}.  An $[n,k]$-code $\C$ is almost MDS if $d_1(\C)=n-k$.  Both of these notions have received considerable attention in the literature over the past several decades.  From \Cref{prop:PavingHamming}~\eqref{prop:PavingHammingb}, a code is
\begin{itemize}
\item[-] Near MDS if and only if it is not MDS and its parity check matroid is sparse paving.
\item[-] Almost MDS if and only if its parity check matroid is paving.
\end{itemize}
By \Cref{thm:Paving-indegree} and \Cref{cor:sparse-paving}, the generalized Hamming weights of $\C$ form a \rsubad{} sequence if $\C$ is near MDS or if $\C$ is almost MDS and $d_1(\C^\perp)\ge \frac{k+2}{2}$.
\end{Example}

Next, we consider Hamming codes and simplex codes. Our reference is~\cite[Section~1.8]{Huffman-Pless-2003}, and we use their notation.
\begin{Example}[Hamming codes and simplex codes]\label{ex:HammingAndSimplex}
The \textit{Hamming code} $\mathcal{H}_{q,m}$ is defined over the field $\KK=\mathbb{F}_q$ and has a parity check matrix $H_{q,m}$ whose columns correspond to the points of the projective geometry $PG(m-1,q)$ which appeared in \Cref{ex:projectiveGeometries}.  It follows that the generalized Hamming weights of $\mathcal{H}_{q,m}$ coincide with those of the projective geometry $PG(m-1,q)$.

The dual $\mathcal{H}_{q,m}^{\perp}$ of the Hamming code $\mathcal{H}_{q,m}$ is an $m$-dimensional \textit{simplex code} over $\mathbb{F}_q$; evidently, the generator matrix for $\mathcal{H}_{q,m}^{\perp}$ is $H_{q,m}$.  It follows that the generalized Hamming weights of $\mathcal{H}_{q,m}^{\perp}$ coincide with those of the matroid $PG(m-1,\mathbb{F}_q)^*$.  By \Cref{ex:projectiveGeometries}, $d_r(\mathcal{H}_{q,m}^{\perp})=\frac{q^{m-r}(q^r-1)}{q-1}$ for $r\in [m]$.  Clearly the initial degrees of $I^{(r)}_{\Delta(M(\mathcal{H}_{q,m}^{\perp}))}$ can be calculated as in \Cref{ex:projectiveGeometries}.
\end{Example}

It is known that simplex codes are both \textit{Griesmer codes} (see~\cite[Example~7.10.13]{Huffman-Pless-2003}) and \textit{constant weight codes} (see~\cite[Theorem~1.8.3]{Huffman-Pless-2003}).  These are our next two examples.

\begin{Example}[Griesmer codes]\label{ex:Griesmer}
Let $\C$ be an $[n,k,d]$-code over $\KK=\mathbb{F}_q$ with $k\ge 1$.  The \textit{Griesmer bound} states that
$n\ge \sum\limits_{i=0}^{k-1} \left\lceil\dfrac{d}{q^i}\right\rceil$ (see~\cite{Griesmer-1960} for the binary case and~\cite{Solomon-Stiffler-1965} for arbitrary finite fields). 
If equality is obtained in this bound, then $\C$ is called a \textit{Griesmer code}.  Due to their optimality, constructing Griesmer codes is quite an active area of research in coding theory.

Suppose that $\C$ is a Griesmer code with matroid $M=M(\C)$ so that $M$ has no loops or coloops (that is, neither the generator matrix for $\C$ nor the parity check matrix has a column of zeros).  It follows from~\cite[Theorem~7.10.12]{Huffman-Pless-2003} that
$
d_r(\C)=\sum\limits_{i=0}^{r-1}\left\lceil \dfrac{d}{q^i}\right\rceil.
$
We prove that $\{d_r(\C)\}_{r=1}^k$ is a \rsubad{} sequence.  To prove subadditivity, suppose that $i,j\in [k]$ so that $i+j\in[k]$.   Then
\[
d_i(\C)+d_j(\C)=
\sum_{s=0}^{i-1} \left\lceil \dfrac{d}{q^s}\right\rceil + \sum_{s=0}^{j-1} \left\lceil \dfrac{d}{q^s}\right\rceil 
\ge \sum_{s=0}^{i-1} \left\lceil \dfrac{d}{q^s}\right\rceil + \sum_{s=i}^{i+j-1} \left\lceil \dfrac{d}{q^s}\right\rceil
= \sum_{s=0}^{i+j-1} \left\lceil \dfrac{d}{q^s}\right\rceil= d_{i+j}(\C),
\]
where the inequality follows from the term-wise comparison $\left\lceil \dfrac{d}{q^s}\right\rceil\ge \left\lceil \dfrac{d}{q^{s+i}}\right\rceil$ for $0\le s\le j-1$.  Now we prove that $\{d_r(\C)\}_{r=1}^k$ is cosubadditive.  Let $r,a,b\in [k]$ so that $r+k=a+b$.  Observe that $r \le a, b.$  Then
\begin{align*}
d_a(\C)+d_b(\C)=\sum_{s=0}^{a-1} \left\lceil \dfrac{d}{q^s}\right\rceil + \sum_{s=0}^{b-1} \left\lceil \dfrac{d}{q^s}\right\rceil = & \sum_{s=0}^{r-1} \left\lceil \dfrac{d}{q^s}\right\rceil + \sum_{s=r}^{a-1} \left\lceil \dfrac{d}{q^s}\right\rceil + \sum_{s=0}^{b-1} \left\lceil \dfrac{d}{q^s}\right\rceil\\
\ge & \sum_{s=0}^{r-1} \left\lceil \dfrac{d}{q^s}\right\rceil + \sum_{s=b}^{a+b-r-1} \left\lceil \dfrac{d}{q^s}\right\rceil + \sum_{s=0}^{b-1} \left\lceil \dfrac{d}{q^s}\right\rceil\\
= & \sum_{s=0}^{r-1} \left\lceil \dfrac{d}{q^s}\right\rceil + \sum_{s=0}^{k-1} \left\lceil \dfrac{d}{q^s}\right\rceil= d_r(\C)+d_k(\C).
\end{align*}
By \Cref{prop:d-sub}\eqref{prop:d-subb}, $\{d_r(\C)\}_{r=1}^k$ is cosubadditive.  It follows from \Cref{thm:SR} that
$
\widehat{\alpha}(I_{\Delta(M^*)})=\frac{n}{k},
$
$
\widehat{\alpha}(I_{\Delta(M)})=\frac{n}{n-k},
$
and $\alpha(I_{\Delta(M^*)}^{(s)})=s'd_k(\C)+d_r(\C)$ for $s'\in \ZZ_{\ge 0}, r\in [k]$ so that $s=s'k+r$.
\end{Example}

\begin{Example}[Constant weight codes]\label{Ex:constant-weight}
An $[n,k,d]$-code is a \textit{constant weight code} if every non-zero codeword has the same weight.  The simplex codes in Example~\ref{ex:HammingAndSimplex} are constant weight, and in fact, Bonisoli characterizes all constant weight codes as replications of simplex codes~\cite{Bonisoli-Constant-Weight} (see also~\cite[Theorem~7.9.5]{Huffman-Pless-2003}).  It is a consequence of \cite[Theorem~1]{Liu-Chen-Notes-Value-Function-2010} (see also \cite[Corollary 2.1]{JV14}) that
$
d_i(\C)= \frac{d(q^i-1)}{q^{i-1}(q-1)} \text{ for all } i\in[k].
$
In fact, by \cite[Corollary 2.2]{JV14}, we have the converse statement: any linear code with generalized Hamming weights given by these formulas is a code of constant weight $d$.

Constant weight codes fall into two categories that we have already seen.  First, the matroid of a constant weight code is a perfect matroid design (see \Cref{gen-PMD}).  One can deduce this from the fact that $I_{\Delta(M^*)}$ has a \textit{pure resolution} by~\cite[Corollary~3.1]{JV14}, and the fact that if $I_{\Delta(M^*)}$ has a pure resolution then $M$ is a perfect matroid design~\cite[Theorem~3.1.7]{Armenoff-thesis-2015}.  Second, a constant weight code is a Griesmer code (see \Cref{ex:Griesmer}); we readily verify that $n=d_k(\C)=\dfrac{d(q^k-1)}{q^{k-1}(q-1)}=\sum\limits_{s=0}^{k-1} \dfrac{d}{q^s}$ satisfies the Griesmer bound with equality.

It now follows from \Cref{thm:PMD-indegree} or \Cref{ex:Griesmer} that, for any $s =s'k+r$ with $s'\in \ZZ_{\ge 0}$ and $r \in [k]$,
\[
\alpha(I_{\Delta(M^*)}^{(s)})=s'd_k(\C)+d_r(\C),\quad \widehat{\alpha}(I_{\Delta(M^*)})= \dfrac{d(q^k-1)}{kq^{k-1}(q-1)}, \quad\mbox{and} \quad \widehat{\alpha}(I_{\Delta(M)})=\dfrac{d(q^k-1)}{d(q^k-1)-kq^{k-1}(q-1)}.
\]
\end{Example}

We next consider perfect codes, which are those that attain equality in \textit{sphere packing bound}~\cite[Section~1.12]{Huffman-Pless-2003}.

\begin{Example}[Perfect Codes]
Suppose $\C$ is an $[n,k,d]$-code in $\mathbb{F}^n_q$ and put $t=\lfloor (d-1)/2\rfloor$.  The code $\C$ is called a \textit{perfect code} if every vector in $\mathbb{F}^n_q$ is contained in exactly one sphere of radius $t$ centered on a codeword of $\C$.  Remarkably, there is a complete classification of perfect codes~\cite[Theorem~1.12.3]{Huffman-Pless-2003}.  We consider this classification only for linear perfect codes; the classification also extends to non-linear codes.

The only linear codes of minimum distance three that are perfect are the Hamming codes $\mathcal{H}_{q,m}$ discussed in \Cref{ex:HammingAndSimplex}.  Aside from the Hamming codes, the only linear perfect codes are the $[23,12,7]$ binary Golay code and the $[11,6,5]$ ternary Golay code~\cite[Sections~10.1, 10.4]{Huffman-Pless-2003}.  

The binary and ternary Golay codes can be constructed by \textit{puncturing} the so-called \textit{extended} binary and ternary Golay codes, respectively~\cite[Section~1.9]{Huffman-Pless-2003}.  It is readily checked that the $[11,6,5]$ ternary Golay code along with the $[12,6,6]$ \textit{extended} ternary Golay code are both Griesmer codes, hence their generalized Hamming weights are subadditive and cosubadditive by \Cref{ex:Griesmer}.

By \cite[Theorem~6]{We}, the generalized Hamming weights of the $[24,12,8]$ extended binary Golay code are
$\{8,12,14,15,16,18,19,$  $20,21,22,23,24\}$.  It is readily checked that this is a subadditive sequence.  Since the $[24,12,8]$ extended binary Golay code is self-dual, the generalized Hamming weights are also cosubadditive.

We could not compute or find a reference for the generalized Hamming weights of the $[23,12,7]$ binary Golay code, but we can argue that they must form a \rsubad{} sequence.  Let $\{d_i\}_{i=1}^{12}$ be the generalized Hamming weights of the extended binary Golay code (see above).  When we puncture a code, the number of generalized Hamming weights stays the same, and their value drops by at most one.  Let $\{\delta_i\}_{i=1}^{12}$ be the generalized Hamming weights of the binary Golay code.  It is well-known that $\delta_1=7$.  By \Cref{lem:genMindistMatroid}~\eqref{lem:genMindistMatroidb} and the fact that $d_j -1 \le \delta_j $, we know that $\delta_j=23-12+j=11+j$ for $6\le j\le 12$.  For $2\le i\le 5$, $\delta_i=d_i-1+\epsilon_i$, for some  $\epsilon_i \in \{0,1\}$.

We argue that $\{\delta_i\}_{i=1}^{12}$ is \rsubad{}.  For subadditivity, let $i,j\in [12]$ so that $i+j\in [12]$. Then, using \Cref{lem:genMindistMatroid}~\eqref{lem:genMindistMatroida}-\eqref{lem:genMindistMatroidb}, $\delta_{i+j} \le 11+i+j =(6+i)+(5+j) \le (d_i-1)+(d_j-1) \le \delta_i +\delta_j.$

For cosubadditivity, we use \Cref{thm:d^*subadd}.  Let $r,a,b\in [12]$ so that $r+12=a+b$ and $a \ge b$. Note that $a\ge b>r$ and $a \ge  \left \lceil \frac{12+r}{2} \right \rceil  \ge 7.$  If $\delta_r=23-12+r$, then by \Cref{lem:genMindistMatroid}~\eqref{lem:genMindistMatroidc}, $\delta_i=23-12+i=11+i$, for all $i\geq r$, and hence, $\delta_r+\delta_{12}=(11+r)+(11+12)=(11 +a)+(11+b)=\delta_{a}+\delta_{b}$. So, assume that $\delta_r\leq 11+r-1=10+r=10+(a+b-12)=a+b-2$. This implies  $\delta_r+\delta_{12}\leq a+b+21=(11+a)+(10+b)=\delta_a+(10+b) \le \delta_a+\delta_b$ if $b \ge 3$. In the case when $b=2$, $r$ has to be one, and $a$ has to be $11$. Thus, $\delta_1+\delta_{12}=7+23=30=22+8\le \delta_{11}+\delta_2.$ Hence, $\{\delta_i\}_{i=1}^{12}$ is \rsubad{}.
\end{Example}

Next, we consider the classic cases of affine and projective Reed-Muller codes. 

\begin{Example}[first-order affine Reed-Muller codes]\label{ex:RMcodes}
Let $\mathbb K=\mathbb F_q$ be the field with $q$ elements. Let $m$ be a positive integer and let $V:=\mathbb K^m$. Set  $n:=|V|=q^m$. Let $P_1,\ldots, P_n$ be all the elements of $V$, listed as column vectors in some given order. Let $a$ be a positive integer, let $A:=\mathbb K[y_1,\ldots,y_m]$ be the ring of polynomials in $m$ variables with coefficients in $\mathbb K$, and let $A_{\le a}$ be the $\mathbb K$-vector space of all polynomials in $A$ of degrees $\le a$. Let $\phi:A_{\le a}\rightarrow \mathbb K^n$ be the $\mathbb K$-linear map given by $\phi(f)=(f(P_1),\ldots,f(P_n))$. The image of $\phi$ is a linear code, called {\em the (affine) $q$-ary Reed-Muller code of order $a$}, denoted $\mathcal{RM}_q(a,m)$. 

If $a<q$, then the dimension of this linear code is $\displaystyle \binom{m+a}{a}$, but if $a\ge q$, the formula for the dimension is quite challenging; see \cite{KLP}.  When $a=1$, $\{1,y_1,\ldots,y_m\}$ is a basis for $A_{\le 1}$ and so $\mathcal{RM}_q(1,m)$ is an $[n,m+1]$-linear code with a generator matrix
\begin{equation}\label{eq:genmatrixReedMuller}
\left[\begin{array}{cccc}
    1&1&\cdots&1  \\
    P_1&P_2&\cdots&P_n 
    \end{array}\right].
\end{equation}
Observe that the matroid $M=M(\mathcal{RM}_q(1,m))$ is thus the affine geometry $AG(m,\KK)$ in \Cref{ex:affineGeometries}.  The generalized Hamming weights of $\mathcal{RM}_q(1,m)$ coincide with those of $AG(m,\KK)^*$, and so the initial degree statistics of both $I_{\Delta(M)}$ and $I_{\Delta(M^*)}$ follow from \Cref{ex:affineGeometries}.
\end{Example}

\begin{Remark}\label{rem:GHWAffineReedMuller}
The formula for generalized Hamming weights of affine Reed-Muller codes of any order is provided in \cite[Theorem 5.10]{HP}, which we recall here. First, order $P_1,\ldots, P_n$ in increasing lexicographic order, that is $(a_1,\ldots,a_m)^T\prec_{{\rm Lex}} (b_1,\ldots,b_m)^T$ if and only if $a_1=b_1,\ldots, a_{l-1}=b_{l-1}$ and $a_l<b_l$, for some $l\in\{1,\ldots,m\}$. Let $1\le r\le \dim(\mathcal{RM}_q(a,m))$, and let $(a_1,\ldots,a_m)^T$ be the $r$-th element of $\{P_1,\ldots, P_n\}$ in the lexicographic order with the property $a_1+\cdots+a_m>(q-1)m-a-1$. Then, the $r$-th generalized Hamming weight of this linear code is:
\[
d_r(\mathcal{RM}_q(a,m))=\sum\limits_{i=1}^ma_{m-i+1}q^{i-1}+1.
\]
In~\Cref{que: GHWReedMuller}, we ask whether this sequence is subadditive or cosubadditive.
\end{Remark}

\begin{Example}[First-order projective Reed Muller codes]
The affine Reed-Muller codes in \Cref{ex:RMcodes} have a projective analog, called {\em projective Reed-Muller codes}  (see~\cite{Lachaud-1986, Sorensen-1991}).

In general, we consider $\mathbb P(V)=\{Q_1,\ldots,Q_s\}$, where $\mathbb K=\mathbb F_q$ is the field with $q$ elements and $V=\mathbb K^m$ is as in the previous examples, but $Q_i$ are now the projective $\mathbb K$-rational points in $\mathbb P^{m-1}$. This implies that $s=(q^m-1)/(q-1)$. For $i=1,\ldots,s$, let $P_i$ be the transpose of the standard representative of $Q_i$; by ``standard representative,'' one understands the vector whose first nonzero entry equals 1.

Let $A=\mathbb K[y_1,\ldots,y_m]$, and for $a\ge 1$ consider $A_a$ the $\mathbb K$-vector space of homogeneous polynomials of degree $a$ (the zero polynomial is assumed to have any degree).  Then $A_a$ has a basis given by all the monomials in the variables $y_1,\ldots,y_m$ of degree $a$. With this, consider the (well-defined) $\mathbb K$-linear map $\psi:A_a\rightarrow \mathbb K^s$, given by $\psi(f)=(f(P_1),\ldots,f(P_s))$. The image of $\psi$ is the projective Reed-Muller code of order $a$.

As before, suppose $a=1$. Then this is an $[s,m]$-linear code with a generator matrix $\displaystyle \left[\begin{array}{cccc}
    P_1&P_2&\cdots&P_s
    \end{array}\right]$; 
This is exactly the generator matrix of the simplex code from~\Cref{ex:HammingAndSimplex}, so order one projective Reed-Muller codes are precisely simplex codes, and their duals are the $\mathcal{H}_{q,m}$ Hamming codes.  Thus, if $M$ is the matroid of a first-order projective Reed-Muller code, the initial degree statistics of $I_{\Delta(M)}^{(s)}$ and $I_{\Delta(M^*)}^{(s)}$ follow from \Cref{ex:HammingAndSimplex}, which in turn follow from \Cref{ex:projectiveGeometries}.
\end{Example}

\begin{Remark}\label{rem:GHWProjReedMuller}
Even though the construction of projective Reed-Muller codes seems naturally related to that of affine Reed-Muller codes, a complete description of the generalized Hamming weights of projective Reed-Muller codes is an open problem.  See~\cite{RVK18, MVV20, SJ-2023} for recent work on these.  We ask in~\Cref{que: GHWReedMuller} whether the sequence of generalized Hamming weights is subadditive or cosubadditive for the higher order projective Reed-Muller codes.
\end{Remark}

\section{Matroid configurations}\label{sec:MatroidConfigurations}

The remarkable paper~\cite{GHMN17} shows that the Stanley-Reisner ideal of a matroid can be used in a natural way to define the so-called \textit{matroid configuration} of hypersurfaces in projective space.  Many properties of matroid configurations may be easily read off from the properties of the Stanley-Reisner ideal of the matroid.  The paper~\cite{GHMN17} builds on a fundamental fact about matroids: a simplicial complex $\Delta$ is the independence complex of a matroid if and only if all the symbolic powers of $I_{\Delta}$ are Cohen-Macaulay~\cite{MT-2011, Varbaro-2011}.  In this section, we use the tools we have developed so far and the machinery of~\cite{GHMN17} to compute the Waldschmidt constant for matroid configurations in terms of generalized Hamming weights and bound the resurgence and asymptotic resurgence for matroid configurations of points.  In particular, we compute the asymptotic resurgence of a matroid configuration of points coming from a perfect matroid design.

\subsection{Resolutions of the Stanley-Reisner ideals of Matroids} \label{sec: res of matroids}

Let $I_{\Delta(M)}\subset S:=\LL[x_1,\ldots,x_n]$ be the Stanley-Reisner ideal of a rank $k$ matroid $M$ in the polynomial ring $S$ over a field $\LL$.  Since $S/I_{\Delta(M)}$ is Cohen-Macaulay it has a minimal (multi-)graded free resolution $\mathbb{F}_{\bullet}\twoheadrightarrow S/I_{\Delta(M)}$ of length equal to ${\rm ht}(I_{\Delta(M)})=n-{\rm rk}(M)=n-k$:
\[
\mathbb{F}_\bullet :\ \ 0\rightarrow {\bf F}_{n-k} \rightarrow \cdots \rightarrow {\bf F}_1\rightarrow S,
\]
where, for each $r \in [n-k]$, ${\bf F}_r$ is a free $\NN^n$-graded $S$-module of finite rank, i.e. ${\bf F}_r=\oplus S(-\alpha)^{\beta_{r,\alpha}(M)}$, where $\alpha\in \NN^n$ and $\beta_{r,\alpha}(M)$ is the number of copies of $S(-\alpha)$ that appear in ${\bf F}_r$.  We refer to the index $r$ in ${\bf F}_r$ as the \textit{homological degree} and the number $\beta_{r,\alpha}(M)$ as the multigraded Betti number.  We define the \textit{coarsely graded} Betti number $\beta_{r,j}(M)$ as $\beta_{r,j}(M):=\sum\limits_{|\alpha|=j} \beta_{r,\alpha}(M)$.  Since $I_{\Delta(M)}$ is squarefree, the tuples $\alpha\in\NN^n$ appearing in the minimal free resolution of $I_{\Delta(M)}$ consist only of zeros and ones.  Thus, we abuse notation and, for a subset $U\subset [n]$, we let $\beta_{r,U}(M)$ be the multigraded Betti number $\beta_{r,\alpha_U}(M)$ where $\alpha_U\in\NN^n$ is the indicator vector of $U$; that is $(\alpha_U)_i=1$ if $i\in U$ and $(\alpha_U)_i=0$ if $i\notin U$.

In the next result, we connect the supports of squarefree monomials in $I_{\Delta(M)}^{(s)}$ with a description of the multigraded Betti numbers of $S/I_{\Delta(M)}$ due to Johnsen and Verdure~\cite{JV13}.

\begin{Proposition}\label{thm:SRResolution}
Let $M$ be a matroid on the ground set $E$ and let $U\subseteq E$.  The following are equivalent.
\begin{enumerate} 
\item \label{thm:SRResolutiona} $U$ is a circuit of $M^{(r)}$
\item \label{thm:SRResolutionb} $x^U$ is a support-minimal squarefree monomial in $I^{(r)}_{\Delta(M)}$
\item \label{thm:SRResolutionc} $E\setminus U$ is a flat of $M^*$ of rank $n-k-r$
\item \label{thm:SRResolutiond} The (multi-)graded Betti number $\beta_{r,U}(M)$ of $S/I_{\Delta(M)}$ is nonzero.
\end{enumerate}
In particular, $d_r(M)$ coincides both with the minimum of $|U|$ such that $\beta_{r,U}(M)\neq 0$ and with the minimum degree of a squarefree monomial in $I^{(r)}_{\Delta(M)}$.
\end{Proposition}
\begin{proof}
The statements \eqref{thm:SRResolutiona}, \eqref{thm:SRResolutionb}, and \eqref{thm:SRResolutionc} are equivalent by \Cref{thm:minDist} and \Cref{lem:ElongationBasesAndCircuits}.  We show that \eqref{thm:SRResolutiona} is equivalent to \eqref{thm:SRResolutiond}.  It follows from~\cite[Theorem~4.1]{JV13} that for a subset $U\subset E$ the multigraded Betti number $\beta_{r,U}(M)$ of $S/I_{\Delta(M)}$ is nonzero precisely when $U$ is minimal (under inclusion) for the property $|U|-\rk_M(U)=r$.  Thus $\beta_{r,U}(M)\neq 0$ precisely when $U$ is a circuit of $\cE^{r-1}(M)=M^{(r)}$.  The final statement is immediate.
\end{proof}

The {\em $($Castelnuovo-Mumford $)$ regularity} of $S/I_{\Delta(M)}$ is by definition
\[
{\rm reg}(S/I_{\Delta(M)}):=\max\{j-r\,:\, \beta_{r,j}(M)\neq 0\}.
\]
The regularity of $I_{\Delta(M)}$ is given by ${\rm reg}(I_{\Delta(M)})={\rm reg}(S/I_{\Delta(M)})+1$.  By \cite[Theorem~4.2, Corollary~4.3, and Remark~4.2]{JV13}, if $M$ is a matroid of rank $k$ on $n$ elements with no coloops, then  $S/I_{\Delta(M)}$ is level and \begin{equation*}\label{regularity}
    {\rm reg}(S/I_{\Delta(M)})=d_{n-k}(M)-n+k=k.
\end{equation*}

\begin{Remark}\label{rem:pure-reso}
The resolution for $S/I_{\Delta(M)}$ is \textit{pure} if there is exactly one nonzero graded Betti number for each homological degree.  By \Cref{thm:SRResolution}, this is equivalent to every flat of any fixed rank of $M^*$ having the same size.  That is the minimal free resolution of $S/I_{\Delta(M)}$ is pure if and only if $M^*$ is a perfect matroid design (see also~\cite{Armenoff-thesis-2015}).  It follows from our discussions in \Cref{gen-PMD} and \Cref{Ex:constant-weight} that if $M$ is a uniform matroid, the dual matroid of an affine or projective geometry (equivalently, the matroid of an order one affine or projective Reed-Muller code), or the matroid of a constant weight code, then $S/I_{\Delta(M^*)}$ has a pure resolution.  Betti numbers for these pure resolutions are computed when $M$ is the matroid of a constant weight code~\cite[Theorem~3.1]{JV14} and when $M$ is the matroid of a first-order affine Reed-Muller code~\cite{GS20}.
\end{Remark}

\subsection{Matroid configurations and their symbolic powers}  
In this subsection, we discuss bounds on the \textit{resurgence} and \textit{asymptotic resurgence} of certain specializations of the Stanley-Reisner ideal of a matroid, following ~\cite{GHMN17}.

These statistics were introduced by Bocci and Harbourne~\cite{BH10} and Guardo, Harbourne, and Van Tuyl~\cite{GHV13} in order to quantify the well-studied \textit{containment problem} in commutative algebra.  The containment problem, in general, is to characterize those pairs $(a,b)\in\NN^2$ for which $I^{(a)}\subset I^b$, where $I$ is an ideal in a commutative ring (see~\cite{Szemberg-Szpond-2017} for a survey of this problem focused on the zero-dimensional case).

The resurgence of an ideal $I$, denoted $\rho(I)$, in a polynomial ring was defined in~\cite{BH10} as follows:
\[
\rho(I)=\sup\left\lbrace\frac{s}{r}~:~ s, r\ge 1 \text{ and } I^{(s)}\not\subset I^r\right\rbrace.
\]
An asymptotic version of resurgence, called asymptotic resurgence and denoted $\widehat{\rho}(I)$, was defined in~\cite{GHV13} as follows:
\[
\widehat{\rho}(I)=\sup\left\lbrace \frac{s}{r}~:~  s, r\ge 1 \text{ and }  I^{(st)}\not\subset I^{rt} \mbox{ for all } t\gg 0\right\rbrace.
\]
 We fix the following Setting and Notation.
\begin{SettingNotation} \label{not: matroid config}
   Let $M$ be a matroid of rank $k$ on a ground set $E$ of size $n$ and $\Delta=\Delta(M)$ its independence complex. Let $S=\LL[x_1,\ldots,x_n]$ be a polynomial ring over a field $\LL$.  Fix an integer $N\ge n-k$, let $R=\LL[y_0,\ldots,y_N]$ be a polynomial ring over $\LL$, and let $f_1,\ldots,f_n$ be homogeneous polynomials in $R$ of degrees $\delta_1,\ldots,\delta_n$, respectively, so that any subset of at most $n-k+1$ of them forms a regular sequence in $R$.  Define the homomorphism of $\LL$-algebras $\phi:S\to R$ by $\phi(x_i)=f_i$.  If $J$ is an ideal of $S$ let $\phi_*(J)$ denote the ideal generated by $\phi(J)$ in $R$.  
\end{SettingNotation}

We now collect the following results from~\cite{GHMN17} -- these come from~\cite[Theorem~3.3, Theorem~3.6, Proposition~3.8, Corollary~4.3, and Corollary~4.6]{GHMN17}.

\begin{Theorem}\cite{GHMN17}\label{thm:MatroidConfigurations}
Adopt the \Cref{not: matroid config}. Then
\begin{enumerate} 
\item \label{thm:MatroidConfigurationsa} If $I_{\Delta}=\bigcap\limits_{B\in\cB(M^*)}P_B$, then $\phi_*(I_{\Delta})=\bigcap\limits_{B\in\cB(M^*)}\phi_*(P_B)=\bigcap\limits_{B\in\cB(M^*)}\langle f_i: i\in B\rangle$.
\item \label{thm:MatroidConfigurationsb} If $\mathbb{F}_{\bullet}$ is a minimal free resolution of $S/I_{\Delta}$ over $S$, then $\mathbb{F}_{\bullet}\otimes_{S}R$ is a minimal free resolution of $R/\phi_*(I_{\Delta})$ over $R$.
\item \label{thm:MatroidConfigurationsc} $\phi_*(I_\Delta^{(s)})=\phi_*(I_\Delta)^{(s)}$ for every $s\ge 1$.
\item \label{thm:MatroidConfigurationsd} If $I^{(s)}_{\Delta}\subset I^r_{\Delta}$, then $\phi_*(I_{\Delta})^{(s)}\subset \phi_*(I_{\Delta})^{r}$ for any nonnegative integers $r,s$.
\item \label{thm:MatroidConfigurationse} If $f_1,\ldots,f_n$ all have the same degree $\delta$, then $\widehat{\alpha}(\phi_*(I_\Delta))=\delta\widehat{\alpha}(I_{\Delta})$.
\item \label{thm:MatroidConfigurationsf} $\rho(\phi_*(I_{\Delta}))\le \rho(I_\Delta)$.
\item \label{thm:MatroidConfigurationsg} $\widehat{\rho}(\phi_*(I_{\Delta}))\le \widehat{\rho}(I_\Delta)$.
\end{enumerate}
\end{Theorem}

From \Cref{thm:MatroidConfigurations}~\eqref{thm:MatroidConfigurationsa}, we see that $\phi_*(I_{\Delta})$ is the defining ideal of a union of complete intersection subvarieties of $\mathbb{P}_{\LL}^N$.  It is, in general, a proper subvariety of the so-called \textit{hypersurface configuration} consisting of all codimension $n-k$ intersections among the hypersurfaces defined by the forms $f_1,\ldots,f_n$.  This is where the terminology \textit{matroid configuration} comes from.

Combining \Cref{thm:MatroidConfigurations} with previous results, we obtain the following.  Below, if $J$ is an ideal, $\omega(J)$ denotes the \textit{largest} degree of a minimal generator of $J$. 

\begin{Corollary} \label{cor: matroid config invariants}
Adopt the  \Cref{not: matroid config}.  Then,
\begin{enumerate} 
\item \label{cor: matroid config invariantsa} $\alpha(\phi_*(I_{\Delta}))=\min\left\{\sum\limits_{i\in C} \delta_i~:~ C\in\circuits(M)\right\}$.
\item \label{cor: matroid config invariantsb} $\omega(\phi_*(I_{\Delta}))=\max\left\{\sum\limits_{i\in C} \delta_i~:~ C\in\circuits(M)\right\}$.
\item \label{cor: matroid config invariantsc} $\widehat{\alpha}(\phi_*(I_{\Delta}))=\min\left\lbrace \dfrac{\sum\limits_{i\in U} \delta_i}{r}~:~ U\in\circuits(M^{(r)}),~r\in [n-k] \right\rbrace$.
\item \label{cor: matroid config invariantsd} $\mathrm{reg}(\phi_*(I_{\Delta}))=\sum\limits_{i\in E} \delta_i-(n-k)+1$.

\end{enumerate}
\end{Corollary}

\begin{proof}
Parts \eqref{cor: matroid config invariantsa} and \eqref{cor: matroid config invariantsb} are immediate from the definition of $\phi_*(I_{\Delta})$ and the fact that the minimal generators of $I_{\Delta}$ are the squarefree monomials coming from circuits of $M$. Now, it follows from \Cref{thm:MatroidConfigurations}~\eqref{thm:MatroidConfigurationsc} and \Cref{thm:GeneratorsOfSymbolicReesAlgebraOfMatroid} that the symbolic Rees algebra $\cR_s(\phi_*(I_{\Delta}))$ is generated as an $R$-algebra by $$\left\{\left(\prod_{j\in U} f_j\right)T^r~:~ U\in \circuits(M^{(r)}),~ r\in [n-k]\right\}.$$   Thus \eqref{cor: matroid config invariantsc} follows from \Cref{thm:ReesAlgebraFiltration}.

By~\cite[Theorem~3.3]{GHMN17}, $\phi_*(I_{\Delta})$ is Cohen-Macaulay, so $\mathrm{reg}(R/\phi_*(I_{\Delta}))$ is determined by the shift in the largest homological degree of its minimal free resolution.  Thus (d) follows from \Cref{thm:SRResolution} and \Cref{thm:MatroidConfigurations}~\eqref{thm:MatroidConfigurationsb}.  
\end{proof}

In general, if $I$ is a radical ideal in a polynomial ring, then the resurgence and asymptotic resurgence are bounded as follows:
\begin{align}\label{res-bounds}
1\le \frac{\alpha(I)}{\widehat{\alpha}(I)}\le \widehat{\rho}(I)\le \rho(I)\le \mathrm{ht}(I).  
\end{align}
The lower bounds on $\widehat{\rho}(I)$ and $\rho(I)$ in~\Cref{res-bounds} are shown in \cite{BH10,GHM13} while the upper bound follows from the seminal containment result of~\cite{ELS01,HH02}.  In case $M$ is a matroid of rank $k$ on a ground set of size $n$, $\mathrm{ht}(I_{\Delta})=n-k$.  We get the following lower bound on the asymptotic resurgence when all the homogeneous polynomials have the same degree.
\begin{Corollary} \label{cor: matroid config resurgence}
Adopt the \Cref{not: matroid config}. Then \begin{enumerate} 
    \item $\widehat{\rho}(\phi_*(I_\Delta)) \le \rho(\phi_*(I_\Delta)) \le \rho(I_\Delta)<n-k$.

    \item  If $f_1,\ldots,f_n$ all have the same degree, then \[
\dfrac{(n-k)d_1(M)}{n}\le\max\left\lbrace \dfrac{rd_1(M)}{d_r(M)}~:~r\in [n-k]\right\rbrace\le \widehat{\rho}(\phi_*(I_\Delta)).
      \]
\end{enumerate}
\end{Corollary}
\begin{proof}
Part $(\rm a)$ follows from  \Cref{thm:MatroidConfigurations}~\eqref{thm:MatroidConfigurationsf} and \cite[Corollary~4.20]{DD21} and part $(\rm b)$ follows from \Cref{thm:MatroidConfigurations}~\eqref{thm:MatroidConfigurationse} and \Cref{res-bounds}.
\end{proof}

If $N=n-k$ in \Cref{thm:MatroidConfigurations}, then $\phi_*(I_{\Delta})$ is the defining ideal of a set of points, which we call a {\em matroid configuration of points}.  For a matroid configuration of points, we have the following upper bounds on resurgence and asymptotic resurgence coming from~\cite{BH10} and~\cite{GHM13}:
\begin{align}
\label{ineq:rhoboundpoints}
\rho(\phi_*(I_{\Delta}))\le \dfrac{\mathrm{reg}(\phi_*(I_\Delta))}{\widehat{\alpha}(\phi_*(I_\Delta))}.\\
\label{ineq:rhohatboundpoints}
\widehat{\rho}(\phi_*(I_{\Delta}))\le \dfrac{\omega(\phi_*(I_\Delta))}{\widehat{\alpha}(\phi_*(I_\Delta))}.
\end{align} 
In the bounds on resurgence and asymptotic resurgence that follow, we assume the forms $f_1,\ldots, f_n$ in \Cref{not: matroid config} all have the same degree; using \Cref{cor: matroid config invariants} it is straightforward to modify the bounds in the case of arbitrary degrees.

\begin{Corollary} \label{cor:matroid config points resurgence}
Adopt \Cref{not: matroid config} with $N=n-k$ so that $\phi_*(I_{\Delta})$ is the defining ideal of a matroid configuration of points.  If $f_1,\ldots,f_n$ all have the same degree $\delta$, then
\[
\rho(\phi_*(I_{\Delta}))\le \max\left\lbrace \frac{r}{d_r(M)}\left( n -\frac{n-k-1}{\delta}\right)~:~ r\in [n-k]\right\rbrace
\]
and
\[
\widehat{\rho}(\phi_*(I_{\Delta}))\le \max\left\lbrace \frac{r |U|}{d_r(M)}~:~ U\in\circuits(M),~r\in [n-k] \right\rbrace.
\]
\end{Corollary}
\begin{proof}
The bounds follow from \eqref{ineq:rhoboundpoints} and~\eqref{ineq:rhohatboundpoints}, using the expressions for $\alpha(\phi_*(I_\Delta))$, $\omega(\phi_*(I_\Delta))$, $\widehat{\alpha}(\phi_*(I_\Delta)),$ and $\mathrm{reg}(\phi_*(I_\Delta))$ from \Cref{cor: matroid config invariants}.
\end{proof}

In what follows, a matroid $M$ of rank $k$ is a \textit{matroid design} if all its flats of rank $k-1$ have the same cardinality (see~\cite{YE73, YME70}).  Since the circuits of $M^*$ are the complements of the rank $k-1$ flats of $M$ (see \Cref{sec:prelims}), the circuits of the matroid dual to a matroid design all have the same size. 

\begin{Corollary}\label{cor:resurgenceboundssparsepaving}
Adopt \Cref{not: matroid config} with $N=n-k$ so that $\phi_*(I_{\Delta})$ defines a matroid configuration of points.  Suppose the homogeneous forms $f_1,\ldots,f_n$ all have the same degree $\delta$.  We have the following:
\begin{enumerate}
\item \label{cor:resurgenceboundssparsepavinga} If the generalized Hamming weights of $M$ form a \rsubad{} sequence, then
\[
\frac{(n-k)d_1(M)}{n}\le \widehat{\rho}(\phi_*(I_{\Delta})) \le \rho(\phi_*(I_{\Delta}))\le \frac{n-k}{n}\left(n-\frac{n-k-1}{\delta}\right)
.\]
\item \label{cor:resurgenceboundssparsepavingb} If the generalized Hamming weights of $M$ form a \rsubad{} sequence and $M^*$ is a matroid design, then 
$
\widehat{\rho}(\phi_*(I_\Delta))=\frac{(n-k)d_1(M)}{n}.
$
\item \label{cor:resurgenceboundssparsepavingc} If $M^*$ is a perfect matroid design, then
$
\widehat{\rho}(\phi_*(I_\Delta))=\frac{(n-k)d_1(M)}{n}.
$
\end{enumerate}
\end{Corollary}
\begin{proof}
For \eqref{cor:resurgenceboundssparsepavinga}, the first two inequalities follow from \Cref{cor: matroid config resurgence}.  For the final inequality in \eqref{cor:resurgenceboundssparsepavinga}, we have $\mathrm{reg}(\phi_*(I_\Delta))=n\delta-(n-k)+1$ by \Cref{cor: matroid config invariants}\eqref{cor: matroid config invariantsa}.  Furthermore, 
$
\widehat{\alpha}(\phi_*(I_\Delta))=\delta \widehat{\alpha}(I_\Delta)=\delta \dfrac{n}{n-k},
$ 
where the first equality follows from \Cref{thm:MatroidConfigurations}\eqref{thm:MatroidConfigurationse} and the second follows from \Cref{thm:SR}\eqref{thm:SRe}.  Applying \Cref{ineq:rhoboundpoints} now yields the final inequality in \eqref{cor:resurgenceboundssparsepavinga}.

For \eqref{cor:resurgenceboundssparsepavingb}, we again have $
\widehat{\alpha}(\phi_*(I_\Delta))=\delta \dfrac{n}{n-k}$.  Moreover $\omega(\phi_*(I_\Delta))=\delta \omega(I_{\Delta})=\delta \alpha(I_{\Delta})$, since $M^*$ is a matroid design and so all generators of $I_{\Delta}$ have the same degree.  Thus, $$\frac{\alpha(\phi_*(I_\Delta))}{\widehat{\alpha}(\phi_*(I_\Delta))}=\frac{(n-k)d_1(M)}{n}=\frac{\omega(\phi_*(I_\Delta))}{\widehat{\alpha}(\phi_*(I_\Delta))},$$ and so $\widehat{\rho}(\phi_*(I_{\Delta}))=\dfrac{(n-k)d_1(M)}{n}$ by \Cref{res-bounds} and \Cref{ineq:rhohatboundpoints}.

For \eqref{cor:resurgenceboundssparsepavingc}, if $M^*$ is a perfect matroid design, then the generalized Hamming weights of $M^*$ form a \rsubad{} sequence by \Cref{thm:PMD-indegree}~\eqref{thm:PMD-indegreea}.  Thus, the generalized Hamming weights of $M$ are also \rsubad{}.  Since $M^*$ is a perfect matroid design, it is also a matroid design.  Now \eqref{cor:resurgenceboundssparsepavingc} follows from \eqref{cor:resurgenceboundssparsepavingb}.
\end{proof}

If $\mathcal{S}=(E,\mathfrak{B})$ is a Steiner system of type $S(t,k,n)$, then the matroid $M(\mathcal{S})$ is a sparse paving matroid, as we proved in \Cref{prop:steinerSystemsAreSparsePaving}.  Moreover, we observe in \Cref{ex:SteinerSystemPMD} that $M(\mathcal{S})$ is a matroid design if and only if $\mathcal{S}$ is a Steiner system of type $S(k-1,k,n)$.  Thus, \Cref{cor:resurgenceboundssparsepaving} recovers ~\cite[Corollary~4.8]{BFGM21}.

We have seen that projective geometries, affine geometries, the perfect matroid design coming from a Steiner system, and the matroids of constant weight codes are all perfect matroid designs (\Cref{ex:projectiveGeometries}, \Cref{ex:affineGeometries}, \Cref{ex:SteinerSystemPMD}, and \Cref{Ex:constant-weight}).  Thus, the resurgence of a matroid configuration of points coming from the dual of any of these is determined by \Cref{cor:resurgenceboundssparsepaving}.

\section{Concluding remarks and questions}\label{sec: conclusion}

In this section, we collect several questions that arose while writing this paper, which we hope will spark some interest for readers.   Let $M$ be a matroid of rank $k$ on a ground set of size $n$ and let $\{d_i(M)\}_{i=1}^{n-k}$ denote its sequence of generalized Hamming weights. First note that if  $\{d_i(M)\}_{i=1}^{n-k}$ is subadditive but not cosubadditive, then the formula for $\alpha(I^{(s)})$ for $s\gg 0$ as in \Cref{thm:SR}~\eqref{thm:SRe} does not hold. One might then wonder if the subadditivity alone might imply the cosubadditive property. This does not hold in general, as seen in the next example.

\begin{Example} \label{ex: paving not cosub}
    Let $\KK$ be a field, and let $\C$ be the linear code whose generator matrix $G$ and parity check matrix $H$ are
    \[
    G=\begin{bmatrix}
    1 & 0 & 0 &-1&0\\
    0 & 1 & 0 & -1 & 0\\
    0 & 0 & 1 & -1 & -1
    \end{bmatrix}
    \quad \mbox{and} \quad
    H=\begin{bmatrix}
    1 & 1 & 1 & 1 & 0\\
    0 & 0 &1 & 0 & 1 \\
    \end{bmatrix}.
        \]
       Then $\mathcal{C}$ is an $[5,3]$-linear code. Let $M=M(\mathcal{C})$ and $M^*=M(\C^{\perp})$. Observe  that the circuits of $M^*$ are $\{1,2\},\{1,4\},\{2,4\},\{1,3,5\},\{2,3,5\},\{3,4,5\}$.  Therefore, the Stanley-Reisner ideal  of $\Delta(M^*)$ is  $$I_{\Delta(M^*)}=\left\langle x_1x_2,x_2x_4,x_1x_4,x_1x_3x_5,x_2x_3x_5,x_3x_4x_5 \right\rangle.$$ 
    
    Notice that $M^*$ is a rank two matroid, and every circuit of $M^*$ has a size of at least two. Therefore, $M^*$ is a paving matroid and $d_1(\C)=2$.  Since the maximal number of columns of $G$ that form a one-dimensional subspace of $\KK^3$ is two, by \Cref{lem:genMindistMatroid}, we have $d_2(\C)=3$. Also, as every column of the generator matrix $G$ is a nonzero column, we have $d_3(\C)=5$.  By inspection (or \Cref{thm:Paving-indegree}~\eqref{thm:Paving-indegreea}) $\{d_i(\C)\}_{i=1}^{3}$ is a subadditive sequence.  However, it is not cosubadditive, as $d_1(\C)+d_3(\C)>d_2(\C)+d_2(\C)$. Alternatively, since $d_1(\C)=2<\frac{5-2+2}{2}$, the sequence of generalized Hamming weights is not cosubadditive, see \Cref{thm:Paving-indegree}~\eqref{thm:Paving-indegreeb}.
\end{Example}
 
In \Cref{thm:SR} we showed that when $\{d_i(M)\}_{i=1}^{n-k}$ is \rsubad{}, then $\frac{1}{\widehat{\alpha}(I_{\Delta})}+\frac{1}{\widehat{\alpha}(I_{\Delta(M^*)})}=1.$  In \Cref{sec:pavingMatroids} and \Cref{sec: gen-CT}, we saw that sparse paving matroids and matroids of various types of linear codes do have subadditive and cosubadditive sequences of generalized Hamming weights, and hence satisfy the above equality on the Waldschmidt constants of the matroid and its dual.  It is then natural to ask the following. 
\begin{Question}\label{ques:reciprocals}
For what classes of matroids $M$ is it true that
$
\dfrac{1}{\widehat{\alpha}(I_{\Delta(M)})}+\dfrac{1}{\widehat{\alpha}(I_{\Delta(M^*)})}=1?
$ Moreover, is it possible that a matroid $M$ satisfies the above equality while its sequence of generalized Hamming weights is not necessarily both subadditive and cosubadditive?
\end{Question}

In \Cref{sec: gen-CT} we also discussed the affine and projective Reed-Muller codes. For the respective first-order Reed-Muller codes, we showed that their sequences of generalized Hamming weights are \rsubad{}. For higher-order (affine or projective) Reed-Muller codes, it is not known whether this holds. Recall that as discussed in \Cref{rem:GHWAffineReedMuller} formulas for the generalized Hamming weights of the higher order affine Reed-Muller codes are known, whereas the situation is even less clear for the higher projective Reed-Muller codes, see \Cref{rem:GHWProjReedMuller}. 

\begin{Question}\label{que: GHWReedMuller}
Do the generalized Hamming weights of higher order affine or projective Reed-Muller codes form a subaddtive and/or cosubadditive sequence?
\end{Question}

In the context of matroid configurations (see \Cref{not: matroid config}), if the specialization of $I_{\Delta}$ is given by forms $f_1,\ldots,f_n$ all of the same degree $\delta$, then \Cref{thm:MatroidConfigurations}~\eqref{thm:MatroidConfigurationse} yields $\widehat{\alpha}(\phi_*(I_{\Delta}))=\delta \widehat{\alpha}(I_{\Delta})$.  If $f_1,\ldots,f_n$ do not all have the same degree, then we can still use \Cref{cor: matroid config invariants}~\eqref{cor: matroid config invariantsc} to compute $\widehat{\alpha}(\phi_*(I_{\Delta}))$.  However, it is not so clear when the minimization procedure of \Cref{cor: matroid config invariants}~\eqref{cor: matroid config invariantsc} is feasible to carry out explicitly.

\begin{Question}
Can we compute (explicitly) the Waldschmidt constant of $\phi_*(I_{\Delta(M)})$ when the polynomials $f_1,\ldots, f_n$ in \Cref{not: matroid config} are not of the same degree?  By \Cref{cor: matroid config invariants}, this amounts to weighting the elements of the ground set of $M$ by the degrees of $f_1,\ldots,f_n$ and computing the minimum of weighted circuit sums for the elongations $M^{(r)}$, $r \in [n-k]$, normalized by $r$.
\end{Question}

The lower bounds on the resurgence and asymptotic resurgence of matroid configurations in \Cref{sec:MatroidConfigurations} also hold (with appropriate modifications depending on the degree of the forms used to specialize) for the Stanley-Reisner ideal of the matroid (before specializing).  However, this is not at all the case for the upper bounds since these are only known to hold if the ideal defines a zero-dimensional scheme (for resurgence~\cite{BH10}) or smooth scheme (for asymptotic resurgence~\cite{GHV13}).  Thus, we ask the following.

\begin{Question}
Do any of the upper bounds for matroid configurations of points in \Cref{sec:MatroidConfigurations} also hold for the asymptotic resurgence or resurgence of the Stanley-Reisner ideal of the matroid before specializing?
\end{Question}

We close with the following open-ended question.

\begin{Question}
If we construct a matroid $M$ from one or more other matroids $M_1,\ldots,M_k$ via standard matroid operations (e.g., contraction, deletion, matroid sum, etc.), can we conclude properties of the generalized Hamming weights of $M$ (e.g., subadditivity/cosubadditivity) from properties of the generalized Hamming weights of $M_1,\ldots,M_k$?
\end{Question}

\bibliography{bibl}
\bibliographystyle{plain}

\appendix

\section{Proof of Theorem~\ref{thm:GeneratorsOfSymbolicReesAlgebraOfMatroid}}\label{app}

We will use the same notation that we have used throughout the paper and the following variant of the basis exchange property for matroids due to Brualdi.

\begin{Proposition}[Bijective basis exchange property~\cite{Brualdi69}]\label{prop:bijectivebasisexchange}
If $M$ is a matroid and $B_1,B_2\in\cB(M)$, then there is a bijection $f:B_1\to B_2$ so that $(B_1\setminus\{e\})\cup\{f(e)\}\in \cB(M)$ for all $e\in B_1$.
\end{Proposition}

\begin{Remark}
Let $f:B_1\to B_2$ be the bijection in \Cref{prop:bijectivebasisexchange}.  We claim that if $e\in B_1\cap B_2$, then $f(e)=e$. To see this, first observe that if $e\in B_1\setminus (B_1\cap B_2)$ then $f(e)\notin B_1\cap B_2$, for if $f(e)\in B_1\cap B_2$ then $|(B_1\setminus\{e\})\cup \{f(e)\}|<|B_1|$, contradicting that $(B_1\setminus\{e\})\cup \{f(e)\}\in \cB(M)$.  Since $f$ is a bijection, it follows that $f(B_1\cap B_2)=B_1\cap B_2$.  Now suppose that $e\in B_1\cap B_2$ and $f(e)=e',$ where $e'\in B_1\cap B_2$ and $e'\neq e$.  Then again, $|(B_1\setminus\{e\})\cup\{e'\}|<|B_1|$ (since $e'\in B_1\setminus\{e\}$), contradicting that $(B_1\setminus\{e\})\cup \{f(e)\}\in \cB(M)$.  So $f(e)=e$.
\end{Remark}

\begin{Proposition}\label{prop:FactoringSymbolicMonomials}
Let $M=(E,\cB)$ be a matroid with $|E|=n$ and $\Delta=\Delta(M)$. 
Let $m=x^{\bf a}\in I^{(s)}_{\Delta}$ be a minimal generator with ${\bf a} =(a_1,\ldots,a_{n})\in \NN^{n}$ and let $A=\max\{a_i:i\in E\}$. Set $U=\{i\in E:a_i=A\}$ and let $t=\min\{|B\cap U|:B\in \cB(M^*)\}$.
We have the following:
\begin{enumerate} 
\item\label{i2:ineq1}  Let $B_1,B_2\in \cB(M^*)$ and $f:B_1\to B_2$ the bijection guaranteed by \Cref{prop:bijectivebasisexchange}.  If $\sum\limits_{i\in B_1} a_i=s$, then $a_i\le a_{f(i)}$ for all $i\in B_1$, and $|B_1 \cap U | \le |B_2 \cap U|$.
\item\label{i3:t} For any $B\in\cB(M^*)$, if $\sum\limits_{i\in B}a_i=s$, then $|B\cap U|=t$.
\item\label{i4:ineq2} If $B\in\cB(M^*)$, then $\sum\limits_{i\in B} a_i\ge s+|B\cap U|-t$ and $t\ge 1$.
\item\label{i5:factor} Factor $m$ as $m=x^Um'$.  Then $m'\in I^{(s-t)}_{\Delta}$ and $x^U\in I^{(t)}_{\Delta}$.

\item\label{i6:fullfactor} There exist  $n_1,\ldots,n_A$ positive integers with $\sum _{i=1}^A n_i=s$ and minimal generators $m_i\in I_{\Delta^{(n_i)}}$ such that $m=m_1\cdots m_A$.  
\end{enumerate}
\end{Proposition}
\begin{proof}
It follows from the description of symbolic powers in \Cref{eq:PDsymbolicpowersSR} that  $x^{\bf a} \in I_{\Delta}^{(s)}$ if and only if $ \sum\limits_{i \in B}a_i \ge s$ for all $B \in \cB(M^*)$. Also, $x^{\bf a} \in I_{\Delta}^{(s)}$ is a minimal monomial generator if and only if $ \sum\limits_{i \in B}a_i \ge s$ for all $B \in \cB(M^*)$ and there exists $B' \in \cB(M^*)$ such that $\sum\limits_{i \in B'}a_i = s$.

For~\eqref{i2:ineq1}, suppose $\sum\limits_{i\in B_1}a_i=s$ and $f:B_1\to B_2$ is the bijection guaranteed by \Cref{prop:bijectivebasisexchange}.  Let $i\in B_1$ and put $B'=(B_1\setminus\{i\})\cup\{f(i)\}$.  Then $s\le \sum\limits_{j\in B'}a_j=a_{f(i)}-a_i+\sum\limits_{j\in B_1}a_j=a_{f(i)}-a_i+s$.  It follows that $a_i\le a_{f(i)}$. 

Next, suppose  $a_i=A$.  Then, $A=a_i\le a_{f(i)}\le A$, so $a_{f(i)}=A$.  It follows that $f(B_1\cap U)\subseteq B_2\cap U$, so $|B_1\cap U|\le |B_2\cap U|$.

For~\eqref{i3:t}, suppose $B\in\cB(M^*)$ and $\sum\limits_{i\in B}a_i=s$.  Observe by the definition of $t$ that $|B\cap U|\ge t$.  Pick $B'\in \cB(M^*)$ so that $|B'\cap U|=t$ (the minimum intersection possible with $U$).  Let $f:B\to B'$ be the bijection guaranteed by \Cref{prop:bijectivebasisexchange}.  By~\eqref{i2:ineq1}, $|B\cap U|\le |B'\cap U|=t$.  So $|B\cap U|=t$.

For~\eqref{i4:ineq2}, let $B\in\cB(M^*)$ and pick $B'\in\cB(M^*)$ so that $\sum\limits_{i\in B'} a_i=s$.  By~\eqref{i3:t} we know that $|B'\cap U|=t$.   Now let $f:B'\to B$ be the bijection from \Cref{prop:bijectivebasisexchange}.  By~\eqref{i2:ineq1}, we know that, for any $i\in B'$, $a_i\le a_{f(i)}$, so in particular $f(B'\cap U)\subseteq  B\cap U$.  

Set $\ell:=|B\cap U|-|B'\cap U|=|B\cap U|-t$.  If $\ell=0$, since $x^{\bf a}\in I^{(s)}_{\Delta}$, we know that  $\sum\limits_{i\in B} a_i\ge s=s+\ell=s+|B\cap U|-t$ .  So suppose $\ell\ge 1$.  Let $(B \cap U)\setminus f(B'\cap U)=\{n_1,\ldots,n_\ell\}$.  Since $f$ is a bijection there are indices $m_1,\ldots,m_\ell\in B'\setminus U$ so that $f(m_i)=n_i$ for $1 \le i \le \ell$.  Then, $a_{m_i}+1\le a_{n_i}=A$, since $m_i\notin U$ and $n_i\in U$, for $i=1,\ldots,\ell$.  It follows from the previous line and~\eqref{i2:ineq1} that 
\[
\sum\limits_{j\in B}a_j= \sum\limits_{i\in B'}a_{f(i)}=\sum\limits_{i=1}^{\ell}a_{f(m_i)}+\sum\limits_{j\in B'\setminus \{m_1,\ldots,m_\ell\}} a_{f(j)} \ge \ell+\sum\limits_{j\in B'}a_j=|B\cap U|-t+s,
\]
proving the first part of \eqref{i4:ineq2}.

Now we show that  $t\ge 1$. 
 Suppose $t=0$. 
 Write $x^{\bf a} =x^{U}m'$ with $m'=x^{\bf a'}$. 
 Then, for any basis $B\in \cB(M^*)$, $\sum\limits_{j\in B} a_j'=\sum\limits_{j\in B}a_j-|B\cap U|\ge s-t=s$, by the first part of~\eqref{i4:ineq2}. Hence $m'\in I^{(s)}_{\Delta}$, contradicting the minimality of $m$.

For~\eqref{i5:factor}, it follows from the definition of $t$ that $|B \cap U | \ge t$ for all $B \in \cB(M^*)$.  Since we know from part \eqref{i4:ineq2} that $t\ge 1$ it follows that $U$ is a dependent set of $M^{(t)}$, see \Cref{lem:ElongationBasesAndCircuits}.  Therefore, by \Cref{lem:squarefreeMonomialsInSymbolicPowers},  $x^U\in I^{(t)}_{\Delta}$. 

We can write $m=x^{\bf a}\in I_{\Delta}^{(s)}$ as $m=x^Um'$, where 
\[
m'=\prod_{i\in E} x_i^{b_i} \quad\mbox{and} \quad b_i=
\begin{cases}
a_i & i\notin U\\
a_i-1 & i\in U.
\end{cases}
\]

We claim $m'\in I^{(s-t)}_{\Delta}$.  Let $B\in\cB(M^*)$. By the definition of $t$, we know $|B \cap U| \ge t$. Suppose $|B\cap U|=t+\ell$ for some $\ell\ge 0$.  By~\eqref{i4:ineq2}, $\sum\limits_{i\in B} a_i\ge s+\ell$, so $$\sum\limits_{i\in B}b_i=\sum\limits_{i\in B} a_i-|B\cap U|=\sum\limits_{i\in B}a_i-(t+\ell)\ge (s+\ell)-(t+\ell)=s-t$$  Since $B$ was arbitrary, $m'\in I^{(s-t)}_{\Delta}$.

For~\eqref{i6:fullfactor} we induct on $A\ge 1$.  If $A=1$, then $m$ is a squarefree minimal generator $I^{(s)}_{\Delta}$, thus $m\in I_{\Delta^{(s)}}$.  If it is not a minimal generator of $I_{\Delta^{(s)}}$ then $m/x_k\in I_{\Delta^{(s)}}$ for some $k\in E$ and so $m/x_k\in I^{(s)}_{\Delta}$, contradicting that $m$ is a minimal generator of $I^{(s)}_{\Delta}$.

Now suppose $A>1$.  By~\eqref{i5:factor} we can factor $m$ in the form $m=x^Um'$, where $x^U\in I_{\Delta}^{(t)}$,  $m'\in I^{(s-t)}_{\Delta}$ and $t \ge 1$.  Suppose $x^U\in I_{\Delta}^{(t)}$ is not a minimal generator.  Then there is some $k\in U$ so that $x^U/x_k\in I_{\Delta}^{(t)}$ and so $m/x_k=(x^U/x_k)m'\in I^{(s)}_{\Delta}$, contradicting that $m$ is a minimal generator of $I^{(s)}_{\Delta}$.  So $x^U$ is a minimal generator of $I^{(t)}_{\Delta}$ and, by the same argument as above, also a minimal generator of $I_{\Delta^{(t)}}$.  Similarly, $m'$ must be a minimal generator of $I^{(s-t)}_{\Delta}$.

Since $m'$ is a minimal generator of $I^{(s-t)}_{\Delta}$ and the largest exponent appearing in $m'$ is one less than the largest exponent appearing in $m$, by induction there exists non-negative integers $n_1,\ldots,n_{A-1}$ with $\sum\limits_{i=1}^{A-1} n_i =s-t$ such that $m'=m_1\cdots m_{A-1}$, where $m_i \in I_{\Delta^{(n_i)}}$ for $i=1,\ldots, A-1$.  Setting $x^U=m_A$ and $n_A=t$, we have the required factorization of $m$, namely $m=m'x^U=m_1\cdots m_A$.
\end{proof}

We now give the proof of \Cref{thm:GeneratorsOfSymbolicReesAlgebraOfMatroid}.  Recall that the \textit{corank} of a flat $F\in\cL(M)$ is $\crk_{M}(F):=\rk(M)-\rk_M(F)$.

\begin{customthm}{\Cref{thm:GeneratorsOfSymbolicReesAlgebraOfMatroid}}
Let $M$ and $\Delta$  be as in Setting and \Cref{set: Notation for section 7}.
Then, the symbolic Rees algebra $\cR_s(I_{\Delta})$ is generated as an $S$-algebra by the set of monomials 
\[
\{x^CT^i~:~C\in \circuits(M^{(i)}),1\le i\le n-k\}.
\]
Equivalently, $\cR_s(I_{\Delta})$ is generated as an $S$-algebra by 
\[
\{x^{E\setminus F}T^{\crk_{M^*}(F)}~:~F\in \cL(M^*) \setminus \{E\}\}.
\]
\end{customthm}
\begin{proof}
Let $m \in I_{\Delta}^{(s)}$ be a monomial and set $|E|=n$. There exists a minimal monomial generator $x^{\bf a} \in I_{\Delta}^{(s)}$ with ${\bf a} =(a_1,\ldots,a_{n}) \in \NN^{n}$ and a monomial $m'$ such that $m=m'x^{\bf a}$. Then, by \Cref{prop:FactoringSymbolicMonomials}, $x^{\bf a}=m_1\cdots m_A$, where $m_i$ is a minimal generator of $I_{\Delta^{(n_i)}}$ and $\sum\limits_{i=1}^{A} n_i=s$.  This yields the factorization
$
mT^s=m'\prod_{i=1}^A m_iT^{n_i}
$
in $\cR_s(I_{\Delta})$. 
Since $m_i$ is a minimal generator of $I_{\Delta^{(n_i)}}$, $m_i=x^{C_i}$ for some circuit $C_i\in M^{(n_i)}$ by \Cref{lem:squarefreeMonomialsInSymbolicPowers}.  Thus, $
mT^s=m'\prod_{i=1}^A x^{C_i}T^{n_i}.
$ which shows that the set $\{x^Ct^i~:~C\in \circuits(M^{(i)}), 1\le i\le n-k\}$ generates $\cR_s(I_{\Delta})$ as an $S$-algebra.
The final statement follows from \Cref{lem:ElongationBasesAndCircuits}.
\end{proof}

\end{document}